\documentclass[10pt]{amsart}
\usepackage{times}
\usepackage{amsfonts,amssymb,amsmath,amsgen,amsthm,faktor}
\usepackage{hyperref}
\usepackage{color}
\usepackage[T1]{fontenc}
\usepackage[utf8]{inputenc}
\usepackage{cancel,graphicx}
\usepackage{dsfont}
\usepackage{multirow,wasysym}
\usepackage[a4paper]{geometry}
\newcommand{\msc}[2][2000]{%
  \let\@oldtitle\@title%
  \gdef\@title{\@oldtitle\footnotetext{#1 \emph{Mathematics subject
        classification.} #2}}% 
}

\theoremstyle{plain}
\newtheorem{theorem}{Theorem}[section]
\newtheorem{definition}[theorem]{Definition}

\newtheorem{lemma}[theorem]{Lemma}
\newtheorem{corollary}[theorem]{Corollary}
\newtheorem{proposition}[theorem]{Proposition}

\theoremstyle{remark}
\newtheorem{remark}[theorem]{Remark}

\def\C{{\mathbb C}}% complex numbers
\def\R{{\mathbb R}}% real numbers
\def\N{{\mathbb N}}% nonnegative integers
\def\Z{{\mathbb Z}}% integers
\def\T{{\mathbb T}}% torus
\def\Sch{{\mathcal S}}% Schwartz space
\def\cA{\mathcal{A}}
\def\cC{\mathcal{C}}
\def\cF{\mathcal{F}}
\def\cL{\mathcal{L}}

\def\({\left(}
\def\){\right)}
\def\<{\left\langle}
\def\>{\right\rangle}

\def\Tend#1#2{\mathop{\longrightarrow}\limits_{#1\rightarrow#2}}

\def\eps{\varepsilon}

\def\op{{\rm op}}
\def\exp{\mathrm{e}^}

\DeclareMathOperator{\Hess}{Hess}
\DeclareMathOperator{\Tr}{Tr}
\DeclareMathOperator{\Sp}{Sp}
\DeclareMathOperator{\Id}{Id}
\DeclareMathOperator{\supp}{supp}
\DeclareMathOperator{\rk}{rk}

\newcommand\restr[2]{{% we make the whole thing an ordinary symbol
  \left.\kern-\nulldelimiterspace % automatically resize the bar with \right
  #1 % the function
  \vphantom{\big|} % pretend it's a little taller at normal size
  \right|_{#2} % this is the delimiter
  }}

\numberwithin{equation}{section}

\title[]{Effective mass theorems with Bloch modes crossings}
\author[V. Chabu]{Victor Chabu}
\address{Institute of Physics, University of São Paulo, DFMA, CP 66.318 05314-970, São Paulo, SP, Brazil}
\email{vbchabu@if.usp.br}
\author[C. Fermanian]{Clotilde~Fermanian Kammerer}
\address{LAMA, UMR CNRS 8050,
Universit\'e Paris EST\\
61, avenue du G\'en\'eral de Gaulle\\
94010 Cr\'eteil Cedex\\ France}
\email{Clotilde.Fermanian@u-pec.fr}
\author[F. Maci\`{a}]{Fabricio Maci\`{a}}
\address{M$^2$ASAI, Universidad Polit\'{e}cnica de Madrid. ETSI Navales. Avda. de la Memoria 4, 28040 Madrid, Spain}
\email{fabricio.macia@upm.es}

\makeindex

\begin{document}

\begin{abstract} We study a Schr\"odinger equation modeling the dynamics of an electron in a crystal in the asymptotic regime of small wave-length comparable to the characteristic scale of the crystal. Using Floquet Bloch decomposition, we obtain a description of the limit of time averaged energy densities. We make rather  general assumption assuming that the initial data are uniformly bounded in a high order Sobolev spaces and 
 that the crossings between Bloch modes are at worst conical. We show that despite the singularity they create, conical crossing do not trap the energy  and do not prevent dispersion. We also investigate the interactions between modes that can occurred when there are some degenerate crossings between Bloch bands.

\end{abstract}

\maketitle

\tableofcontents

\section{Introduction} 

\subsection{Description of the problem}
%\subsubsection{The Schr\"odinger equation}
We consider the dynamics of an electron in a crystal in the regime of small wave-length comparable to the characteristic scale of the crystal. After a suitable rescaling (see for instance~\cite{PR96}), such an analysis leads to an $\eps$-dependent Schr\"odinger equation  where $\eps$ is a small parameter $\eps\ll 1$
\begin{equation}\label{eq:schro}
\left\{
\begin{array}{l}
i\partial_t \psi^\eps (t,x)+\dfrac{1}{2} \Delta_x \psi^\eps(t,x) - \dfrac{1}{\eps^{2}} V_{\rm per}\left(\dfrac{x}{\eps}\right)\psi^\eps(t,x) - V_{\rm ext}(t,x)\psi^\eps(t,x) =0,\vspace{0.2cm} \\
\psi^\eps |_{t=0}=\psi^\eps_0.
\end{array}\right.
\end{equation}
The potential~$V_{\rm per}$ is supposed to be smooth, real-valued and  $\Z^d$-periodic; it models the interactions due to the crystalline structure. The external potential $V_{\rm ext}$, takes into accounts the impurities; we assume that $t\mapsto V_{\rm ext}(t,\cdot)$ is a bounded map from~$\R$ into the set of smooth, real-valued functions on $\R^d _x$ with bounded derivatives.
The times-scales of the equation~\ref{eq:schro} are characteristic of the analysis of the obstructions to the dispersion of the energy. It is the long time scaling studied in~\cite{BLP78, PR96, AP05, AP06, Spar06, HW11, BBA11}, by contrast to the short time analysis  that allows to analyze transport effects and is performed  for example in~\cite{BMP01,HST:01,PST:03,CS12,Watson} (the Schr\"odinger equation therein is obtained from~\eqref{eq:schro} by changing $t$ into $\tau=\eps t$).

\subsubsection{The wave function, observables and quadratic quantities}
We are interested in the asymptotic behavior of the time-averaged position densities $|\psi^\eps(t,x)|^2$ as~$\eps$ goes to~$0$. In other words, we would like to characterize the limit $\eps\to 0^+$ of the quantities
\begin{equation}\label{eq:energydensity}
\int_a^b\int_{\R^d} \phi(x)|\psi^\eps(t,x)|^2 dx\,dt,\;\; \phi\in{\mathcal C}_0(\R^d),\;\; a<b,
\end{equation}
where $\mathcal{C}_0(\R^d)$ stands for the space of continuous compactly supported functions on $\R^d$.

We will derive representations of these limits in terms of {\it Effective mass equations}. In its full generality, our result also describes the evolution of the action of observables on the wave functions. Indeed, the wave function itself has no physical meaning and it is the evolution of quadratic quantities such as those of Section~\ref{sec:semiclas} that carries information like the average momenta or the average energy. In~\eqref{eq:energydensity}, the averaging in time takes into account the fact that a physical observation is not instantaneous and, though small, its duration is not negligible. However, we will discuss situations where local in time point-wise information can be derived about the evolution of the energy densities (see Section~\ref{subsec:single}).
\smallskip

\subsubsection{Floquet-Bloch theory}
It is classical in this context to use Floquet-Bloch theory in order to diagonalize $-\frac{1}{2}\Delta_x+V_{\rm per}$. To this aim, one  introduces, for $\xi\in\R^d$, the operator
$$P(\xi):= \frac{1}{2} |\xi+D_y|^2 + V_{\rm per}(y),\;\; y\in\T^d,$$
where $\T^d=\R^d\backslash\Z^d$ is a flat torus. It is well known that this operator is essentially self-adjoint on $L^2(\T^d)$ with domain $H^2(\T^d)$, and  
has a compact resolvent, hence a non-decreasing sequence of eigenvalues counted with their multiplicities, which are called \textit{Bloch energies} or \textit{band functions}
$$ \varrho_1(\xi)\leq \varrho_2(\xi)\leq \cdots\leq \varrho_n(\xi)\longrightarrow +\infty,$$
and an orthonormal basis of eigenfunctions $\left(\varphi_n(\cdot,\xi)\right)_{n\in\N^*}$ called \textit{Bloch waves} or \textit{Bloch modes}, satisfying for all~$\xi\in\R^d$ and $n\in\N^*$:
\begin{equation}\label{eq:underline}
P(\xi)\varphi_n(\cdot,\xi)=\varrho_n(\xi)\varphi_n(\cdot,\xi).
\end{equation}
Both Bloch waves and Bloch energies are continuous functions of the $\xi$-variable. Besides,
for all $k\in2\pi\Z^d$, the operator $P(\xi+k)$ is unitarily equivalent to $P(\xi)$ through  multiplication by  $y\mapsto {\rm e}^{ik\cdot y}$, which implies that for all~$n\in \N^*$,  the maps $\xi\mapsto \varrho_n(\xi) $ are $2\pi\Z^d$-periodic and the map $\xi \mapsto \varphi_n (\cdot,\xi)$ belongs to $\mathcal C(\R^d_\xi, L^2(\T^d_y))$.
The spectrum of $-\frac{1}{2}\Delta_x+V_{\rm per}$ is then the union of the  \textit{Bloch bands} $B_n:=\varrho_n([0,2\pi]^d)$, which are closed intervals:
\[
\Sp\left(-\frac{1}{2}\Delta_x+V_{\rm per}\right)=\bigcup_{n\in\N^*}B_n.
\]
The Bloch energies $\xi\mapsto \varrho_n(\xi)$ are Lipschitz functions which are analytic outside a  set of zero Lebesgue measure (see~\cite{Wilcox78}). In particular Bloch energies that are of constant multiplicity as $\xi$ varies are always analytic functions of $\xi$. These energies are then called \textit{isolated}. The opposite situation, that is, when two, otherwise distinct, Bloch energies coincide at some point $\xi$ is referred to as a \textit{crossing}. At those points, the multiplicity is greater than one and the corresponding Bloch bands have non-empty intersection. When the space dimension is one, two Bloch bands can touch at one edge and their crossing set consists on isolated points (see Appendix~\ref{sec:1d} and the references therein); in higher dimensions more complicated situations can occur: most bands overlap (in fact as soon as $d\geq 2$ only a finite number of gaps exist) and the crossing set may be a higher dimensional manifold (in fact, the union of the graphs of the band functions form a real analytic variety). The survey article \cite{Kuch16} provides additional details on these issues.

\subsubsection{Effective mass theory} Sometimes also called \textit{effective Hamiltonian theory}, effective mass theory consists in showing that, under suitable assumptions on the initial data~$\psi^\eps_0$, the  energy density associated with the solutions of~\eqref{eq:schro} can be approximated for~$\eps$ small by those of a simpler Schrödinger equation, the \textit{Effective mass equation}, which does not depend on $\eps$ and involves quantities related to the Bloch energies. 

Effective mass equations have then been derived in various contributions~\cite{BLP78, PR96, AP05, AP06, Spar06, HW11, BBA11} under the assumptions that the orthogonal projection of the initial datum $\psi^\eps_0$ on spectral subspaces corresponding to the non-simple Bloch energies is negligible, and that critical points of this band functions are non-degenerate.

All these contributions emphasize the important role played by the set of critical points of the Bloch energies. Indeed, the group velocity of the $n$-th mode is $\eps^{-1} \nabla \varrho_n(\eps\xi)$, which becomes infinite in the limit~$\eps \rightarrow 0$; this implies that the obstructions to the dispersive effects created by the $n$-th band, $n\in\N^*$,  have to be found above the set $\Lambda_n$ of critical points  of the function $\varrho_n$:
\begin{equation}\label{def:Lambdan}
\Lambda_n:=\{\xi\in\R^d\,:\,\nabla\varrho_n(\xi)=0\}.
\end{equation}

A second  feature that is assumed in the aforementioned references is the simplicity of the band functions, which is an important technical ingredient in the proofs. Simple band functions are smooth, and therefore group velocity is well defined everywhere. This property may fail in the presence of bands crossings which  creates at worst  loss of regularity at the crossing points. 
In that case, the group velocity $\eps^{-1}\nabla_\xi \varrho_n (\eps\xi)$ is no longer defined at the crossing points, even though it may have directional limits, the archetype being the conical singularity $\varrho_n (\xi)\sim \xi/|\xi| $ close to $\xi=0$. Our aim here is to deal with this difficulty. As a consequence, this motivates the introduction of 
the crossing set of two distinct Bloch energies:
\begin{equation}\label{def:Sigmann'}
\Sigma_{n,n'}:=\{\xi\in\R^d\,:\,\varrho_n(\xi)=\varrho_{n'}(\xi)\},\quad  n,n'\in\N^*, \;\varrho_n\neq \varrho_{n'}.
\end{equation} 
The band functions $\varrho_n$, $n\in\N^*$, are piece-wise real analytic; their non-smoothness points lie in the union of crossing sets $\bigcup_{\varrho_n\neq \varrho_{n'}}\Sigma_{n,n'}$.
We will also consider the sets 
\begin{equation}\label{def:Sigman}
\Sigma_n:=\Sigma_{n,n+1},\quad n\in\N^*. 
\end{equation}
The crossing problematic has been addressed since long for equations that are scaled differently in the small parameter, in particular by George Hagedorn in the 90s~\cite{Hag94}. Since then, different approaches have been devoted to understand propagation through crossings, from the use of normal forms~\cite{CdV1,CdV2}, the analysis in terms of Wigner measures~\cite{FG02,FG03,F1} and Wigner functions~\cite{FL08,FL17,FM}, up to, more recently, the analysis in terms of wave packets~\cite{Watson} in the case of non-singular crossings. Indeed, the growing interest in crossings, especially conical ones, is linked with the technological interest of new materials that are topological insulators (see~\cite{Drouot,DrouotWeinstein} and references therein).
However, this question has never been addressed in the context of the particular scaling in~$\eps$ of equation~\eqref{eq:schro}.

\smallskip

In~\cite{CFM1,CFM2}, the range of validity of the Effective Mass Theory has been extended to include degenerate critical points,  through the introduction of a new class of Effective mass equation which are of von Neumann type. However, the Bloch modes involved in the description of the initial data are still assumed to be of constant multiplicity.  
Our aim here is to consider situations where 
different Bloch energies may have non-empty intersections inducing singularities and to treat rather general initial data.
Our result gives a 
 a complete description of the weak limits of the densities $|\psi^\eps(t,x)|^2 $ as~$\eps$ goes to~$0$  when the crossings are conical, in a sense that we shall make precise later.  This is done through the analysis of the  weak limits of the Wigner function associated with $\psi^\eps(t,x)$. Indeed, the Wigner function introduced in Section~\ref{sec:basic} below, plays the role of a generalized energy-density in the phase space~$T^*\R^d= \R^d_x\times \R^d_\xi$, the density  $|\psi^\eps(t,x)|^2 $ being its projection in the configuration space~$\R^d_x$. Our result covers all possible cases when $d=1$ and generic situations in higher dimensions.  
We also complete the description of the picture by providing a characterisation of these limits when crossings are degenerate, exhibiting the persistence of terms due to interactions between the Bloch energies that cross. 
 Our results rely on the use of a two-microlocal analysis in the spirit of~\cite{MaciaTorus,AM14,AFM15,ALM16,MacRiv18}, using two-scale Wigner distributions~\cite{Fermanian_note1,Fermanian_Note2,NierScat,MillerThesis}.

\subsection{General assumption on the initial data}
We denote by $A(\eps D_x)$, for $\eps>0$, the scaled Fourier multiplier associated with the function $A(\xi)$, \textit{i.e.} the operator satisfying
$$\forall f\in{\mathcal S}(\R^d),\quad\widehat{A(\eps D_x) f}(\xi)= A(\eps \xi) \widehat f(\xi),$$
where the following normalization has been used for the Fourier transform:
\[
\widehat f(\xi)=\int_{\R^d} \exp{-i\xi\cdot x}f(x)dx.
\]
Along the paper, we will consider the functions spaces $H^s_\eps(\R^d)$, defined for $s\geq 0$, that are  the Sobolev spaces equipped with the norms:
\[
\|f\|_{H^s_\eps(\R^d)}:=\|\left\langle \eps D_x\right\rangle^s f\|_{L^2(\R^d)},
\]
where $\left\langle \xi\right\rangle:=(1+|\xi|^2)^{1/2}$. 
\smallskip

Any function $U\in L^2(\R^d_x\times\T^d_y)$ can be written in terms of Fourier series as:
\[
U(x,y)=\sum_{k\in\Z^d}U_k(x){\rm e}^{i2\pi k\cdot y}\;\;{\rm 
with}\;\;
\|U\|_{L^2(\R^d\times\T^d)}^2=\sum_{k\in\Z^d}\|U_k\|_{L^2(\R^d)}^2.
\]
We denote by $H^s_\eps(\R^d\times\T^d)$, for $s\geq 0$, the Sobolev space consisting of those functions $U \in L^2(\R^d\times\T^d)$ such that there exists $\eps_0,C>0$ for which we have 
\begin{equation}\label{def:normHs}
\forall \eps\in(0,\eps_0),\;\;\|U\|_{H^s_\eps(\R^d\times\T^d)}^2:=\sum_{k\in\Z^d}\int_{\R^d}(1+|\eps\xi|^2+|k|^2)^s|\widehat{U_k}(\xi)|^2d\xi\leq C.
\end{equation}
These functions can be projected on the bands as follows. For every $n\in\N^*$ and $\xi\in\R^d$, we denote by~$\Pi_n(\xi)$ the projector from $L^2(\T^d)$ onto the eigenspace corresponding to $\varrho_n(\xi)$. The corresponding Fourier multiplier $\Pi_n(\eps D_x)$ acts on $L^2(\R_x^d\times \T_y^d)$, since band functions are bounded.
Finally, we
define the operator $L^\eps$ acting on functions $F\in H^s_\eps(\R^d\times\T^d)$, $s>d/2$, by 
$$(L^\eps F)(x):=F\left(x,\frac x\eps\right).$$ Then there exists $C_s>0$ such that, for every $F\in H^s_\eps(\R^d\times\T^d)$,
\begin{equation}\label{def:Leps}
\|L^\eps F\|_{L^2(\R^d)}\leq C_s \|F\|_{H^s_\eps(\R^d\times\T^d)},
\end{equation}
uniformly in $\eps>0$. See \cite[Lemma~6.2]{CFM2}.
\smallskip 

We will make the following assumption on the family of initial data in \eqref{eq:schro}.
\begin{itemize}
\item[\textbf{H0}] There exists a bounded family $(U^\eps_0)$ in $H^s_\eps(\R^d\times\T^d)$ for some $s>d/2$ such that:
\[
\psi^\eps_0=L^\eps U^\eps_0.
\]
\end{itemize}
Note that if  $(\psi^\eps_0)_{\eps>0}$ is  bounded in $H^s_\eps(\R^d)$ with $s>d/2$, {\bf H0} holds with~$U^\eps_0(y,x)= \psi^\eps_0(x)\otimes {\bf 1}_{\T^d}(y)$.

\subsection{The case of the dimension one}

Let us first state our results in  dimension $d=1$.
When $d=1$, 
one can prove that~$\Lambda_n$ is contained in $\pi\Z$, and consists only on non-degenerate critical points. In addition, when $|n-n'|>1$, $\Sigma _{n,n'}= \emptyset$ and $\Sigma_n\cap \Lambda_n=\emptyset$
 %the crossing set corresponding to $\varrho_n$ is precisely $\pi\Z\setminus \Lambda_n$
  (see  Lemma~\ref{lem:critb}).  
  In this specific case, 
we are able to
 give a complete description of the limit of the energy density of families of solutions to~(\ref{eq:schro}) with initial data of the form stated in {\bf H0}.

\begin{theorem}\label{theorem:1d}
Assume $(\psi^\eps_0 )$ satisfies {\bf H0}. Then there exists  a subsequence $(\psi^{\eps_\ell}_0)_{\ell\in\N}$ of the initial data,
such that for every $a<b$ and every $\phi\in \mathcal{C}_0(\R)$ the following holds:
$$\displaylines{\qquad
\lim_{\ell\to\infty}\int_a^b\int_{\R}\phi(x)|\psi^{\eps_\ell}(t,x)|^2dxdt=
\sum_{n\in\N^*}\;
\sum_{\xi\in \Lambda_n}\int_a^b\int_{\R}\phi(x)|\psi_{\xi}^{(n)} (t,x)|^2dxdt
}$$
where, for every $n\in\N^*$ and $\xi\in \Lambda_n$,  $\psi_\xi ^{(n)}$ solves the effective mass Schrödinger equation:
\begin{equation}\label{eq:schrohprofil}
i\partial_t \psi_{\xi}^{(n)}(t,x) ={1\over 2}\partial^2_\xi\varrho_n(\xi)\partial_x^2 \psi_{\xi}^{(n)}(t,x) +V_{\rm ext}(t,x)\psi_{\xi}^{(n)}(t,x),
\end{equation}
with initial datum:
\begin{center}
$\psi_{\xi}^{(n)}|_{t=0}$ is the weak limit in $L^2(\R)$ of the sequence $\left({\rm e}^{-\frac{i}{\eps_\ell} \xi  x} L^{\eps_\ell} \Pi_n(\eps_\ell D_x)U^{\eps_\ell}_0\right)$. 
\end{center}
\end{theorem}

Note some of  the accumulation points of ${\rm e}^{-\frac{i}{\eps_\ell} \xi  x} L^{\eps_\ell} \Pi_n(\eps_\ell D_x)U^{\eps_\ell}_0$  may just be~$0$. For example if one has~$V_{\rm per}=0$, only the first Bloch energy $\varrho_1$ has critical points and they are precisely $\Lambda_1=2\pi\Z$. Besides, the associated  projector $\Pi_1(\xi)$ coincides with the orthogonal projection onto $\C {\rm e}^{i ky}$ whenever $\xi\in (k-\pi,k+\pi)$ and $k\in 2\pi\Z$.  Therefore, if one takes $U^\eps_0(x,y)=\psi^{\eps}_0(x)\otimes {\bf 1}_{y\in\T^d}$, then 
\[
\Pi_1(\eps\xi)\widehat{U^{\eps_\ell}_0}(\xi,\cdot)=\mathbf{1}_{(-\pi,\pi)}(\eps\xi)\widehat{\psi^{\eps_\ell}_0}(\xi)
\]
and ${\rm e}^{-\frac{i}{\eps_\ell} 2\pi k  x} \Pi_1(\eps D_x)(\psi^{\eps_\ell}_0\otimes {\bf 1}_{y\in\T})$ weakly converges to zero when $k\neq0$. 
\smallskip

 Theorem~\ref{theorem:1d} is derived as a consequence of a more general analysis that is valid in any dimension under assumptions that are   satisfied for all Bloch energies when $d=1$, and that is presented in the next section.

\subsection{The generic case with $d\geq 1$ - conical crossings} 
We now present results that, under a set of assumptions that always hold when $d=1$, give a description of effective mass equations in higher dimension  under the presence of generic crossings for data satisfying~{\bf H0}.

\subsubsection{Assumptions}
Our first assumption concern the multiplicity of Bloch bands.
\begin{itemize}
\item[\textbf{H1}] For $n\in\N^*$, the multiplicity of the Bloch energy $\varrho_n$ is one,
except at crossing points, where it is two. This implies that a global labeling of the band functions exists such that $\Sigma_{n,n'}\neq\emptyset$ implies $|n-n'|=1$. 
\end{itemize}

\begin{remark}
 Hypothesis~\textbf{H1} is thought to be generic, as follows from the variational characterization of eigenvalues of Schrödinger operators with Bloch periodicity conditions. We make it in order to avoid  having statements that are unnecessarily involved.  As we stated it, it prevents from having simultaneous crossings of more than two Bloch energies, and higher multiplicities (both scenarios are  non-generic). The proofs we provide can be adapted in order to deal with these situations.
\end{remark}

We also consider a   generic assumption on the set of critical points~$\Lambda_n$ defined  in~\eqref{def:Lambdan}.

\begin{itemize}
\item[\textbf{H2}]For $n\in\N^*$, we assume that $\Hess \varrho_n$ is of constant rank in a neighborhood of each connected component of $\Lambda_n$. 
\end{itemize}

\begin{remark}
Let $X\subseteq \Lambda_n$ be a connected component of $ \Lambda_n$. By the constant rank level set theorem, 
this hypothesis implies that each connected component $X\subseteq \Lambda_n$ is a closed submanifold of $\R^d$
of dimension $d-\rk \Hess \varrho_n|_X$.
\end{remark}

Finally, our third set of hypothesis concerns the geometry of the crossing sets $\Sigma_n$.
For stating this assumption,
we introduce geometric objects associated with  a submanifold~$X$ of $(\R^d)^*$: we consider its tangent spaces $T_\xi X$ and 
 define the fibre of the normal bundle $NX$ of $X$ above $\xi\in X$ as the vector space~$N_\xi X$ consisting of those $\eta\in(\R^d)^{**}=\R^d$ that annihilate $T_\xi X$
\begin{equation}\label{def:NX}
NX:=\{(\xi,\eta )\in  X\times \R^d \, : \, \eta \cdot \zeta=0 ,\;\;\forall \zeta\in T_\xi X \}.
\end{equation}
With a Bloch mode $\varrho_n$  presenting crossings on a manifold $\Sigma_n$,  we associate the function~$g_n$ defined on~$N\Sigma_n$ by   
\begin{equation}\label{def:gn}
(\xi,\eta)\mapsto  g_n(\xi,\eta):=\frac12\left( \varrho_{n+1}(\xi+\eta) -\varrho_n(\xi+\eta)\right),\;\;\xi\in \Sigma_n,\;\;\eta\in N_\xi\Sigma_n.
\end{equation}
Note that $g_n(\xi,\eta)\geq 0$ and $g_n(\xi,\eta)=0$ if and only if $\eta=0$. Besides, for any $\xi\in\Sigma_n$, $\eta\mapsto g_n(\xi,\eta)$ is differentiable in all  $\eta\not=0$ (see Appendix~\ref{app:crossings}). We denote by $\nabla_\eta g_n(\xi,\eta)$ this differential, which can be identified with a vector of $N_\xi\Sigma_n\subset T_\xi (\R^d)^*$. 
\begin{definition}\label{def:con}
We say that the crossings of $\Sigma_n$ are conic if and only if there exists a neighborhood~$U$ of~$\Sigma_n$ such that~$\varrho_n$ 
and~$\varrho_{n+1}$ are of multiplicity~$1$ outside $\Sigma_n$ in $U$ and there exists $c>0$ such that 
\[\forall(\xi,\eta)\in N \Sigma_n,\;\;  |g_n(\xi,\eta)|\geq c  |\eta|.
\]
\end{definition}
Note that the critical sets $\Lambda_n$ contain no conical crossing point. Besides, one can prove (see Appendix~\ref{app:crossings}) that, generically, as soon as the crossing set $\Sigma_n$ is a closed submanifold of $\R^d$, either  $\varrho_n$ has a conical singularity along $\Sigma_n$, either $\varrho_n$ is in~$\mathcal C^{1,1}$. 
We set for  $n\in \N^*$:
\begin{equation}\label{def:lambda_g}
\lambda_n(\xi)=\frac 12 \left(\varrho_n(\xi)+\varrho_{n+1}(\xi)\right),\;\;\xi\in\R^d,\;\; n\in \N^*.
\end{equation}
and we introduce the following  last assumption

\begin{itemize}
\item[\textbf{H3}]For $n\in\N^*$, we assume that the crossing set $\Sigma_n$ is a smooth closed submanifold of $\R^d$. Moreover, the crossing is of conic type in the sense of Definition~\ref{def:con} and for all $\xi\in \Sigma_n $, $\eta\in N_\xi\Sigma_n$ with $\eta\not=0$, 
$$\nabla_\xi \lambda_{n} (\xi) \pm \nabla_{\eta} g_n(\xi,\eta) \not=0.$$
\end{itemize}

\subsubsection{The result: dispersion above conical crossings} 
For stating the result, we need to introduce other geometric objects associated with a 
 submanifold~$X$ of~$(\R^d)^*$.  We define its cotangent bundle  as the union of all cotangent spaces to $X$
\begin{equation}\label{def:T*X}
T^*X:=\{(\xi,x)\in X\times  \R^d \, : \, x \in T_\xi^* X \},
\end{equation}
each fibre $T_\xi^*X$ is the dual space of the tangent space $T_\xi X$. We shall denote by $\mathcal{M}_+(T^*X)$ the set of non-negative Radon measures on~$T^*X$. 
We observe that 
every point $x\in\R^d$ can be uniquely written as 
\[
x=v+z\;\; \mbox{ where} \;\; v\in T^*_\xi X\;\;\mbox{ and}\;\; z\in N_\xi X.
\]
Then,   given a function $\phi\in L^\infty(\R^d)$ and a point $(\xi,v)\in T^*X$, we denote by $m^X_\phi(\xi,v)$  the operator acting on $L^2(N_\xi X)$ by multiplication by $\phi(v+\cdot)$. 
We shall denote by $\mathcal L(L^2(N_\xi X))$ the set of bounded operators acting on $L^2(N_\xi X)$ and by $\mathcal L^1_+(L^2(N_\xi X))$ the set of operators  that are non-negative and trace-class.  
When $X=\Lambda_n$ and assumption {\bf H2} holds, we will consider the operator ${\rm Hess}\, \varrho_n(\xi)D_z\cdot D_z$ acting  on $N_\xi\Lambda_n$ for any $\xi\in\Lambda_n$.

\begin{theorem}\label{theorem0}
Assume {\bf H1}, {\bf H2} and {\bf H3}  are satisfied for all  $n\in\N^*$
 and consider $(\psi^\eps)_{\eps>0}$ a family of solutions to equation~\eqref{eq:schro} with an initial data $(\psi^\eps_0)_{\eps>0}$ that satisfies {\bf H0}. Then, there exist a subsequence $(\psi^{\eps_\ell}_0)_{\ell\in\N}$ of the initial data, a sequence of non negative measures $(\nu_n)_{n\in\N}$  on $T^*\Lambda_n$, and a sequence of measurable non negative
 trace-class operators $(M_n)_{n\in\N}$ 
$$M_n:T^*_\xi\Lambda_n\ni (\xi,v) \mapsto M_n(\xi,v)\in\mathcal L^1_+(L^2(N_\xi \Lambda_n)),\;\;{\rm Tr} _{L^2(N_\xi \Lambda_n)}M_n(\xi,v)  =1,$$
both depending only on  $(\psi^{\eps_\ell}_0)_{\ell\in\N}$,
such that for every $a<b$ and every $\phi\in{\mathcal C}_0(\R^d)$ one has
\begin{equation}\label{eq:theorem0}
\lim_{\ell\rightarrow +\infty} \int_a^b\int_{\R^d} \phi(x) |\psi^{\eps_\ell} (t,x)|^2 dx dt =
 \sum_{n\in\N} \int_a^b \int_{T^*\Lambda_n}{\rm Tr} _{L^2(N_\xi \Lambda_n)}
 \left(m^{\Lambda_n}_\phi(\xi,v)M^t_n(\xi,v)\right) 
 \nu_n(d\xi,dv)dt,
 \end{equation}
 where $t\mapsto M^t_n(\xi,v) \in\mathcal C(\R, \mathcal L^1_+(L^2(N_\xi \Lambda_n))$ solves the von Neumann equation
 \begin{equation}\label{eq:heis}
 \left\{\begin{array}l
 i\partial_t M^t_n (\xi,v) =\left[ \frac 12{\rm Hess} \varrho_n(\xi) D_z\cdot D_z+ m_{V_{\rm ext}}^{\Lambda_n} (\xi,v) \;,\; M^t_n(\xi,v) )\right] \\
 M_n^0= M_n.
 \end{array}
\right.
\end{equation}
(recall that $m^{\Lambda_n}_\phi(\xi,v)$ (resp. $m^{\Lambda_n}_{V_{\rm ext}}(\xi,v)$)  denotes the operator acting on $L^2(N_\xi \Lambda_n)$ by multiplication by $\phi(v+\cdot)$ (resp. $V_{\rm ext}(v+\cdot)$)). 
\end{theorem}

Above and throughout this article, when $A$ and $B$ are two operators acting on the same Hilbert space, the notation $[A,B]$ denotes the commutator $AB-BA$.
Several remarks are in order.
First, note  that  the $n$-th term of the sum in~\eqref{eq:theorem0}  measures how much the critical points of the $n$-th Bloch mode trap the energy and prevent the dispersion effects. Theorem~\ref{theorem0} also  tells that  conical crossings do not trap energy. This Theorem~\ref{theorem0} has exactly the same form than~Theorem~2.2 in~\cite{CFM2} while the assumptions are quite different since crossings between Bloch energies are authorized as long as they are conical. We shall see in the next subsection that crossing points may trap energy when they also are critical points of the Bloch energies (and thus they are no longer conical). 
\smallskip

Secondly, we emphasize that  Remark~\ref{rem:M0} (1)  comments the determination of $(M_n)_{n\in\N^*}$ and $(\nu_n)_{n\in\N^*}$ from the initial data. The special case where $\Lambda_n$ is a point is discussed in the next subsection. 
\smallskip

Thirdly, recall that  when $d=1$, the assumptions {\bf H1}, {\bf H2} and {\bf H3} are automatically satisfied (see Appendix~\ref{sec:1d}). Therefore, Theorem~\ref{theorem:1d} is a consequence of Theorem~\ref{theorem0} in the case where critical points are isolated.
\smallskip

Finally, we  emphasize that Theorem~\ref{theorem0} extends to situations where the Fourier transform of the  initial data is localized on a set of the form $\{\eps\xi\in \Omega+2\pi\Z^d\}$ for some open subset~$\Omega$ of a unit cell of $2\pi\Z^d$, provided the  assumptions {\bf H2} and {\bf H3}  are satisfied for all $n\in\N^*$ above points of $ \Omega+2\pi\Z^d$. We esquiss this approach in Section~\ref{sec:extension} and explain how the arguments of the proofs detailed below can be adapted to this setting by localisation (see Lemma~\ref{lem:Ueps}).

\subsubsection{The special case of isolated critical points} \label{subsec:single}
When {\bf H1}, {\bf H2} and {\bf H3} are satisfied for all $n\in\N^*$ and, moreover, all the sets $\Lambda_n$ consist in a family of isolated critical points then $T^*\Lambda_n=\Lambda_n\times \{0\}$, $N\Lambda_n=\R^d$ so that the operators $M_n^t$ only depend on the parameter $\xi\in\Lambda_n$ and the operator $m_\phi^{\Lambda_n}(\xi,v)$ simply is the operator of multiplication by $\phi$.
One can prove a statement very similar to Theorem~\ref{theorem:1d} (see also Corollary~1.3 in~\cite{CFM2} and Remark~\ref{rem:M0} (2)): the measures $\nu^t_n$ are linear combinations of Dirac masses at the points $\xi_n\in\Lambda_n$ and $M_n^t(\xi_n)$ are orthogonal projectors on $\C\psi^{\xi_n} (t)$, the solution to 
\begin{equation}\label{eq:tata}
i\partial_t \psi^{\xi_n}(t,x) = \frac 12 {\rm Hess} \varrho_n(\xi_n) D_x\cdot D_x\psi^{\xi_n}(t,x) + V_{\rm ext} \psi^{\xi_n}(t,x)
\end{equation}
with initial data  $\psi^{\xi_n}(0)$ that is a weak limit of $\left( {\rm e}^{-\frac i{\eps_\ell} x\cdot \xi_{n}} L^{\eps_\ell} \Pi_{n}(\eps_\ell D_x) U^{\eps_\ell}_0\right)_{\eps_\ell>0}$ in $L^2(\R^d)$.
  
In order to illustrate this type of result in the multidimensional setting, we state it in the particular case of  
well-prepared data that satisfy $U^\eps_0(y,x)=\varphi_{n_0}(y,\eps D_x) u_{n_0}^\eps(x)$ for some $n_0\in\N^*$ (and therefore $\Pi_n(\eps D_x) U^\eps_0=0$ for $n\not=n_0$). 
More specifically, the class of well-prepared initial data that we will consider is those of the form 
 \begin{equation}\label{def:wpdata}
    \psi^\eps_{0,n}(x):= \varphi_{n}\left(\frac x\eps,\eps D_x\right) u^\eps_{n}(x),\quad u^\eps_{n}(x)={\rm e}^{ \frac i\eps x\cdot \xi_{n}} v_{n}^\eps(x),\;\;n\in\N^*
 \end{equation}
with $\xi_{n}\in\Lambda_{n}$, $(v_{n}^\eps)_{\eps>0}$ bounded in $H^s(\R^d)$, $s>d/2$ or $s>1$ when $d=1$, and such that:
\[
v_{n}^\eps\rightharpoonup v_{n},\quad \eps\to 0^+,\quad \text{ in }L^2(\R^d).
\]
These data are closely related with those considered in~\cite{AP06,AP05} for example (this connection is explained in detail in Lemma~\ref{lem:APdata}).  Their main properties are studied in Appendix~\ref{app:wpdata}). 

\begin{proposition}\label{prop:special_case}
Assume {\bf H1}, {\bf H2} and {\bf H3}  are satisfied for all  $n\in\N^*$ and that, for some $n_0\in\N^*$, the set~$\Lambda_{n_0}$ is discrete. Consider $(\psi^\eps)_{\eps>0}$ a family of solutions to equation~\eqref{eq:schro} with  initial data~$\psi^\eps_0$ satisfying~\eqref{def:wpdata} for some~$\xi_{n_0}\in\Lambda_{n_0}$ and $v^\eps_{n_0}$, $v_{n_0}$ in $L^2(\R^d)$.
 Then, there exists a subsequence $(\psi^{\eps_\ell}_0)_{\ell\in\N}$ of the initial data  
 such that for every $a<b$ and every $\phi\in{\mathcal C}_0(\R^d)$ one has
\begin{equation*}
\lim_{\ell\rightarrow +\infty} \int_a^b\int_{\R^d} \phi(x) |\psi^{\eps_\ell} (t,x)|^2 dx dt =
\int_a^b  \int_{\R^d} \phi(x) |\psi^{\xi_{n_0}}(t,x)|^2 dt,
 \end{equation*}
 where $\psi^{\xi_{n_0}}(t)$ solves~\eqref{eq:tata} with initial data $\psi^{\xi_{n_0}}(0) = v_{n_0}$
\end{proposition}

If $v^\eps_0\to v_{n_0}$ as $\eps\to 0^+$ in $L^2(\R^d)$, then  one has for all $t\in\R$,  $\|\psi^\eps_0\|_{L^2}\to \| \psi^{\xi_{n_0}}(t)\|_{L^2}$,
 which implies that no energy is dispersed. As a consequence of Remark~\ref{rem:tata99}, one obtains that the result of Proposition~\ref{prop:special_case} holds locally in time: for all $T>0$, there exists a subsequence $(\psi^{\eps_\ell}_0)$
 of the initial data  
 such that for every $\phi\in{\mathcal C}_0(\R^d)$ one has for all $t\in [0,T]$
\begin{equation*}
%\label{eq:theorem0}
\lim_{\ell\rightarrow +\infty} \int_{\R^d} \phi(x) |\psi^{\eps_\ell} (t,x)|^2 dx dt =
\int_{\R^d} \phi(x) |\psi^{\xi_{n_0}}(t,x)|^2.
 \end{equation*}
Proposition~\ref{prop:special_case} is an improvement of  Theorem~3.2 in~\cite{AP05} since it holds without any simplicity assumption on the Bloch energy $\varrho_{n_0}$.

\subsection{A non-generic case with interactions of Bloch bands above non-conical crossing points} 
In that section, we consider  degenerate crossing points that we define as follows. 

\begin{definition}\label{def:degq}
We say that the crossing set $\Sigma_n$ is degenerate if some $q\geq 2$ exists such that the function $g_n$ defined in~\eqref{def:gn}, satisfies
$$\exists c>0,\;\; \forall(\xi,\eta)\in N\Sigma_n ,\;|\eta|<1,\;g_n(\xi, \eta) \leq c|\eta|  ^q.$$
If the above inequality fails for any $q'>q$, we say the crossing set has degeneracy of order $q$.
\end{definition}

\subsubsection{Assumptions} 
We consider two Bloch energies that cross in a degenerate manner, though
 isolated from the remainder of the spectrum.

\begin{itemize}
\item[\textbf{H1'}] $\varrho_n(\xi)$ and $\varrho_{n+1}(\xi)$ are two Bloch energies that cross on $\Sigma_n$ and are of multiplicity~$1$ outside $\Sigma_n$.  
\item[\textbf{H2'}] $\Hess \varrho_n$ (resp. $\Hess \varrho_{n+1}$) is of constant rank  in a neighborhood of each connected component of $\Lambda_n$ (resp. $\Lambda_{n+1}$).
\item[\textbf{H3'}]  The crossing set $\xi_n$ is a smooth closed submanifold of $\R^d$ included in $\Lambda_n \cap \Lambda_{n+1}$ and is degenerated of order~$q$. Besides, $\Sigma_{n-1}=\Sigma_{n+1}=\emptyset$ and 
\begin{itemize}
    \item  if $q> 2$,  for all $\xi\in \Sigma_n$, the Hessian of $\lambda_n $ is of rank $d-{\rm Rank} \Sigma_n$ in a neighborhood of $\Sigma_n$,
    \item  if $q=2$, for all $(\xi,\eta)\in N\Sigma_n$  with $|\eta|=1$,
$$ {\rm Hess}\, \lambda_n (\xi)\eta\pm  \nabla _\eta g_n(\xi,\eta)\not=0  .$$
\end{itemize}
\end{itemize}

Note that the latter assumption can be considered as a maximal rank assumption. Indeed, if
$q> 2$, the Bloch energies are ${\mathcal C}^{2}$ and 
${\rm Hess} \, \varrho_n(\xi)= {\rm Hess} \, \varrho_{n+1}(\xi)= {\rm Hess} \, \lambda_n(\xi)$
above points $\xi\in \Sigma_n$. Moreover, if $q=2$ and
$(\xi_k,\omega_k)_{k\in\N}$ is a sequence of points  satisfying for all $k\in\N$,  $\omega_k\in N_{\xi_k} \Sigma_n$ with $|\omega_k|=1$, we have the following property:  if $(\xi_k,\omega_k)\Tend{\ell}{+\infty}(\xi,\omega)\in N \Sigma_n$ then 
$${\rm Hess} \, \varrho_n (\xi_k) \omega_k \Tend{(\xi_k,\omega_k)}{(\xi,\omega)}  {\rm Hess}\, \lambda_n(\xi) \omega   \pm   \nabla _\eta g_n(\xi,\omega).$$
This property shows the link between   $ {\rm Hess}\, \lambda_n (\xi)\eta\pm  \nabla _\eta g_n(\xi,\eta)$ and ${\rm Hess} \, \varrho_n (\xi)\eta$.  
\smallskip

The assumption {\bf H3'} implies that $\nabla\lambda_n=0$ on $\Sigma_n$.  We shall see in the proof that if~$\nabla \lambda_n$ does not vanish on~$\Sigma_n$, then these degenerate crossing points do not contribute. The result then is comparable to the one of Theorem~\ref{theorem0}.  Notice that one always has
$$
\Lambda_n\cap \Lambda_{n+1} = \Sigma_n \cap \Lambda_n=  \Sigma_n\cap \Lambda_{n+1}= \Sigma_n \cap \{ \nabla \lambda_n(\xi)=0\}.$$

At the difference with the results of the preceding section that were obtained under the assumption that {\bf H1}, {\bf H2} and~{\bf H3} were satisfied for all $n\in\N^*$, we will assume in this section that we have {\bf H1'}, {\bf H2'} and~{\bf H3'} for some single $n\in\N^*$ and we will choose well-prepared data concentrating on the bands $\varrho_n$ and $\varrho_{n+1}$ involved in the assumptions. 
\smallskip

We assume that 
\begin{equation}\label{data:wp}
\psi^\eps_0(x)= \varphi_n\left(\frac x\eps,\eps D\right) u^\eps_{n}(x) + \varphi_{n+1}\left(\frac x\eps,\eps D\right) u^\eps_{n+1}(x)
\end{equation}
where $\varphi_n$ (resp. $\varphi_{n+1}$) is the Bloch wave associated with $\varrho_n$ (resp. $\varrho_{n+1}$), and $(u^\eps_{n})_{\eps>0}$ and $(u^\eps_{n+1})_{\eps>0}$ are bounded families in $H^s_\eps(\R^d)$, $s>d/2$.  
These data are somehow more general than those of~\eqref{def:wpdata} since no assumption is made on $u^\eps_{n}$ ans $u^\eps_{n+1}$.

\subsubsection{The result: concentration above degenerate crossings}
We prove that for data as in~\eqref{data:wp}, the way their components interact  above the crossing set plays a role in the determination of the weak limits of the time-averaged energy density. 
We associate with the function $g_n$ defined on $N\Sigma_n$ (see~\eqref{def:gn})  the operator $Q_{g_n}^{\Sigma_n} (\xi)$ acting on $L^2(N_\xi \Sigma_n)$ as a Fourier multiplier.

\begin{theorem}\label{theo:interaction}
Let~$(\psi^\eps)_{\eps>0}$ be a family of solutions to equation~\eqref{eq:schro} with initial data  satisfying~\eqref{data:wp}. 
Assume the Bloch energies $\varrho_n$ and $\varrho_{n+1}$ satisfies~{\bf H1'},~{\bf H2'} and~{\bf H3'} with
 $\Lambda_n=\Lambda_{n+1}=\Sigma_n$.
Then, there exists a subsequence $(\psi^{\eps_\ell}_0)_{\ell\in\N}$ of the initial data,
 a non negative measure $\nu^0\in\mathcal M^+(\Sigma_n)$  and a matrix~$M$ of measurable trace-class operators   
\begin{align*}
M:\,&T^*_\xi\Sigma_n\ni (\xi,v) \mapsto M(\xi,v) 
\in\mathcal L^1_+(L^2(N_\xi \Sigma_n,\C^2)),\;\;{\rm Tr} _{L^2(N_\xi \Sigma_n,\C^2)}M (\xi,v)  =1\; d\nu^0\; a.e.
\end{align*}
both depending only on   $(\psi^{\eps_\ell}_0)_{\ell\in\N}$, 
such that
for every $a<b$ and every $\phi\in{\mathcal C}_0(\R^d)$ one has
\begin{align*}
&\lim_{\ell\rightarrow +\infty} \int_a^b\int_{\R^d}  \phi(x) |\psi^{\eps_\ell} (t,x)|^2 dx dt\\
&\; =   \int_a^b \int_{T^*\Sigma_n}{\rm Tr} _{L^2(N_\xi \Sigma_n, \C)}[m^{\Sigma_n}_\phi(\xi,v)
(m_n^t + m_{n+1}^t +2 {\rm Re} (m_{n,n+1}^t ))(\xi,v) ] \nu^0(d\xi, dv)dt ,
 \end{align*}
 where 
  \begin{equation*}
M^t(\xi,v) = \begin{pmatrix} m_n^t(\xi,v) & m_{n,n+1}^t(\xi,v) \\ m_{n,n+1}^t(\xi,v)^* & m_{n+1}^t(\xi,v)\end{pmatrix}
\end{equation*}
is a non negative trace class operator on $L^2( N_\xi \Sigma_n,\C^2) $.
Besides, the map 
$$t\mapsto M^t (\xi,v)\in\mathcal C(\R, \mathcal L^1_+(L^2(N_\xi \Sigma_n, \C^2))$$
solves a von Neumann equation that depends on the value of q:
 \begin{itemize}
\item If $q>2$, it solves 
 \begin{equation}\label{ed:heis2bis}
 i\partial_t M^t (\xi,v)  =\left[ \left(\frac 12{\rm Hess} \lambda_n(\xi) D_z\cdot D_z  + m^{\Sigma_n}_{V_{\rm ext}}(\xi,v)\right){\rm Id_{\C^2}}  \;,\; M^t(\xi,v) \right],\;M^0=M. 
\end{equation}
\item If $q=2$, it solves 
\begin{equation}\label{ed:heis2}
i\partial_ t M^t (\xi,v) =
 \left[ 
 \left(\frac 12 {\rm Hess}\, \lambda_n(\xi) D_z\cdot D_z + m^{\Sigma_n}_{V_{\rm ext}(t,\cdot)}(\xi,v) \right) {\rm Id}_{\C^{2}}
  - Q^{\Sigma_n} _{g_n}(\xi)  \, J \;,\; M^t(\xi,v) 
  \right],\; M^0=M \\
\end{equation}
with  $J=\begin{pmatrix} 1 & 0 \\ 0 & -1 \end{pmatrix}$. 
\end{itemize}
(recall that $m^{\Sigma_n}_\phi(\xi,v)$ (resp. $m^{\Sigma_n}_{V_{\rm ext}}(\xi,v)$)  denotes the operator acting on $L^2(N_\xi \Sigma_n)$ by multiplication by $\phi(v+\cdot)$ (resp. $V_{\rm ext}(v+\cdot)$), and $Q^{\Sigma_n} _{g_n}(\xi)$ is the Fourier multiplier on $L^2(N_\xi \Sigma_n)$ associated with the function $g_n(\xi, \cdot)$). 
\end{theorem}

In the latter statement, we have assumed for simplicity that the critical sets of the two Bloch energies~$\varrho_n$ and~$\varrho_{n+1}$ coincide with the crossing set $\Sigma_n$.  We indeed prove this result without this assumption; however, the resulting statement is more involved (see Theorem~\ref{theo:general2cross'}).

\subsubsection{The special case of isolated degenerate crossing points} 
When $\Sigma_n$ consists in a family of isolated degenerate crossing points, the preceding statement admits a simpler, more transparent formulation. Indeed, as in Section~\ref{subsec:single},  the operator $M^t$ only depend on the parameter $\xi\in\Sigma_n$ and the operator $m_\phi^{\Lambda_n}(\xi,v)$ simply is the operator of multiplication by $\phi$ and $Q^{\Sigma_n}_{g_n}= g_n(\xi, D_x)$.
We now assume that $\psi^\eps_0$ satisfies~\eqref{data:wp} with 
\[ u^\eps_{n}(x)={\rm e}^{ \frac i\eps x\cdot \xi_{n}} v_{n}^\eps(x)\;\;\mbox{and}
\;\;u^\eps_{n+1}(x)={\rm e}^{ \frac i\eps x\cdot \xi_{n}} v_{n+1}^\eps(x)
\]
with $\xi_{n}\in\Sigma_{n}$, and for $j\in\{n,n+1\}$, $(v_{j}^\eps)_{\eps>0}$ bounded in $H^s(\R^d)$, $s>d/2$ or $s>1$ when $d=1$, and such that:
\[
v_{j}^\eps\rightharpoonup v_{j},\quad \eps\to 0^+,\quad \text{ in }L^2(\R^d).
\]

\begin{proposition}\label{prop:special_cas2}
Let~$(\psi^\eps)_{\eps>0}$ be a family of solutions to equation~\eqref{eq:schro} with initial data  satisfying~\eqref{data:wp}. 
Assume the Bloch energies $\varrho_n$ and $\varrho_{n+1}$ satisfy~{\bf H1'},~{\bf H2'} and~{\bf H3'}.
 Assume moreover $\Lambda_n=\Lambda_{n+1}=\Sigma_n$ consists in a family of isolated critical crossing points.
 Then, there exists a subsequence $(\psi^{\eps_\ell}_0)_{\ell\in\N}$ of the initial data  
 such that for every $a<b$ and every $\phi\in{\mathcal C}_0(\R^d)$ one has
\[
\lim_{\ell\rightarrow +\infty} \int_a^b\int_{\R^d} \phi(x) |\psi^{\eps_\ell} (t,x)|^2 dx dt =
\sum_{\xi\in \Sigma_n}
\int_a^b  \int_{\R^d} \phi(x) |\psi^{\xi_n}_n(t,x)+\psi^{\xi_n}_{n+1}(t,x)|^2 dx dt,
\]
 where
 \begin{itemize}
     \item if $q>2$,  for $j\in\{n,n+1\}$, $\psi^{\xi}_{j}$  solves the effective mass Schr\"odinger equation 
 \begin{equation}\label{eq:tata'}
i\partial_t \psi^{\xi_n}_j(t,x) = \frac 12 {\rm Hess} \lambda_n(\xi_n) D_x\cdot D_x\psi^{\xi_n}_j(t,x) + V_{\rm ext}(t,x) \psi^{\xi_n}_j(t,x)
\end{equation}
\item if $q=2$, for $j\in\{n,n+1\}$, $\psi^{\xi}_{j}$  solves the effective mass Schr\"odinger equation 
 \begin{equation}\label{eq:tata''}
i\partial_t \psi^{\xi_n}_j(t,x) = \frac 12 {\rm Hess} \lambda_n(\xi_n) D_x\cdot D_x\psi^{\xi_n}_j (t,x)+\eta_j g_n(\xi_n, D_x)\psi^{\xi_n}_j (t,x)+ V_{\rm ext} (t,x)\psi^{\xi_n}_j(t,x),
\end{equation}
with $\eta_n=1 \;\mbox{and}\; \eta_{n+1}=-1$. 
 \end{itemize}
Besides, for $j\in\{n,n+1\}$, the initial data  
 $\psi^{\xi_n}_{j}(0)=v_j$.
\end{proposition}

We stress the fact that, in contrast to what happened in the presence of a singular crossing, the description of the limiting position density limit involves a term of the form $2{\rm Re} (\psi^{\xi_n} \overline{\psi^{\xi_{n+1}}})$ which takes into account the coupling between the modes corresponding to the two Bloch bands.

\subsection{Ideas of the proofs and organisation of the paper}
We follow the semi-classical approach developed in~\cite{CFM1,CFM2} which is based on semi-classical analysis. In these references, the Bloch energies in consideration are smooth, and we have  exhibited the role of the critical points of the Bloch energies as principal contributors to the weak limits of the time-averaged energy densities. We have also explained how a second microlocalisation allows to compute quantitatively this contribution. We follow here this scheme of thoughts with additional 
  difficulties that are two-fold. 
\smallskip

 Firstly, in order to consider general initial data as in Theorem~\ref{theorem:1d} and~\ref{theorem0},  and to decompose them on the Bloch energies, we shall need to treat infinite series. The assumption that the data satisfy~{\bf H0} 
is the key point that we use technically for treating this issue. We explain in Section~\ref{sec:prooflemma} how we perform the decomposition and which properties of the solution we use. 
\smallskip

 The second difficulty comes from the lack of regularity of the Bloch energies close to the crossing sets, which requires to perform semi-classical calculus with symbols of low regularity, what we do by using and developing ideas from~\cite{FGL}. We explain and construct  in Sections~\ref{sec:basic} and~\ref{section:twomic} the semi-classical and two-microlocal analysis of our problem. 
\smallskip 

 This laid us to the statement of two theorems that are interesting by themselves: in  Theorems~\ref{theo:2mic1} and~\ref{theo:2mic2}, we describe the evolution of two-microlocal semi-classical measures associated to the concentration of the solutions of~\eqref{eq:schro} on one of the sets of critical points~$\Lambda_n$, $n\in\N^*_0$, and in the context given by hypothesis~\textbf{H1}, \textbf{H2} and~\textbf{H3} for the first one and~\textbf{H1'}, \textbf{H2'} and~\textbf{H3'} for the second one. These two theorems are proved in Sections~\ref{sec:proof1} and~\ref{sec:proof2} respectively; they are the core of the proofs of Theorems~\ref{theorem0} and~\ref{theo:interaction}, which are performed themselves in Section~\ref{sec:proofs}, together with the proof of Theorem~\ref{theorem:1d}, Propositions~\ref{prop:special_case} and~\ref{prop:special_cas2}. 
\smallskip 

Finally,  some appendices are  devoted to technical elements that we use in the proofs of this paper:  special features of the Bloch decomposition in dimension~1 (Appendix~\ref{sec:1d}), properties of the Bloch energies at a crossing (Appendix~\ref{app:crossings}), elements of matrix-valued pseudo-differential calculus, in particular with low regularity (Appendix~\ref{app:pseudo}), two scale pseudodifferential calculus (Appendix~\ref{app:twomic}) and various remarks above well-prepared data (Appendix~\ref{app:wpdata}). 
\bigskip

\noindent \textbf{Acknowledgement. }VC was supported by
grant 2017/13865-0, São Paulo Research Foundation (FAPESP). FM is supported by grant MTM2017-85934-C3-3-P, MINECO, Spain. CFK and FM are grateful to the Technische Universit\"at M\"unchen and especially Professors Caroline Lasser and Simone Warzel for their hospitality during the redaction of this article and to the National Science Foundation under Grant No. DMS-1440140 while they were in residence at the Mathematical Sciences Research Institute in Berkeley, California, during the fall 2019 semester.  Finally, the authors thank the anonymous referees for their careful reading and their remarks that have greatly contributed to ameliorate our manuscript.

%%%%%%%%%%%%%%%%%%%%%%%%%%%%%%%%%%%%%%%%%%%%%%%%%%%%%%%%%%%%%%%%%%%

\section{Separation of scales and control of the oscillations}\label{sec:prooflemma}

Here we present the first steps of the strategy that will lead to the proof of Theorems~\ref{theorem:1d},~\ref{theorem0} and~\ref{theo:interaction}. 
 Our starting point is the following \textit{Ansatz} that is widely used in this context and consists in separating the slow and fast scales of oscillation. We look for a solution to \eqref{eq:schro} of the form:
\begin{equation}\label{def:U}
\psi^\eps(t,x)=U^\eps\left(t,x,{x\over \eps}\right),\;\;(t,x)\in\R\times\R^d,
\end{equation}
where $U^\eps(t,\cdot,\cdot)$ is a function on $\R^d\times\T^d$. In order to make sense of this, some regularity on the solutions $\psi^\eps$ is required, and this partly justifies our assumption {\bf{ H0}}.  Uniqueness for solutions to the initial value problem for the Schr\"odinger equation~(\ref{eq:schro}) with initial data satisfying {\bf H0} implies that~\eqref{def:U} holds provided $U^\eps$ is a solution to the system 
\begin{equation}\label{eq:U}
\left\{\begin{array}{l}
i\eps^2\partial_t U^\eps (t,x,y)=P(\eps D_x)U^\eps (t,x,y)+\eps^2 V_{\rm ext}(t,x) U^\eps(t,x,y), \vspace{0.2cm}\\
U^\eps_{|t=0}= U^\eps_0,\end{array}\right.
\end{equation}
By Lemma~6.5 in~\cite{CFM2}, this equation is well-posed  in $H^s_\eps(\T^d\times \R^d)$:
there exists $C_s>0$ such that for every $t\in\R$ and $U^\eps_0\in H^s_\eps(\R^d\times \T^d)$,
\begin{equation}\label{eq:wpxy}
\|U^\eps(t,\cdot)\|_{H^s_\eps(\R^d\times\T^d)}\leq \|U^\eps_0\|_{H^s_\eps(\R^d\times\T^d)}+C_s\eps|t|,
\end{equation}
uniformly in  $\eps >0$.
\smallskip

We shall use a decomposition of 
 the solution $U^\eps$ in the basis of  Bloch modes. 
 We set, for $j\in\N_0$,
 \begin{equation}\label{def:Pphi}
P^\eps_{\varphi_j} W(x,y):=\varphi_j\left(y,\eps D_x\right)\int_{\T^d}\overline{\varphi_j}(z,\eps D_x) W(x,z)dz,\;\;\forall W\in L^2(\T^d\times\R^d).
\end{equation}

\begin{lemma}\label{lem:proj}
If $(\psi^\eps_0)_{\eps>0}$ satisfies {\bf H0}, then for all $t\in \R$, the solution of~\eqref{eq:schro} is given by
\begin{equation}\label{eq:sum}
\psi^\eps(t,\cdot)=\sum_{n\in\N^*} \psi^\eps_n(t,\cdot),
\end{equation}
where the convergence of the series takes place in $L^2(\R^d)$ and
\begin{equation}\label{eq:psiepsn}
\psi^\eps_n(t,x):=L^\eps P^\eps_{\varphi_n} U^\eps(t,x)= \varphi_n\left(\frac {x}\eps , \eps D_x\right) \int_{\T^d}\overline{\varphi_n}(y,\eps D_x) U^\eps(t,x,y) dy.
\end{equation}
Moreover, for every $t\in\R$, 
\begin{equation}\label{convergence} 
\limsup_{\eps\to 0^+}  \left\| \sum_{n>N} \psi^\eps_n(t,\cdot)\right\|_{L^2(\R^d)} 
\Tend{N}{\infty}0.
\end{equation}
\end{lemma}
The proof of this result requires two important technical facts that we gather in the next two remarks. 
\begin{remark}\label{rem:eqnorm}
Modulo the addition of a positive constant to equation \eqref{eq:schro}, we may assume that $P(\eps D_x)$ is a positive
operator (this will modify the solutions only by a constant phase in time). In that case there exists constants $\eps_0,c>0$ such that:
\[
c^{-1}\|U\|_{H^s_\eps(\R^d\times\T^d)}\leq\
\|\left\langle\eps D_x\right\rangle^s U\|_{L^2(\R^d\times\T^d)}+\|P(\eps D_x)^{s/2}U\|_{L^2(\R^d\times\T^d)}
\leq
 c\|U\|_{H^s_\eps(\R^d\times\T^d)},
 \]
for every $U\in H^s(\R^d\times\T^d)$ and $0<\eps<\eps_0$.
\end{remark}

\begin{remark}\label{rem:ueq}
In view of \eqref{eq:underline}, and the fact that $P(\xi)$ depends analytically on $\xi$, follows that $\xi\mapsto \varphi_j(\cdot,\xi)$, $j\in\N^*$,  are continuous functions from  $\R^d$ to $L^2(\T^d)$ (see also \cite{Wilcox78}) and that, for every $s>0$, the family $P(\xi)^s$ depends continuously on $\xi$ (with $P(\xi)$ positive).
\end{remark}

These remarks imply the boundedness of the operators $P_{\varphi_j}$ in $H^s_\eps(\T^d\times \R^d)$ for all $j\in\N_0$. To see this, note that 
Remark \ref{rem:ueq} implies that formula~\ref{def:Pphi}
defines a bounded operator on $L^2(\T^d\times \R^d)$. Moreover, we have 
\[
[P(\eps D_x)^{s/2},P^\eps_{\varphi_j}]=[\left\langle\eps D_x\right\rangle^s,P^\eps_{\varphi_j}]=0,
\]
if follows from Remark \ref{rem:eqnorm} that there exists $c_1>0$ such that for all $W\in H^s_\eps(\R^d\times \T^d)$,
\[
\|P^\eps_{\varphi_j}W\|_{H^s_\eps(\R^d\times\T^d)}\leq  c_1\|W\|_{H^s_\eps(\R^d\times\T^d)},
\]  
and, more generally, that every $W\in H^s_\eps(\R^d\times\T^d)$ can be expressed in the topology of $H^s_\eps(\R^d\times\T^d)$ as:
\[
W=\sum_{n\in\N^*}P^\eps_{\varphi_n} W.
\]

\begin{proof}[Proof of Lemma \ref{lem:proj}] The boundedness in $H^s_\eps(\T^d\times \R^d)$ of the operator $P_{\varphi_j}$ and the boundedness of $L^\eps$ from $H^s_\eps(\T^d\times \R^d)$ to $L^2(\R^d)$ for $s>d/2$ imply that that~~\eqref{eq:sum} holds in $L^2(\R^d)$. 

It remains to prove~\eqref{convergence}. 
In view of~\eqref{eq:wpxy},~\eqref{def:Leps},
it is enough to show that if
 $(V^\eps)_{\eps>0}$ is a  bounded family in $H^s_\eps(\R^d\times \T^d )$, $s>d/2$, we have, for $d/2<r<s$, 
\[
\limsup_{\eps\to 0^+}\left\|\sum_{n>N} P^\eps_{\varphi_n} V^\eps\right\|_{H^r_\eps(\R^d\times \T^d)}\Tend{N}{\infty}0.
\]
Remark~\ref{rem:eqnorm} implies that we only have to prove
\begin{equation}\label{eq1}
\limsup_{\eps\to 0^+}\left\|\sum_{n>N} P(\eps D_x)^{r/2}P^\eps_{\varphi_n} V^\eps\right\|^2_{L^2(\R^d\times\T^d)}
+ \limsup_{\eps\to 0^+}\left\|\sum_{n>N} \langle\eps D_x\rangle^{r} P^\eps_{\varphi_n} V^\eps  \right\|^2_{L^2(\R^d\times\T^d)}\Tend{N}{\infty}0.
\end{equation}
We thus focus on proving~\eqref{eq1}.
\smallskip 

Let us consider the series $\sum_{n>N} P(\eps D_x)^{r/2} P^\eps_{\varphi_n} V^\eps$  (the proof for~$\sum_{n>N}\langle\eps D_x\rangle^{r} P^\eps_{\varphi_n} V^\eps$
is similar).
In view of \eqref{def:Pphi}, 
\begin{align*}
P(\eps D_x) P^\eps_{\varphi_n} V^\eps(x,y)& =\varphi_n(y,\eps D_x)  \varrho_n(\eps D_x)   \,\int_{\T^d} \overline{\varphi_n}(z,\eps D_x)  V^\eps(x,z) dz,
\end{align*}
This implies 
$$\left\| \sum_{n>N} P(\eps D_x)^{r/2}P^\eps_{\varphi_n} V^\eps \right\|^2_{L^2(\R^d\times\T^d)}=
\sum_{n>N} \left\|P(\eps D_x)^{r/2}P^\eps_{\varphi_n} V^\eps\right\|^2_{L^2(\R^d\times\T^d)}.$$
 We decompose $V^\eps$ in Fourier series and write $V^\eps(x,y) =\sum_{j\in\Z^d} V^\eps_j(x) {\rm e}^{2i\pi j\cdot y}$, whence
\begin{align*}
P(\eps D_x) P^\eps_{\varphi_n} V^\eps(x,y)& =\varphi_n(y,\eps D_x)  \sum_{j\in\Z^d} \varrho_n(\eps D_x)   \,\left(\int_{\T^d} \overline{\varphi_n}(z,\eps D_x) {\rm e}^{2i\pi j\cdot z}  dz \right)V^\eps_j(x)
\end{align*}
and by functional calculus
\begin{align*}
P(\eps D_x)^{r/2} P^\eps_{\varphi_n} V^\eps(x,y)& =\varphi_n(y,\eps D_x)  \sum_{j\in\Z^d} 
d_{n} (\eps D_x,j)
V^\eps_j(x)
\end{align*}
with
\[
d_{n} (\xi,j)= \varrho_n(\xi)^{r/2}   \,\left(\int_{\T^d} \overline{\varphi_n}(z,\eps D_x) {\rm e}^{2i\pi j\cdot z}  dz \right)
\]
We use three observations. 
\begin{enumerate}
\item First, if  $\delta>0$ is fixed, there exists~$J_0$ such that
$$\limsup_{\eps\to 0^+} \sum_{|j|>J_0}\int_{\R^d} (1+|\eps\xi|^2 +|j|^2) ^{r}  |\widehat{V^\eps_j}(\xi)|^2d\xi <\delta .$$ 
To see this note that:
\[
\sum_{|j|>J_0}\int_{\R^d} (1+|\eps\xi|^2 +|j|^2) ^{r}  |\widehat{V^\eps_j}(\xi)|^2 d\xi \leq (1+|J_0|^2)^{r-s} \|V^\eps\|_{H^{s}(\R^d\times \T^d)}^2,
\]
due to the definition of the $H^{s}_\eps$-norm~\eqref{def:normHs}. Since 
$(V^\eps)_{\eps>0}$ is uniformly bounded in $H^{s}_\eps(\R^d)$, the claim follows.

\item Second, given $\delta>0$ and $J_0\in\N$, one can find $R=R(\delta, J_0) >0$ such that, 
$$
 \limsup_{\eps\to 0^+}\sum_{|j|< J_0}\int_{|\eps\xi|>R}  (1+|\eps\xi|^2 +|j|^2) ^{r} |\widehat{V^\eps_j}(\xi)|^2d\xi <\delta.
$$
This follows from the estimate:
\begin{align*}
\int_{|\eps\xi|>R}  (1+|\eps\xi|^2 +|j|^2) ^{r}  |\widehat{V^\eps_j}(\xi)|^2 d\xi 
&\leq 
(1+R^2)^{r-s}\|V^\eps\|_{H^{s}(\R^d\times \T^d)}^2,
\end{align*}
and again from the fact that $(V^\eps)_{\eps>0}$ is uniformly bounded in $H^s_\eps(\R^d\times \T^d)$. 

\item Third, given $J_0, R>0$, 
\begin{equation*}
D_N(R,J_0):=\sup_{|j|\leq J_0}\sup_{|\xi|\leq R}\sum_{n>N}\left|
d_{n}(\xi,j)
\right|^2\Tend{N}{\infty}0.
\end{equation*}
To see why this holds note that,
for $j\in\Z^d$,
\begin{equation}\label{def:pk}
\R^d\ni \xi\longmapsto\sum_{n\in\N^*}\left| 
d_{n} (\xi,j)
\right|^2=\left\|P(\xi)^{r/2} {\rm e}^{2i\pi j\cdot } \right\|_{L^2(\T^d)}^2
\in (0,\infty)
\end{equation}
is a non-negative continuous function.
The claim then follows from Dini's theorem, which ensures that for every $R>0$, $j\in\Z^d$ one has:
\begin{equation*}
\sup_{|\xi|\leq R}\sum_{n>N}\left|
d_{n}(\xi,j)
\right|^2\Tend{N}{\infty}0.
\end{equation*}
\end{enumerate}
We now use these observations to treat  the series whose terms are 
\begin{align*} 
\left\| P(\eps D_x)^{r/2} P^\eps_{\varphi_n}V^\eps\right\|_{L^2(\R^d\times\T^d))}^2
&= \sum_{j\in\Z^d} 
\int _{\R^d} 
|d_{n}(\eps \xi,j) |^2
|\widehat V^\eps_j(\xi)|^2
 d\xi.
\end{align*} 
Fix $\delta>0$, and consider $J_0$ given by Point (1)
and $R=R(\delta,J_0)$ given by Point (2). Decompose the sum of integrals in three terms 
$$
\sum_{j\in\Z^d} \int_{\R^d} = \sum_{|j|\leq J_0} \int_{|\eps\xi|\leq R} + \sum_{|j|\leq J_0} \int_{|\eps\xi| > R} 
+\sum_{|j| > J_0} \int_{\R^d}.
$$
We start by analyzing the third term. Note that 
\[
\sum_{n\in\N^*}|d_n(\xi,j)|^2=\left\|P(\xi)^{r/2} {\rm e}^{2i\pi j\cdot } \right\|_{L^2(\T^d)}^2\leq c_r(1+|\xi|^2+|j|^2)^{r}
\]
Therefore,
\begin{align*}
 \limsup_{\eps\to 0^+}\sum_{n>N} \sum_{|j| > J_0}& \int_{\R^d} 
|d_n (\eps\xi ,j)|^2
|\widehat V^\eps_j(\xi)|^2
 d\xi
 \leq  \limsup_{\eps\to 0^+}\sum_{|j| > J_0} \int_{\R^d} 
\sum_{n\in\N^*} |d_n (\eps\xi ,j)|^2
|\widehat V^\eps_j(\xi)|^2
 d\xi
 \\
 &\leq c_r \limsup_{\eps\to 0^+}  \sum_{|j| > J_0}  \int_{\R^d} (1+|\eps\xi|^2 +|j|^2) ^{r}  |\widehat V^\eps_j(\xi)|^2
 d\xi <c_r \delta,
\end{align*}
using observation (1).

The second term is analyzed using observation (2):
\begin{align*}
 \limsup_{\eps\to 0^+} \sum_{n>N} \sum_{|j|\leq J_0} 
\int_{|\eps\xi| > R}& 
|d_n (\eps\xi ,j)|^2
|\widehat V^\eps_j(\xi)|^2
 d\xi\\
 &\leq  c_r \limsup_{\eps\to 0^+} \sum_{|j|\leq J_0}  \int_{|\eps\xi| > R}  (1+|\eps\xi|^2 +|j|^2) ^{k}  |\widehat V^\eps_j(\xi)|^2
 d\xi<c_r \delta .
\end{align*}
Observation (3) ensures that 
\begin{align*}
\sum_{n>N} \sum_{|j|\leq J_0} &\int_{|\eps\xi|\leq R}
|d_n (\eps\xi ,j)|^2
|\widehat V^\eps_j(\xi)|^2
 d\xi
 \leq 
 D_N(R, J_0) \|V^\eps\|_{L^2(\R^d\times\T^d)}^2.
\end{align*}
As a consequence of this analysis:
$$\limsup_{N\to +\infty}\;  \limsup_{\eps\to 0^+}
\sum_{n>N} \sum_{j\in\Z^d} \int_{\R^d} \left| \int_{\T^d}  \varrho_n(\eps \xi)^{r/2} \overline{\varphi_n}(z,\eps\xi) {\rm e}^{2i\pi j\cdot z} dz\right|^2 |\widehat V^\eps_j(\xi)|^2
 d\xi<2c_r\delta .$$
Since $\delta$ is arbitrary, the result follows. 
\end{proof}

Lemma~\ref{lem:proj}  provides an important element of  the proof of the Theorems~\ref{theorem:1d} and~\ref{theorem0} of this paper. It allows to reduce the problem to solutions consisting only of finite superposition of Bloch modes, that we are going to study with a semi-classical perspective, as explained in the next section.

%%%%%%%%%%%%%%%%%%%%%%%%%%%%%%%%%%%%%%%%%%%

\section{Semi-classical approach to the energy dynamics }\label{sec:basic}

The nature of the propagation of the asymptotic energy density for high-frequency solutions to semi-classical dispersive-type equations is better understood if the usual, physical-space, energy density is lifted to a phase-space energy density. There is no canonical lifting procedure, roughly speaking these choices correspond to different quantization procedures. Here we will work with the lifting procedure that corresponds to the Weyl quantization, from which the Wigner functions are obtained (see the definition in~\eqref{def:Wigner}). 

\smallskip 

It should be noted that although the asymptotic limit in equation \eqref{eq:schro} we are interested might not appear to fit in the semi-classical regime one can indeed place it in that context. One can check that any solution $\psi^\eps$ of \eqref{eq:schro} becomes, after rescaling in time as $u^\eps(t,\cdot):=\psi^\eps(\eps t,\cdot)$, a solution to a semi-classical Schrödinger equation with highly oscillating potential:
\[
i\eps \partial_t u^\eps(t,x)+\frac{\eps^2}{2}\Delta_x u^\eps(t,x)-V_{\rm per}\left(\dfrac{x}{\eps}\right)u^\eps(t,x) - \eps^2V_{\rm ext}(\eps t,x)u^\eps(t,x) =0.
\] 
Hence, the asymptotic limit we are interested in can be viewed as performing simultaneously the semi-classical and long-time limits. This approach was pursued in \cite{CFM1, CFM2} to deal with the case where no crossings between Bloch bands are present, and this point of view is also adopted in references~\cite{AP05,Spar06}. Here the situation is more complicated, as interactions between projections on different Bloch bands may occur. This regime involving performing simultaneously the semi-classical and long time limit  has been useful in other contexts, we refer the reader to the survey articles \cite{AM12,MaciaDispersion,MaciaLille}.

 \smallskip 
 
 In Section~\ref{sec:semiclas} below, we recall elements of the theory of semi-classical measures that we apply to $(\psi^\eps)_{\eps>0}$ in the next two sections.  We first  discuss in Section~\ref{subsec:psiepspsiepsn} the relations between the semi-classical measures of~$(\psi^\eps)_{\eps>0}$ and those of the families $(\psi^\eps_n)_{\eps>0}$ that have been introduced in~\eqref{eq:sum} and~\eqref{eq:psiepsn}. Then, we analyze the localisation properties of these semi-classical measures in Section~\ref{subsec:wigner}, which motivates a two-microlocal approach. 
\subsection{Semi-classical measures and energy densities}\label{sec:semiclas}

 Let us recall briefly some basic facts of the theory of semi-classical measures \cite{Ge91,GerLeich93,LionsPaul} that will be needed in the sequel. From now on, for every $s\in\R,N\in\N^*$, $H^s_\eps(\R^d,\C^N)$ will denote the space $H^s_\eps(\R^d)^N$ equipped with the norm:
\[
\|\Psi\|_{H^s_\eps(\R^d,\C^N)}=\left(\sum_{j=1}^N\|\Psi_j\|_{H^s_\eps(\R^d)}^2\right)^{1/2},\;\;\Psi=(\Psi_1,\cdots, \Psi_N).
\] 
We associate to every $\Psi\in L^2(\R^d,{\C}^N):=H^0_\eps(\R^d,\C^N)$ a 
 microlocal version $W^\eps_{\Psi}$ of the (matrix-valued) energy density 
 $$\Psi\otimes\overline{\Psi}=(\Psi_i\overline \Psi_j)_{1\leq i,j\leq N}\in \C^{N\times N}.$$
 The matrix-valued function $W^\eps_{\Psi}\in L^2(\R^{2d},\C^{N\times N})$ is defined by 
\begin{equation}\label{def:Wigner}
W^\eps_{\Psi}(x,\xi)=\int_{\R^d} {\rm e}^{i\xi\cdot v}\Psi\left(x-\frac{\eps v}{2}\right)\otimes\overline{\Psi}\left(x+\frac{\eps v}{2}\right) \frac{dv}{(2\pi)^d},
\end{equation}
and its action on  symbols $a\in{\mathcal C}_0^\infty(\R^{2d},\C^{N\times N})$  is related with semi-classical pseudodifferential calculus according to 
 $$
\int_{\R^{2d}} \Tr_{\C^{N\times N}}(a(x,\xi)W^\eps_{\Psi}(x,\xi))dx\,d\xi =  \left( \op_\eps(a) \Psi, \Psi \right)_{L^2(\R^d,\C^N)}
$$
 where $\op_\eps(a)$ denotes the matrix-valued semi-classical pseudodifferential operator of symbol $a$.
The Wigner function satisfies the following bounds for every $a\in\cC_0^\infty(\R^{2d},\C^{N\times N})$ and $\eps>0$:
\begin{equation} \label{eq:cv}
\left| \left( \op_\eps(a) \Psi, \Psi \right)_{L^2(\R^d,\C^N)} \right| \leq C_d \|\Psi\|^2_{L^2(\R^d,\C^{N})}\|a\|_{\cC^{d+2}(\R^{2d},\C^{N\times N})};
\end{equation}
for $C_d>0$ depending only on $d$. If in addition $a\geq 0$ (meaning that $a$ takes values in the set of non-negative Hermitian matrices),
\begin{equation} \label{eq:garding}
\int_{\R^{2d}} \Tr_{\C^{N\times N}}(a(x,\xi)W^\eps_{\Psi}(x,\xi))dx\,d\xi\geq -C_a\eps\|\Psi\|^2_{L^2(\R^d,\C^{N})},
\end{equation}
for some $C_a>0$ that can be computed in terms of $a$ and its derivatives. 
Estimate \eqref{eq:cv} is a consequence of the Calderón-Vaillancourt theorem \eqref{eq:pseudo}, whereas \eqref{eq:garding} is a reformulation of Gårding's inequality \eqref{eq:gengard}.  A direct computation also shows that $W^\eps_\Psi$ is actually a lift of $\Psi\otimes\overline{\Psi}$:
\begin{equation}\label{eq:lift}
\int_{\R^d}W^\eps_\Psi(x,\xi)d\xi=\Psi\otimes\overline{\Psi}(x).
\end{equation}
Suppose now that $(\Psi^\eps)_{\eps>0}$  is a bounded sequence in $L^2(\R^d,{\C}^N)$; then \eqref{eq:cv} ensures that $(W^\eps_{\Psi^\eps})$ is a bounded sequence of distributions. In addition, \eqref{eq:garding} implies that all its accumulation points are non negative Radon matrix-valued measures, that is, measures valued on the set of complex $N\times N$ Hermitian positive-semidefinite matrices. Moreover, any measure $\mu$ obtained from $(W^\eps_{\Psi^\eps})$ along some subsequence $(\eps_\ell)$ satisfies (see \eqref{eq:husimi}, \eqref{eq:bddhusimi})
\[
\mu(\R^{2d})\leq\liminf_{\ell\to \infty}\|\Psi^{\eps_\ell}\|_{L^2(\R^d)}^2.
\]
These measures are called \textit{semi-classical} or \textit{Wigner measures} of the family~$(\Psi^\eps)_{\eps>0}$. 
\begin{remark}\label{rem:lowreg}
If $(W^{\eps_\ell}_{\Psi^{\eps_\ell}})$ converges in $\Sch'(\R^{2d})$ to the semiclassical measure $\mu$ then, for every $a\in C_0(\R^{d}_{x}\times\R^d_{\xi})$ that is $d+2$ times continuously differentiable in $x$ the following holds:
\[
\int_{\R^{2d}} \Tr_{\C^{N\times N}}(a(x,\xi)W^{\eps_\ell}_{\Psi^{\eps_\ell}}(x,\xi))dx\,d\xi\Tend{\ell}{\infty}\int_{\R^{2d}} \Tr_{\C^{N\times N}}(a(x,\xi)\mu(dx,d\xi)).
\] 
This follows from Remark \ref{rem:sep} and assertion (1) of Lemma \ref{lem:singulier}.
\end{remark}
Finally, the lift property \eqref{eq:lift} is transferred to an accumulation point $\mu$ generated from a subsequence~$(\Psi^{\eps_\ell})$:
\[
\forall \phi\in\cC_0(\R^d,\C^{N\times N}),\quad \lim_{\ell\to\infty}\int_{\R^d}\Tr_{\C^{N\times N}} (\phi(x) (\Psi^{\eps_\ell}\otimes\overline{\Psi^{\eps_\ell}})(x))dx=\int_{\R^d}\Tr_{\C^{N\times N}}(\phi(x)\mu(dx,d\xi)),
\] 
provided that no mass of $(\Psi^{\eps_\ell})$ is lost at infinity in Fourier space:
\begin{equation}\label{def:epsosc}
\limsup_{\ell\to\infty}\int_{|\eps_\ell\xi|>R}|\widehat{\Psi^{\eps_\ell}}(\xi)|^2 d\xi\Tend{R}{\infty} 0.
\end{equation}
This condition, referred sometimes to as $\eps$\textit{-oscillation}, is fulfilled as soon as the sequence $(\Psi^\eps)_{\eps>0}$ is bounded in $H^s_\eps(\R^d)^N$ for some $s>0$. 

\smallskip 

Let us conclude this concise review of semi-classical measures by recalling how the matrix-valued semi-classical measure $\mu=(\mu_{i,j})_{1\leq i,j,\leq N}$ is related to the semi-classical measures of the families of components $(\Psi^\eps_j)$, for $j=1,\hdots,N$. Suppose that the subsequence $(\Psi^{\eps_\ell})$ gives the semi-classical measure~$\mu$. Then, for every $1\leq i,j, \leq N$,
\begin{equation}\label{eq:matel}
\forall a\in\cC_0^\infty(\R^{2d}),\quad \lim_{\ell\to\infty}\left( \op_\eps(a) \Psi^{\eps_\ell}_i, \Psi^{\eps_\ell}_j \right)_{L^2(\R^d)}=\int_{\R^{2d}}a(x,\xi)\mu_{i,j}(dx,d\xi).
\end{equation}
Moreover, since $\mu$ takes values on the set of Hermitian positive-semidefinite matrices, one also has that the~$\mu_{i,i}$ are non-negative  (scalar) Radon measures and that $\mu_{i,j}$ is absolutely continuous with respect to both $\mu_{i,i}$ and $\mu_{j,j}$. The latter condition implies that $\mu_{i,j}=0$ as soon as $\mu_{i,i}$ and $\mu_{j,j}$ are mutually singular. In particular:
\begin{equation}\label{eq:diag0}
\mu_{i,i}\,\bot\,\mu_{j,j}\implies \forall a\in\cC_0^\infty(\R^{2d}),\quad \lim_{\ell\to\infty}\left( \op_\eps(a) \Psi^{\eps_\ell}_i, \Psi^{\eps_\ell}_j \right)_{L^2(\R^d)}=0.
\end{equation}

In this article, we are mainly interested in time-dependent versions of these objects. The modifications  required in order to adapt the theory to this context are rather straightforward. Suppose now that $(\Psi^\eps)_{\eps>0}$ is bounded in $L^\infty(\R_t; L^2(\R^d_x,\C^N))$. Define $W^\eps_{\Psi^\eps}$ as
\begin{equation}\label{def:VWigner}
W^\eps_{\Psi^\eps}(t,x,\xi):=W^\eps_{\Psi^\eps(t,\cdot)}(x,\xi)=\int_{\R^d} {\rm e}^{i\xi\cdot v}\Psi^\eps\left(t,x-\frac{\eps v}{2}\right)\otimes\overline{\Psi^\eps}\left(t,x+\frac{\eps v}{2}\right) \frac{dv}{(2\pi)^d}.
\end{equation}
Then \eqref{eq:cv} again implies that for every $\theta\in L^1(\R)$ and every $a\in{\cC}_0^\infty(\R^{2d},\C^{N\times N})$,
\begin{multline}\label{eq:Vwbdd}
\left|\int_\R\int_{\R^{2d}} \theta(t)  \Tr_{\C^{N\times N}}( a(x,\xi) W^\eps_{\Psi^\eps}(t,x,\xi))dx\,d\xi\,dt\right|\\
\leq C_d \|\Psi^\eps\|_{L^\infty(\R_t; L^2(\R^d_x,\C^N))}^2 \|\theta\|_{ L^1(\R)} \|a\|_{\cC^{d+2}(\R^{2d},\C^{N\times N})}.
\end{multline}
This ensures that $(W^\eps_{\Psi^\eps})$ is bounded in $\Sch'(\R\times \R^{2d})$. Moreover, any accumulation point $\mu$ of this family is a non negative Radon measure on $\R\times\R^{2d}$, because of \eqref{eq:garding}. It follows from \eqref{eq:Vwbdd} that the projection of~$\mu$ onto the $t$-variable  is absolutely continuous with respect to the Lebesgue measure on $\R$. Therefore, we conclude using the disintegration theorem the existence of a measurable map from $t\in\R$ to non negative, finite,  matrix-valued Radon measures $\mu^t$ on $\R^{2d}$ such that $\mu(dt,dx,d\xi)=\mu^t(dx,d\xi) dt$.

\smallskip 

Summing up, for every sequence $(\eps_\ell)_{\ell \in \N}$ going to $0$ as~$\ell$ goes to~$+\infty$  such that $(W^{\eps_\ell}_{\Psi^{\eps_\ell}})$ converges in the sense of distributions the following holds: for all $\theta\in L^1(\R)$ and $a\in{\cC}_0^\infty(\R^{2d},\C^{N\times N})$, 
\begin{equation}\label{def:Vconv}
\int_\R\int_{\R^{2d}} \theta(t) \Tr_{\C^{N\times N}}(a(x,\xi)W^{\eps_\ell}_{\Psi^{\eps_\ell}}(t,x,\xi))dx\,d\xi\, dt\Tend{\ell}{\infty} \int _\R \int_{\R^{2d}}\theta(t)   \Tr_{\C^{N\times N}}(a(x,\xi)\mu^t(dx,d\xi)) dt.
\end{equation}
If the sequence $(\Psi^{\eps_\ell}(t,\cdot))$ is in addition $\eps$-oscillating \eqref{def:epsosc} for almost every $t\in\R$, the projections of the measures $\mu^t$ on the $\xi$-variable are the limits of the energy densities: for every $\theta\in L^1(\R)$, $\phi\in\cC_0(\R^d,\C^{N\times N})$,
\begin{equation}\label{def:Venden}
\int_\R\int_{\R^d}  \theta(t) \Tr_{\C^{N\times N}} (\phi(x) (\Psi^{\eps_\ell}\otimes\overline{\Psi^{\eps_\ell}})(t,x))dx \Tend{\ell}{\infty} \int_\R\int_{\R^d} \theta(t) \Tr_{\C^{N\times N}}(\phi(x)\mu^t(dx,d\xi))\,dt.
\end{equation}

\begin{remark}\label{rem:otherprop}
Time-dependent analogues of \eqref{eq:matel}, \eqref{eq:diag0} also hold after replacing $\mu_{i,j}$ by $\mu^t_{i,j}$ and averaging in the $t$-variable. So does the analogue of Remark \ref{rem:lowreg}.
\end{remark} 

\subsection{The semi-classical measure of $(\psi^\eps)_{\eps>0}$ in terms of those of the sequences $(\psi^\eps_n)_{\eps>0}$, $n\in\N$} \label{subsec:psiepspsiepsn}

We now focus on the basic properties of semi-classical measures associated to a sequence~$(\psi^\eps)_{\eps>0}$ of solutions to \eqref{eq:schro}, issued from initial data~$(\psi^\eps_0)_{\eps>0}$ that are bounded in $H^s_\eps(\R^d)$ for some $s>d/2$, and onto clarifying how they are related to those of the families of projections $(\psi^\eps_n)_{\eps>0}$ defined in \eqref{eq:psiepsn}.

\smallskip 

First note that the highly oscillating character of the Schrödinger propagator prevents in general to be able to extract a subsequence along which $W^\eps_{\psi^\eps(t,\cdot)}$ will converge for every $t\in\R$. Following  \cite{MaciaAv,AM14,AFM15} we consider time averages of the Wigner functions we just described.  

Identity \eqref{def:Vconv} applied to this context states that whenever $(W^{\eps_\ell}_{\psi^{\eps_\ell}})$ converges in the sense of distributions for some sequence $(\eps_\ell)_{\ell \in \N}$ going to $0$ as~$\ell$ goes to~$+\infty$  the following holds: for all $\theta\in L^1(\R)$ and $a\in{\cC}_0^\infty(\R^{2d})$, 
\begin{equation}\label{def:varsigma}
\int_\R\int_{\R^{2d}} \theta(t) a(x,\xi)W^{\eps_\ell}_{\psi^{\eps_\ell}}(t,x,\xi)dx\,d\xi\, dt\Tend{\ell}{\infty} \int _\R \int_{\R^{2d}}  \theta(t) a(x,\xi)\varsigma^t(dx,d\xi) dt,
\end{equation}
where, for a.e. $t\in\R$,  $\varsigma^t$ is a non negative Radon measure on $\R^{2d}$. 

Besides, $\varsigma^t$  can be related to the weak limits of the energy densities since the family $(\psi^\eps)_{\eps>0}$ is $\eps$-oscillating. 

\begin{remark}\label{rem:psiepsosc}
If $(\psi^\eps_0)_{\eps>0}$ is bounded in $H^s_\eps(\R^d)$ for some $s>d/2$ then $(\psi^\eps(t,\cdot))_{\eps>0}$, the corresponding family of solutions to \eqref{eq:schro},   is $\eps$-oscillating for every $t\in\R$. This follows from \cite[Lemma~6.2]{CFM2} applied to the family $(U^\eps)_{\eps>0}$ of solutions to \eqref{eq:U} once one notices that for  $r\in({d\over 2},s)$, $R>0$ and $t\in\R$,
\[
\int_{|\eps\xi|>R}\|\widehat{U^\eps}(t,\xi,\cdot)\|_{H^r(\T^d )}^2d\xi\leq R^{-2(s-r)}(\|\psi^\eps_0\|_{H^s_\eps(\R^d)}+C_s\eps|t|)^2,
\] 
as follows from estimate~(\ref{eq:wpxy}).
\end{remark}
As a consequence of this, \eqref{def:Venden} implies one  has for the subsequence $(\eps_\ell)$ of~(\ref{def:varsigma}) and for $\theta\in  {\mathcal C}_0^\infty(\R)$, $\phi\in{\mathcal C}_0^\infty(\R^d)$, 
\begin{equation}\label{eq:projmut}
\int_\R\int_{\R^d}\theta(t) \phi(x)|\psi ^{\eps_\ell}(t,x)|^2 dx\, dt \Tend{\ell}{+\infty} \int_\R\int_{\R^{2d}} \theta(t)  \phi(x)  \varsigma^t( dx,d\xi)dt.
\end{equation}
For $n,n'\in\N^*$, we use the notation $W^\eps_{n,n'}$ to refer to the Wigner function of the pair $\psi^\eps_n,\psi^\eps_{n'}$. In other words, for every $a\in{\mathcal C}_0^\infty(\R^{2d})$,
$$
\int_{\R^{2d}} a(x,\xi) W^\eps_{n,n'}(t,x,\xi)dx\,d\xi  =  \left( \op_\eps(a) \psi^\eps_n (t,\cdot) , \psi^\eps_{n'}(t,\cdot) \right)_{L^2(\R^d)}.
$$
The same argument presented before shows that for $n,n'\in\N$, there exists a sequence $(\eps_\ell)_{\ell \in \N}$ going to $0$ as~$\ell$ goes to~$+\infty$  such that, for all $\theta\in L^1(\R)$, $a\in{\mathcal C}_0^\infty(\R^{2d})$,
\begin{equation}\label{def:mu}
\int_\R\int_{\R^{2d}} \theta(t) a(x,\xi) W^\eps_{n,n'}(t,x,\xi)dx\,d\xi\, dt\Tend{\ell}{\infty} \int _\R \int_{\R^{2d}}  \theta(t) a(x,\xi)\mu^t_{n,n'} (dx,d\xi)) dt,
\end{equation}
where, for a.e. $t\in\R$, $\mu^t_{n,n'}$ is a (signed) Radon measure on $\R^{2d}$.

\begin{proposition} \label{prop:sum}
There exist a subsequence $(\eps_\ell)_{\ell \in \N}$ going to $0$ as~$\ell$ goes to~$+\infty$  such that \eqref{def:varsigma} and~\eqref{def:mu} hold simultaneously for all $n,n'\in\N^*$. In addition, for a.e. $t\in\R$,
\[
\varsigma^t= \sum_{ n,n'\in\N^*} \mu^t_{n,n'},
\]
the convergence of the series being understood in the weak-$\ast$ topology of the space of Radon measures on $\R^{2d}$.
\end{proposition}
\begin{proof}
We proceed to a first extraction to have~\eqref{def:varsigma} and  we keep denoting by $\eps$ the resulting subsequence. We
put
\[
\Psi^\eps_N:=(\psi^\eps_1,\hdots,\psi^\eps_N)\in \cC(\R_t; L^2(\R^d_x,\C^N)).
\]
We know that $(W^\eps_{\Psi^\eps_N})$, defined by \eqref{def:VWigner}, are uniformly bounded in $\mathcal C(\R_t;\Sch'(\R^{2d},\C^{N\times N}))$, both in $\eps>0$ and $N\in\N^*$. 

By \eqref{def:Vconv}, any accumulation point of $(W^\eps_{\Psi^\eps_N})$ obtained along some subsequence $(\eps_\ell)_{\ell \in \N}$ is a time-dependent family of non negative matrix-valued Radon measures $\mu_N^t$. 
By diagonal extraction, we can find a sequence $(\eps_\ell)_{\ell \in \N}$ such that $(W^{\eps_\ell}_{\Psi^{\eps_\ell}_N})_{\eps>0}$ converge for every $N\in\N^*$. We denote by $(\mu^t_N)_{N\in\N^*}$ their respective limits. By \eqref{eq:matel} we know that, for every $n,n'\leq N\leq N'$ one has:
\[
(\mu^t_N)_{n,n'}=(\mu^t_{N'})_{n,n'}=\mu^t_{n,n'},
\]
where $\mu^t_{n,n'}$ is obtained through \eqref{def:mu}. This shows that we can find a sequence $(\eps_\ell)_{\ell \in \N}$ as claimed.

Define now $\psi^{N,\eps}:=\sum_{n=1}^N\psi^\eps_n$. One has that for $a\in\cC_0^\infty(\R^{2d})$ and $t\in\R$,
\[
\int_{\R^{2d}}  a(x,\xi)W^{\eps_\ell}_{\psi^{N,\eps_\ell}}(t,x,\xi)dx\,d\xi
=\int_{\R^{2d}}a(x,\xi)  \Tr_{\C^{N\times N}}\left( Q\, W^{\eps_\ell}_{\Psi^{\eps_\ell}_N}(t,x,\xi) \right)dx\,d\xi,
\]
where $Q$ is the $N\times N$ matrix whose all entries are equal to one. Therefore, $(W^{\eps_\ell}_{\psi^{N,\eps_\ell}})_{\ell \in \N}$ converges to the semi-classical measure given, for a.e. $t\in\R$, by
\[
\varsigma^t_N=\sum_{1\leq n,n'\leq N}\mu^t_{n,n'}.
\]
Finally, \eqref{eq:sum} and Lemma \ref{lem:proj} imply that for every $\theta\in L^1(\R)$,
\[
\limsup_{\ell\to\infty} \int_\R\theta(t) \|\psi^{\eps_\ell}(t,\cdot)-\psi^{N,\eps_\ell}(t,\cdot)\|_{L^2(\R^d)}^2dt\Tend{N}{\infty} 0;
\]
which in turn guarantees that, for every $\theta\in L^1(\R), a\in\cC_0^\infty(\R^{2d})$,
\[
\int_{\R}\int_{\R^{2d}} \theta(t)a(x,\xi)\varsigma^t_N(dx,d\xi)dt\Tend{N}{\infty}  \int_{\R}\int_{\R^{2d}} \theta(t)a(x,\xi)\varsigma^t(dx,d\xi)dt.
\]
\end{proof}

The rest of the article is devoted to computing the measures $\mu^t_{n,n'}$.

\subsection{Localization of semi-classical measures}\label{subsec:wigner}

If the operator $L^\eps P^\eps_{\varphi_n}$ is applied to problem \eqref{eq:U}, one deduces that $\psi^\eps_n$ (which is defined by \eqref{eq:psiepsn}) satisfies the pseudo-differential equation
\begin{equation}\label{eq:mode_n}
\left\{\begin{array}{l}
i\eps^2 \partial_t \psi_n^\eps(t,x)) =  \varrho_n(\eps D_x) \psi_n^\eps (t,x)+ \eps^2 f^\eps_n(t,x),\quad (t,x)\in\R\times\R^d,\vspace{0.2cm}\\
\psi_n^\eps(0,x)=\varphi_n\left({x\over\eps},\eps D_x\right) \int_{\T^d} \overline{\varphi_n}(y,\eps D_x) \psi^\eps_0(x)dy,
\end{array}\right.
\end{equation} 
with 
$$f^\eps_n(t,x):= \varphi_n\left({x\over\eps},\eps D_x\right) \int_{\T^d} \overline{\varphi_n}(y,\eps D_x) (V_{\rm ext} (t,x) U^\eps(t,x,y))dy.$$

This fact will be used to obtain all the information on the measures $\mu^t_{n,n'}$ defined in \eqref{def:mu} that is relevant to our purposes. In this section we gather some basic facts; in following sections we will introduce a more precise machinery that will allow us to obtain a complete picture.

\begin{proposition}\label{prop:easysupp}
Let $(\psi^\eps_0)$ be bounded in $H^s_\eps(\R^d)$ for some $s>d/2$. For any $n,n'\in\N^*$, let $(\psi^\eps_n)_{\eps>0}$ and $(\psi^\eps_{n'})$ be defined by~ \eqref{eq:psiepsn} and let $\mu^t_{n,n'}$ be given by \eqref{def:mu}. Let $\Omega\subseteq\R^d$ be open and invariant by translations by $2\pi\Z^d$. Then the following hold.
\begin{enumerate}
\item If $\nabla_\xi\varrho_n\in{\rm Lip}(\R^d)$ on $\Omega$ and $\nabla_\xi \varrho_n|_\Omega\not=0$, then
\[
\mu_{n,n} ^t(\R^d\times\Omega) =0,\quad \text{ for a.e. }t\in\R.
\]
\item Let $\delta>0$ and suppose that
$$\Omega\subset  \{\xi\in\R^d\,:\,  |\varrho_n(\xi)-\varrho_{n'}(\xi)|\geq \delta\}.$$
Then $\displaystyle{|\mu_{n,n'} ^t|(\R^d\times\Omega) =0,\quad \text{ for a.e. }t\in\R.}$
\end{enumerate}
\end{proposition}

\begin{proof} Point~1 is proved in an analogous manner than Proposition~3.4 in~\cite{CFM2}. Using the calculus of semi-classical pseudo-differential operators with low regularity of Lemma \ref{lem:singulier} it is possible to prove that for every $\theta\in{\mathcal C}_0^\infty (\R)$ and $a\in\cC^\infty_0 (\R^d\times\Omega)$,
$$ \int_\R\theta(t) ({\rm op}_\eps(\nabla_\xi\varrho_n \cdot \nabla_x a) \psi^\eps_n(t,\cdot) ,\psi^\eps_{n}(t,\cdot))_{L^2(\R^d)}dt \Tend{\eps}{0} 0.$$
By \eqref{def:mu}, this implies that, for almost every $t\in\R$, 
\[
\int_{\R^{d}\times\Omega}\nabla_\xi\varrho_n (\xi)\cdot \nabla_x a(x,\xi)\mu_{n,n}^t(dx,d\xi)=0.
\]
This implies that the measure $\mu_{n,n}^t {\bf 1 }_{\R^d\times \Omega}$ is invariant by the flow $(x,\xi) \mapsto (x+s\nabla\varrho_n(\xi), \xi)$. Since $\mu_{n,n}^t $ is non negative and finite, necessarily it is identically~$0$. 

\smallskip

For proving Point~2, it is enough to obtain,  for every  $\theta\in{\mathcal C}_0^\infty (\R)$ and
$a\in{\mathcal C}_0^\infty (\R^{d}\times\Omega)$:
$$\int_\R \theta(t) \left({\rm op}_\eps(a) \psi^\eps_n(t,\cdot),\psi^\eps_{n'}(t,\cdot)\right)_{L^2(\R^d)} dt\Tend{\eps}{0}0.$$ 
We have 
\begin{multline}\label{eq:derpsi}
i\eps^2 \frac d{dt}  \left({\rm op}_\eps(a) \psi^\eps_n(t,\cdot),\psi^\eps_{n'}(t,\cdot)\right)_{L^2(\R^d)} \\ =
 \left( \left(\varrho_{n'}(\eps D_x) {\rm op}_\eps(a) - {\rm op}_\eps(a)\varrho_n(\eps D_x)\right) \psi^\eps_n(t,\cdot),\psi^\eps_{n'}(t,\cdot)\right)_{L^2(\R^d)}+\eps^2 R^\eps(t),
\end{multline}
where $ |R^\eps(t)|\leq C\|f^\eps_n(t,\cdot)\|_{L^2(\R^d)}^2$ is locally uniformly bounded in $t\in\R$ for every $\eps>0$. 

By Lemma~\ref{lem:singulier} (2), the following holds with respect to  the  ${\mathcal L}(L^2(\R^d))$ norm:
$$\varrho_{n'}(\eps D_x) {\rm op}_\eps(a) - {\rm op}_\eps(a)\varrho_n(\eps D_x)= 
{\rm op}_\eps\left( (\varrho_{n'}-\varrho_n) a\right) +O(\eps ).$$
This identity together with integration by parts transforms \eqref{eq:derpsi} into
\begin{multline*}
\int_\R \theta(t) \left({\rm op}_\eps\left((\varrho_{n'}-\varrho_n) a\right) \psi^\eps_n(t,\cdot),\psi^\eps_{n'}(t,\cdot)\right)_{L^2(\R^d)} dt 
  \\=  \frac{\eps^2}{i} \int _\R \theta'(t) \left({\rm op}_\eps(a) \psi^\eps_n(t,\cdot),\psi^\eps_{n'}(t,\cdot)\right)_{L^2(\R^d)} dt + O(\eps).
\end{multline*} 
Taking limits $\eps \rightarrow 0$, which is possible by Remarks \ref{rem:lowreg} and \ref{rem:otherprop}, we obtain 
\[
 \int_\R\int_{\R^{2d}}\theta(t)(\varrho_{n'}(\xi)-\varrho_n(\xi)) a(x,\xi)\mu_{n,n'} ^t(dx,d\xi) dt=0.
\] 
By density, this relation holds for all  $a\in{\mathcal C}_0 (\R^{d}\times\Omega)$, in particular for $\tilde{a} = (\varrho_n-\varrho_{n'})^{-1} a$. This shows that, as we wanted to prove
\[
\forall \theta\in\cC_0^\infty (\R),\; \forall a\in{\mathcal C}_0 (\R^{d}\times\Omega),\quad \int_\R\int_{\R^{2d}}\theta(t) a(x,\xi)\mu_{n,n'} ^t(dx,d\xi) dt=0.
\]
\end{proof}

Proposition~\ref{prop:easysupp} shows that~$\mu^t_{n,n}$ can only charge the sets~$\Lambda_n$ of critical points of~$\varrho_n$ or the set where~$\varrho_n$ has a conical crossing with another Bloch energy (\textit{i.e.} where~$\varrho_n$ ceases to be~$\cC^{1,1}(\R^d)$). It also shows that~$\Sigma_{n,n'}$ is the only region where the measures~$\mu^t_{n,n'}$ can be non-zero. Since, assuming {\bf H1}, the crossing sets reduce to $\cup_{n\in\N_0} \Sigma_n$, the analysis of the measures $\mu^t_{n,n'}$ will be performed in the following sections by means of a second microlocalisation above the sets $\Lambda_n$ and $\Sigma_n$, under the assumption~{\bf H1}, {\bf H2} and {\bf H3}.

   %%%%%%%%%%%%%%%%%%%%%%%%%%%%%%%%%%%%%%%%%%%%%%%%%%%%%%%%%%%%%%%%%%

\section{Two microlocal analysis}\label{section:twomic}

The analysis of the concentration of a family on a submanifold of the phase space turned out to be an important element of the analysis of its behavior. Two-microlocal semi-classical measures gives a quantitative overview on these concentration phenomena. They were first introduced simultaneously and independently in~\cite{NierScat} and by one of the author in her thesis (see the articles~\cite{FermanianShocks,Fermanian_note1,Fermanian_Note2}),  and  they have found applications in different fields, as for example~\cite{FG02,MaciaTorus}  and articles connected to these ones. We recall in Section~\ref{sec:twomic}  the definition of two-microlocal semi-classical measures that is useful in our context and  apply the theory to families $(\psi^\eps)_{\eps>0}$ of solution to~\eqref{eq:schro} in Section~\ref{sec:twomic} in the frameworks of Theorem~\ref{theorem0} and of Theorem~\ref{theo:interaction}.  This leads to the statement of two results (Theorems~\ref{theo:2mic1} and~\ref{theo:2mic2}) that will be proved in the next Sections~\ref{sec:proof1} and~\ref{sec:proof2}.

\subsection{Two-scale semi-classical  measures}\label{sec:twomic}

We study here the concentration of a bounded family~$(\Psi^\eps)_{\eps>0}$ of $L^\infty(\R, L^2( \R^d,\C^N))$ on a set~$\R^d\times X$ 
where $X$ is assumed to be a connected, closed embedded submanifold of $(\R^d)^*$ of codimension $p$. Following~\cite{CFM2}, we achieve  a second microlocalization above~$\R^d\times X$ and we crucially use that the geometric properties of $X$ imply that there exists a tubular neighbourhood~$U$ of $\{(\sigma,0)\,:\,\sigma\in X\}\subseteq N X$ such that the tubular coordinate map:
$$
U\ni (\sigma,v)\longmapsto \sigma+v\in(\R^d)^*;
$$
is a diffeomorphism onto its image $V$. In that case, there exists a smooth map $\sigma_X:V\longrightarrow X$ such that, for every $\xi\in V$, 
\begin{equation}\label{eq:x}
\begin{array}{ccl}
\xi=\sigma+v,\quad (\sigma,v)=(\sigma_X(\xi),\xi-\sigma_X(\xi))\in U.
\end{array}
\end{equation}
We extend the phase space $T^*\R^d:=\R^d_x\times(\R^d)^*_\xi$ with a new variable $\eta\in \overline{ \R^d}$, where~$\overline{\R^d}$ is the compactification of $(\R^d)^*$ obtained by adding a sphere ${\bf S}^{d-1}$ at infinity. The space 
$\mathcal{A}^{(2)}$
 of test functions associated with this extended phase space is formed by those functions
$$a\in{\mathcal C}^\infty(T^*\R^d_{x,\xi}\times\R^d_\eta,\C^{N\times N})$$ which satisfy the two following properties:
\begin{enumerate}
\item there exists a compact $K \subset T^*\R^d$ such that, for all $\eta\in\R^d$, the map $(x,\xi)\longmapsto a(x,\xi,\eta)$ is a smooth matrix-valued function compactly supported in $K$;
\item there exists a smooth matrix-valued function $a_\infty$ defined on $T^*\R^d\times{\bf S}^{d-1}$ and $R_0>0$ such that, if $|\eta|>R_0$, then $a(x,\xi,\eta)=a_\infty(x,\xi,\eta/|\eta|)$.  
\end{enumerate}
For $a\in  \mathcal{A}^{(2)}$ supported in $\R^d\times V\times \R^d$, we write:
$$a_\eps(x,\xi):=a\left(x,\xi,\frac{\xi-\sigma_X(\xi)}{\eps}\right).$$

\smallskip

We associate to $\Psi^\eps(t)$   a  two-microlocal Wigner distribution 
$$W^{X,\eps}(t)\in\mathcal{D}'(\R^d\times V\times \overline{\R^d}), \;\;W^{X,\eps}_{\Psi^\eps}(t)=(W_{j,k}^{X,\eps})_{1\leq j,k\leq N};$$ its action on test functions $a\in{\mathcal A}^{(2)}$ supported in $\R^d\times V\times \R^d$
 is defined by: 
\begin{equation}
\left\langle W^{X,\eps}_{\Psi^\eps}(t),a\right\rangle :=\left(\op_\eps(a_\eps) \Psi^\eps(t),\,\Psi^\eps(t)\right)_{L^2(\R^d,\C^N)}.
\end{equation}
Since the family of operators $(\op_\eps(a_\eps))_{\eps>0}$ is uniformly bounded in $L^2(\R^d,\C^{N\times N})$ (as a consequence of the Calderón-Vaillancourt theorem, see Appendix~\ref{app:pseudo}), it follows that $(W^{X,\eps}_{\Psi^\eps}(t))$ is a bounded sequence of distributions. In addition, any  smooth, compactly supported test function $a\in {\mathcal C}^\infty_0(\R^d\times V,\C^{N\times N})$ can be naturally identified to an element of ${\mathcal A}^{(2)}$ which does not depend on the last variable. For such $a$, one clearly has 
$$\left\langle W^{X,\eps}_{\Psi^\eps}(t) ,a\right\rangle= \left\langle W^{\eps}_{\Psi^\eps}(t) ,a\right\rangle;$$
hence $W^{X,\eps}_{\Psi^\eps}(t)$ is a lift of $W^{\eps}_{\psi^\eps}(t)$ to the extended phase-space.  We thus focus on the asymptotic description of the quantities
\begin{equation} \label{wignertwo}
\int_\R \theta(t) \langle W^{X,\eps}_{\Psi^\eps}(t),a\rangle dt,\;\;\theta\in L^1(\R),\;\;a\in{\mathcal A}^{(2)}.
\end{equation}

\smallskip

In order to describe the limits of these quantities, we must introduce some notations. We consider an open subset $W$ of~$V$ where there exists  $\varphi: W \longrightarrow \R^p$ a smooth function such that the $\xi\in W$ for which $\varphi(\xi)=0$ are precisely those which are in $W\cap X$. We also assume  that $d\varphi(\sigma)$ for $\sigma\in W\cap X$ is of maximal rank. These coordinates functions give parametrization of the manifolds under consideration and  for every $\sigma\in  W\cap X$, we can write 
$$N_\sigma X= \{ \, ^td\varphi(\sigma) z\;:\; z\in\R^p\}.$$
This parametrization allows to define a measure on $N_\sigma X$ and  the space $L^2(N_\sigma X,\C^2)$. Different $\varphi$ will give equivalent norms. 
The function $\varphi$ also induces a smooth map $B$ from the neighbourhood~$W$ of~$\sigma$ into the set of $d\times p$ matrices such that 
\begin{equation}\label{def:B}
\xi-\sigma_X(\xi)=B(\xi)\varphi(\xi),\;\;\xi\in W.
\end{equation}
Therefore, given a function $a\in {\mathcal C}^\infty_0(\R^d\times W\times \R^d,\C^{N\times N})$ and a point $(\sigma,v)\in T^*_\sigma X$, we can use $\varphi$ to define an operator acting on $f\in L^2(N_\sigma X,\C^N)$ by:
$$Q_a^\varphi (\sigma,v)f(z)=\int_{\R^p\times \R^p}a\left(v+ \, ^td\varphi(\sigma) \frac{z+y}{2},\sigma,  B(\sigma) \eta\right)f(y){\rm e}^{i\eta\cdot (z-y)}\frac{d\eta\,dy}{(2\pi)^p}.$$
In other words, $Q_a^\varphi(\sigma,v)$ is obtained from $a$ by applying the non-semi-classical Weyl quantization to the symbol $a\left(v+ \, ^td\varphi(\sigma) \, \cdot\,,\sigma,  B(\sigma) \, \cdot\,\right)\in {\mathcal C}^\infty_0(\R^p\times\R^p,\C^{N\times N})$,
\begin{equation}\label{eq:Qa}
Q_a^\varphi (\sigma,v)= a^W\left(v+ \, ^td\varphi(\sigma) z,\sigma,  B(\sigma) D_z\right).
\end{equation}
Using invariance properties with respect to changes of coordinate systems that are precisely described in~\cite{CFM2}, Section~4, one can conclude that $a$ induces an operator $Q_a^X$ on $L^2(N_\sigma X,\C^N)$.
Clearly, $Q_a^X(\sigma,v)$ is smooth and compactly supported in $(\sigma,v)$; moreover, $Q_a^X(\sigma,v)$ is a compact operator on $L^2(N_\sigma X,\C^N)$
for every $(\sigma,v)\in T^*X$.

\begin{proposition}[Proposition~4.2 and~4.4 of \cite{CFM2}]\label{prop:compact}
There exist a sequence $(\eps_\ell)$, a measurable map $t\mapsto  \gamma^t$ valued in the set of non negative (matrix-valued) measures on $T^*\R^d\times {\bf S}^{d-1}$,
 a measurable family of  (scalar) non negative measures $\nu^t$ on $T^*X$ and a measurable map $t\mapsto M^t$, where
\[
M^t:T^*X\ni(\sigma,v)\longmapsto M^t(\sigma,v)\in \mathcal{L}^1(L^2(N_\sigma X,\C^N))
\]
\[
\mbox{and}\qquad\Tr_{L^2(N_\sigma X,\C^N)} M^t(\sigma,v)=1,\quad \nu^t\text{-a.e. }(\sigma,v)\in T^*X
\]
such that, for every $\theta\in{\mathcal C}_0^\infty(\R)$ and $a\in {\mathcal A}^{(2)}$ supported in $\R^d\times V\times \R^d$, one has:
\begin{eqnarray}
\nonumber
\label{eq:nu1nu2}
\int_\R\theta(t) \left\langle W^{X,\eps_\ell}(t),a\right\rangle dt&\Tend{\eps_\ell}{0} &\int _\R\theta(t) \int_{T^*X}\Tr_{L^2(N_\sigma X,\C^N)}(Q_a^X(\sigma,v)M^t(\sigma,v))\nu^t(d\sigma,dv)dt\\
&& + \int_\R \theta(t) \int_{ T^*\R^d\times {\bf S}^{d-1}} \Tr_{\C^{N\times N}}\left(a_\infty (x,\sigma,\omega) \gamma^t(dx,d\sigma,d\omega) \right) dt
 .\end{eqnarray}
\end{proposition}

The family of operators $M^t(\sigma,v)$ describes the part of the concentration that comes from finite distance while the measure $\gamma^t(dx,d\sigma,d\omega)$ is often called the part at infinity of the two-scale semi-classical  measure. In particular, 
$$\gamma^t(x,\xi,\omega) {\bf 1}_{\xi\notin X}= \mu^t (x,\xi) \otimes \delta\left( \omega -\frac{\xi-\sigma_X(\xi)}{|\xi-\sigma_X(\xi)|}\right)
$$
where  $\mu^t$ is a semi-classical measure of $(\Psi^\eps(t))_{\eps>0}$.

\begin{remark}\label{rem:M0}
\begin{enumerate}
\item In a stationary setting, similar objects can be associated with (non time dependent) bounded families in~$L^2(\R^d)$. More precisely, if~$(f^\eps)_{\eps>0}$ is a bounded family in~$L^2(\R^d,\C^N)$,  one can associate with~$(f^\eps)_{\eps>0}$ a pair~$M_0d\nu_0$ defined by the existence of a subsequence~$(\eps_\ell)$ such that for all~$a\in{\mathcal C}_0^\infty(\R^{3d},\C^{N\times N})$,
$$ \left( {\rm op}_{\eps_\ell}(a_{\eps_\ell})  f^{\eps_\ell},f^{\eps_\ell}\right)\Tend{\eps_\ell}{0} \int_{T^*X}\Tr_{L^2(N_\sigma X,\C^N)}(Q_a^X(\sigma,v)M_0(\sigma,v))\nu_0(d\sigma,dv).$$
The initial data in the
 the von Neumann Equation~\eqref{eq:heis} (Theorem~\ref{theorem0})  and in the von Neumann  Equations~\eqref{ed:heis2bis} and~\eqref{ed:heis2} (Theorems~\ref{theo:interaction}) are constructed in that manner, with $N=1$, $X=\Lambda_n$ and $f^\eps=\psi^\eps_n(0)$ for  Equation~\eqref{eq:heis} and with $N=2$, $X=\Sigma_n$ and $f^\eps =\,^t( \psi^\eps_n(0),\psi^\eps_{n+1}(0))\in\C^2$ for Equations~\eqref{ed:heis2bis} and~\eqref{ed:heis2}. 
\item When $X=\{\xi_0\}$, then $\sigma_X(\xi)=\xi-\xi_0$ and one has  for $a\in\mathcal C_0^\infty(\R^{3d})$
 \[
 \left({\rm op}_\eps(a_\eps) f^\eps,f^\eps\right)= \left( {\rm op}_1(a(x,\xi_0+\eps \xi, \xi) {\rm e}^{-\frac i\eps x\cdot \xi_0} f^\eps, {\rm e}^{-\frac i\eps x\cdot \xi_0} f^\eps\right)
 \]
 and the operator ${\rm op}_1(a(x,\xi_0, \xi)$ is compact
 As a consequence, the part at finite distance of any two-microlocal measure associated with the concentration of $(f^\eps)$ on $X$ is a projector $|f^{\xi_0}\rangle \langle f^{\xi_0}|$ where $f^{\xi_0}$ is a weak limit in $L^2$ of the family $({\rm e}^{-\frac i\eps x\cdot \xi_0} f^\eps)$.
\item Note that for determining the part of the concentration that comes from finite distance, it is enough to consider symbols $a$ that are compactly supported in all the variables. 
\end{enumerate}
\end{remark}

\subsection{Two microlocal semi-classical measures for the families $(\psi^\eps_n(t))_{n\in N}$}\label{sec:twomicapplication}

These objects allow to determine the semi-classical measure $\varsigma^t$. Indeed, in~\cite{CFM2}, we have proved that they allow to describe $\varsigma^t$ above critical points of $\varrho_n$ for which the hessian of $\varrho_n$ is of maximal rank on the set of critical points~$\Lambda_n$ (see  assumption~{\bf H2} and~{\bf H2'}). We will use them to prove that $\varsigma^t=0$ above all  crossing sets satisfying~{\bf H3} and to show that  $\varsigma^t$ can be non zero because of modes interactions above degenerate crossing points satisfying~{\bf H3'}. 

\subsubsection{Critical points}

We recall here results from~\cite{CFM2,CFM1}, mainly  Theorem~2.2 in~\cite{CFM2} which gives  a precise description of the measures $\mu^t_{n,n}$ above the set $\Lambda_n$ of critical points of $\varrho_n$ (see~\eqref{def:Lambdan}).
Let~$\Omega$ be an open set of $\R^d$ such that $\Lambda_n\cap \Omega$ is a submanifold.

\begin{theorem}[\cite{CFM2}]\label{theorem43}
Let $(M^t_{n}d\nu^t_{n}, \gamma^t_{n})$ be a pair of two-microlocal semi-classical measures associated with the concentration of $(\psi^\eps_n(t))$ above $\Lambda_n\cap\Omega$. Then, there exists $M_n d\nu_{n}$, a two-microlocal measure associated to the concentration at finite distance of $(\psi^\eps_n(0))$ on $\Lambda_n\cap \Omega$ such that  
$\nu^t_{n}=\nu^0_{n}=\nu_n$, 
$t\mapsto M^t_{n}(\xi,v)$ belongs to the space $\mathcal{C}(\R;\mathcal{L}_+^1(L^2(N_\xi\Lambda_n))$ and solves the von Neumann equation~\eqref{eq:heis}
with
 initial data $M^0_{n}=M_n$.
 Moreover, if the Hessian of $\varrho_n$ is of maximal rank on~$\Lambda_n\cap \Omega$, then $\gamma^t_{n}=0$. 
 \end{theorem}

 \begin{remark}\label{rem:critical}
 \begin{enumerate} 
 \item The maximal rank assumption consists in saying that  
 $${\rm Rank}\, {\rm Hess}\,\varrho_n(\sigma)= {\rm codim}\, \Lambda_n,\;\;\sigma\in \Lambda_n,$$
 or equivalently: 
 $\displaystyle{{\rm Ker} \,  {\rm Hess}\,\varrho_n(\sigma)= T_\sigma \Lambda_n,\;\;\sigma\in \Lambda_n.}$
 \item It is important to notice that the families $(M^t_{n})$ are completely determined by the initial data:
 up to a subsequence  for which one has
$$\left( {\rm op}_{\eps_\ell}(a)  \psi^{\eps_\ell}_n(0), \psi^{\eps_\ell}_n(0) \right)\Tend{\eps_\ell}{0} 
\int_{T^*\Lambda_n}\Tr_{L^2(N_\xi\Lambda)}\left[Q_{a}^{\Lambda_n}(\xi,v)M^0_{n}(\xi,v)\right]\nu^0_{n}(d\xi,dv).$$
\end{enumerate}
\end{remark} 

\subsubsection{Conical crossing points}

When~{\bf H1}, {\bf H2} and~{\bf H3} for all $n\in\N^*$,  the crossing sets 
$\Sigma_n$  are manifolds.
Besides, because of the periodicity of the Bloch energies,  $\Sigma_n$ thus is the union of  connected, closed embedded submanifold of $(\R^d)^*$ and we can focus on each of these connected components by considering
the two-microlocal setting of Section~\ref{sec:twomic} with $N=1$ and the family $(\psi^\eps_n)_{\eps>0}$ for this submanifold.

\begin{theorem}\label{theo:2mic1} 
Assume~{\bf H1}, {\bf H2} and~{\bf H3} holds for some $n\in\N^*$.  Let $\Sigma$ be a connected component of $\Sigma_n$. Then any pair 
$(M^t_{n} d\nu_{n}^t, d\gamma_{n}^t)$ of two-microlocal semi-classical measures associated with the concentration of $(\psi^\eps_n(t))$ on $\Sigma$ satisfy 
 $\nu^t_{n}=0$ and $\gamma^t_{n}=0$.
 Therefore $\mu^t_{n,n} {\bf 1} _\Sigma=0$.
\end{theorem}
The proof of this result is performed in Section~\ref{sec:proof1}.  

\subsubsection{Degenerate crossing points}

We now  suppose that $n$ is fixed and we consider the concentration of $\psi^\eps_n(t)$ and $\psi^\eps_{n+1}(t)$ when the crossing set $\Sigma_n$ involving the two Bloch energies $\varrho_n$ and $\varrho_{n+1}$  satisfies~{\bf H3'}. We consider a connected component $Y$ of $\Sigma_n$ which is assumed to be 
  included  into $\Lambda_n$ and $\Lambda_{n+1}$, the sets of critical points of $\varrho_n$ and $\varrho_{n+1}$ respectively. 
We consider the two-microlocal setting of Section~\ref{sec:twomic} for  $N=2$, the submanifold $Y$  and the family 
$$\Psi^\eps(t)=(\psi^\eps_n(t),\psi^\eps_{n+1}(t))\in\C^2.$$
In view of Lemma~\ref{lem:structurevarrho}, the equation satisfied by $\Psi^\eps$ is 
\begin{equation}\label{eq:Lambda}
i\eps^2 \partial_t \Psi^\eps = \Theta(\eps D) \Psi^\eps+\eps^2  V_{\rm ext} (t,x) \Psi^\eps +\eps^3 F^\eps(t,x)
\end{equation}
with 
$(F^\eps(t))$ uniformly bounded in $L^2(\R^d)$ and 
\begin{align}\label{eq:Theta}
\Theta(\xi)= {\rm Diag} (\varrho_n(\xi), \varrho_{n+1}(\xi))=\lambda_n(\xi) {\rm Id} -  g_n\left(\xi, \xi-\sigma_{Y}(\xi)\right)J,\;\;\;\; 
J=\begin{pmatrix} 1 & 0\\ 0 & -1\end{pmatrix}
\end{align}
where $g_n\in\cC^\infty\left(\sqcup_{\xi\in\Omega} \left(\{\xi\} \times N_{\sigma_{\Sigma_n}(\xi)}\Sigma_n\right)\right)$ (see also~\eqref{def:gn} where the restriction of $g_n$ to  points of $\Sigma_n$ is introduced), and the function  $\lambda_n$ is defined in~\eqref{def:lambda_g}. Note that by assumption~{\bf H3'}, there exists $c>0$ such that we have $g_n(\sigma,\eta)\leq c |\eta|^q$ for all $\sigma\in\Sigma_n$ and $\eta\in N_\sigma\Sigma_n$, and $g_n(\xi,\eta)=|\eta|^2 \theta_n(\xi)$ by (2) of Lemma~\ref{lem:structurevarrho}.

\begin{theorem}\label{theo:2mic2}
We suppose that {\rm\bf H1'}, {\rm\bf H2'} and {\rm\bf H3'} hold. Consider a connected component $Y$ of $\Sigma_n$ that is included in $\Lambda_n\cap\Lambda_{n+1}$.
Let $(M^td\nu^t,d\gamma^t)$ be a pair associated with the concentration of the family~$(\Psi^\eps(t))_{\eps>0}$ on $Y$. Then, $\gamma^t=0$ and there exists $Md\nu^0$ associated with the concentration at finite distance of $(\Psi^\eps(0))_{\eps>0}$ on $Y$ such that $\nu^t=\nu^0$ and the following holds: 
\begin{enumerate}
\item  If
 $q=2$, $M^t$ satisfies~\eqref{ed:heis2}  with initial data $M$. 
\item
If $q>2$,  $M^t$ satisfies~\eqref{ed:heis2bis}  with initial data $M$.
\end{enumerate}
\end{theorem}

\begin{remark}\label{rem:tata45}
Note that, even in Case (2), it can happen that the modes interact above the crossing, if they were doing so at time $t=0$. Corollary~\ref{cor:titi} provides examples of such initial data.
\end{remark}

The proof of Theorem~~\ref{theo:2mic2} is the subject of Section~\ref{sec:proof2}.

%%%%%%%%%%%%%%%%%%%%%%%%%%%%%%%%%%%%%%%%%%%%%%%%%%%%%%%%%

\section{Proof of Theorem~\ref{theo:2mic1}}\label{sec:proof1}

We prove Theorem~\ref{theo:2mic1} in two steps: first we focus on the part of the two-scale semi-classical measure that comes from infinity in Section~\ref{sec:infinity1}, then we concentrate on the part at finite distance in Section~\ref{sec:finite1}. We use the characterization of Lemma~\ref{lem:structurevarrho} and write
\begin{equation}\label{eq:toto9}
\varrho_n(\xi) =\lambda_n(\xi)- 
g_n(\xi,\xi-\sigma_{\Sigma}(\xi)),\;\;
\varrho_{n+1}(\xi) = \lambda_n(\xi)+
g_n(\xi,\xi-\sigma_{\Sigma}(\xi))
\end{equation}
with $\lambda_n$ smooth and $g_n\in\cC^\infty\left(\sqcup_{\xi\in\Omega} \left(\{\xi\} \times N_{\sigma_{\Sigma_n}(\xi)}\Sigma_n\right)\right)$ and $\eta\mapsto g_n(\xi,\eta)$ homogeneous of order~$1$ in~$\eta$ (see (1) in Lemma~\ref{lem:structurevarrho}).
Note that the function introduced in the introduction in~\eqref{def:gn} is the restriction of $g_n$ to $N\Sigma_n$ (and thus have been denoted similarly).

\subsection{The two-scale semiclassical  measures  at infinity}\label{sec:infinity1}

Let $a\in{\mathcal A}^{(2)}$ supported in $\R^d\times W\times \R^d$ where~$W$ is an open subset of $\R^d$ where we have tubular coordinates for $\Sigma$. Let   $\chi\in{\mathcal C}_0^\infty(\R^d)$ such that $\chi=1$ on $B(0,1)$ and $\chi=0$ on $B(0,2)^c$ with $0\leq \chi\leq 1$.
We set for $R,\delta>0$
$$a^{R,\delta}(x,\xi,\eta)= a(x,\xi,\eta) ((1-\chi(\eta/R)) \chi ((\xi-\sigma_{\Sigma}(\xi))/\delta).$$
Then, in view of equation~\eqref{eq:mode_n},
\begin{equation}\label{eq:op(a)}
i\eps {d\over dt} ({\rm op}_\eps(a^{R,\delta}_\eps) \psi^\eps_n(t),\psi^\eps_n(t))=\eps^{-1} \left([{\rm op}_\eps(a^{R,\delta}_\eps) ,\varrho_n(\eps D)] \psi^\eps_n(t),\psi^\eps_n(t)\right) +O(\eps).
\end{equation}
Using~\eqref{eq:toto9}, %with  $g_n(\xi,\eta)=\frac 12  \left(\varrho_{n+1}(\sigma(\xi)+\eta)-\varrho_n(\sigma(\xi)+\eta)\right)$ 
the homogeneity of $g_n$, 
and  the notation introduced in~\eqref{def:geps}, we
write 
$$\varrho_n(\eps D)= \lambda_n(\eps D)-\eps g_n(\eps D,  D -\eps^{-1} \sigma_\Sigma(\eps D)) =  \lambda_n(\eps D) -\eps (g_n)_\eps(\eps D).$$
Therefore, we have 
$$\eps^{-1} \left[{\rm op}_\eps(a^{R,\delta}_\eps) ,\varrho_n(\eps D)\right] = {\rm op}_\eps (\nabla_x a^{R,\delta}_\eps\cdot \nabla \lambda_n)  - \left[{\rm op}_\eps(a^{R,\delta}_\eps) , (g_n)_\eps(\eps D)\right] +O(\eps).$$
We can now apply Lemma~\ref{lem:calculRdelta} with $k=0$,  and we obtain
$$\eps^{-1} \left[{\rm op}_\eps(a^{R,\delta}_\eps) ,\varrho_n(\eps D)\right] = {\rm op}_\eps ( b_\eps )+O(\eps) + O(R^{-1})+O(\delta)$$
with $b = \nabla_x a^{R,\delta}\cdot \nabla \lambda_n -\nabla_x a^{R,\delta}\cdot \nabla_\eta g_n$.
Passing to the limits  $\eps\rightarrow 0$, then $R\rightarrow +\infty$ after time integration against $\theta\in{\mathcal C}_0^\infty(\R)$, we obtain 
by~\eqref{eq:op(a)}, 
$$\int_{\R}\theta(t) \left( {\rm op}_\eps (b_\eps) \psi^\eps_n(t),\psi^\eps_n(t) \right)  dt= O(\eps) + O(R^{-1})+O(\delta).$$
We deduce 
$$\int_{\R\times \R^d\times \Sigma\times{\bf S}^{d-1}} \theta(t) (\nabla \lambda_n(\sigma) -\nabla _\eta g_n(\sigma,\omega) )\cdot \nabla_xa_\infty(x,\sigma,\omega)  d\gamma^t_{n}(x,\sigma,\omega)=0.$$
This implies that the measure $\gamma^t _{n}(x,\sigma,\omega)$ is invariant by the flow 
$$(x,\sigma,\omega)\mapsto (x+s(\nabla \lambda_n(\sigma)-\nabla _\eta g_n(\sigma,\omega) ,\sigma,\omega),\;\;s\in\R.$$ 
As a consequence, $\gamma^t_{n}$ is supported on $\{\nabla \lambda_n(\sigma)-\nabla _\eta g_n(\sigma,\omega)=0\}$, and by~\textbf{H3}, 
 $\gamma_{n}^t=0$.

\subsection{The two-scaled semiclassical  measures  coming from  finite distance}\label{sec:finite1}
In view of Remark~\ref{rem:M0} (3), we now choose  $a\in{\mathcal C}_0^\infty(\R^{d}\times W\times \R^d)$ where $W$ is as above. Let  $\theta\in L^1(\R)$.
Arguing as in~\eqref{eq:op(a)}, we observe
$$\int_\R \theta(t) \left([{\rm op}_\eps(a_\eps), \eps^{-1} \varrho_n(\eps D_x) ] \psi^\eps_n(t) ,\psi^\eps_n(t)\right)=O(\eps).$$
Using that $a$ is compactly supported in the variable $\eta$ and the homogeneity of $g$, we obtain in $\mathcal L(L^2(\R^d))$,
$$\frac 1\eps [{\rm op}_\eps(a_\eps),  \varrho_n(\eps D_x) ] =i {\rm op}_\eps (\nabla \lambda_n(\xi)\cdot \nabla_x a_\eps)  -[{\rm op}_\eps(a_\eps), (g_n)_\eps(\eps D) ] +O(\eps) .$$
Passing to the limit $\eps\rightarrow 0$ thanks to Lemma~\ref{lem:gepsfinite}, we obtain 
\begin{equation}\label{eq:ovm1}
\int_\R \theta(t) {\rm Tr}_{L^2(N_\sigma \Sigma)} ( i Q^{\Sigma}_{\nabla \lambda_n \cdot \nabla_x a}-[Q_a^{\Sigma}(\sigma,v), Q_g^{\Sigma}(\sigma)] M^t_n(\sigma,v))\nu^t_n(\sigma,v) dt=0.
\end{equation}
This relation has important consequences on the structure of $M^t_{n,n}$ and $\nu^t_n$. For stating them, we write
$$\nabla \lambda_n(\sigma)=\nabla^\perp \lambda_n(\sigma)+\nabla^\sharp \lambda_n(\sigma),\;\;\nabla^\perp \lambda_n(\sigma)\in N_\sigma \Sigma\;\;\mbox{ and}\;\; 
\nabla^\sharp \lambda_n(\sigma)\in T_\sigma \Sigma.$$

\begin{lemma}\label{lem:toto59}
Equation~\eqref{eq:ovm1} implies 
$${\rm supp} (\nu^t_{n}) \subset \{ (\sigma,v)\in T{\Sigma},\;\; \nabla^\sharp \lambda_n(\sigma )=0\}$$
and
$$ [Q_F^\Sigma( \sigma), M^t_{n}(\sigma,v) ]=0\;\;d\nu^t_{n} \, a.e. \, (\sigma,v)\in N\Sigma$$
where 
$F(\sigma,\eta)=\nabla^\perp \lambda_n(\sigma) \cdot \,  \eta + g_n(\sigma, \eta).$
\end{lemma}

\begin{proof}[Proof of Lemma~\ref{lem:toto59}]
We use  a system of equations $\varphi(\xi)=0$ of $\Sigma$ and the matrix $B$ defined in~\eqref{def:B}.
For $\sigma\in X$ and $\zeta\in T_\sigma\R^d$, we have
$$({\rm Id} -d\sigma_{\Sigma}(\sigma))\zeta= B(\sigma) d\varphi(\sigma)\zeta,$$
which allows to decompose $\zeta$ as 
$$\zeta = d\sigma_{\Sigma}(\sigma)\zeta+ B(\sigma) d\varphi(\sigma)\zeta,\;\; d\sigma_{\Sigma}(\sigma)\zeta\in T_\sigma {\Sigma}\;\;\mbox{and} 
\;\;B(\sigma) d\varphi(\sigma)\zeta\in N_\sigma {\Sigma}.$$
In particular, $B(\sigma) d\varphi(\sigma)={\rm Id} $ on $N_\sigma {\Sigma}$. 
In view of this observation, 
we 
 write for $(\sigma,v)\in T\Sigma$, and $(z,\zeta) \in (N_\sigma {\Sigma})^*$ 
\begin{align*}
\nabla \lambda_n(\sigma)\cdot  \nabla_x a(v+ \, ^t d\varphi(\sigma) z,\sigma , B(\sigma) \zeta) 
&=\nabla ^\sharp \lambda_n(\sigma) \cdot \nabla_v a(v+ \, ^t d\varphi(\sigma) z,\sigma , B(\sigma) \zeta)\\
&\qquad
 +\nabla^\perp  \lambda_n(\sigma) \cdot (B( \sigma)\nabla_z)\cdot \left(a(v+ \, ^t d\varphi(\sigma) z,\sigma , B(\sigma) \zeta) \right)
\end{align*}
 We obtain 
$$Q^\varphi_{\nabla \lambda_n(\xi)\cdot \nabla_x a}(\sigma,v) = 
\nabla^\sharp \lambda_n(\sigma) \cdot \nabla_v Q_a^\varphi (\sigma, v )
-i [ Q_F^\varphi(\sigma), Q_a ^\varphi(\sigma,v)],$$
with $Q_F^\varphi(\sigma)= F(\sigma, B(\sigma)D_z)$. 
Therefore,  
equation~\eqref{eq:ovm1}  
writes
$$\int_\R \theta(t) {\rm Tr}_{L^2(N_\sigma \Sigma)} ( i \nabla^\sharp \lambda_n(\sigma) \cdot \nabla_v Q^\Sigma_a (\sigma, v )+[Q_a^{\Sigma}(\sigma,v), Q_F^\Sigma(\sigma)] M^t_n(\sigma,v))\nu^t_n(\sigma,v) dt=0,$$
We deduce 
$$ i \nabla^\sharp \lambda_n (\sigma) \cdot \nabla _v (M^t_{n} d\nu_{n}^t) + [Q^\Sigma_F(\sigma) , M^t_{n} d\nu^t_{n}] =0 .$$ 
Taking the trace, it gives 
$$\nabla^\sharp \lambda_n(\sigma ) \cdot \nabla_v \nu^t_{n}=0,$$
whence the invariance of $\nu^t_n$ by the flow defined on $T\Sigma$ by 
$$(\sigma,v)\mapsto (\sigma, v+s\nabla^\sharp \lambda_n(\sigma )),\;\;s\in\R,$$
which implies the results.
\end{proof}

  We conclude the analysis of the two-scaled Wigner measures at finite distance $M^t$ 
 by using Lemma~\ref{lem:comm} below. For this, we need to check that its assumptions are satisfied.  Hypothesis~\textbf{H3} implies that if $\nabla^\sharp \lambda_n(\sigma )=0$, then for all $\eta\in N_\sigma {\Sigma}\setminus\{0\}$,
$$ \nabla^\perp \lambda_n(\sigma)-\nabla_\eta g_n(\sigma,\eta) \not=0.$$
 Considering $\nabla_\zeta (F(\sigma,B(\sigma)\zeta))$, we have 
$$\nabla_\zeta(F(\sigma,B(\sigma)\zeta)) =\, ^t B(\sigma)\left( \nabla^\perp \lambda_n(\sigma)-\partial_\eta g_n(\sigma,\eta)  \right)\not=0,$$
because $B(\sigma)$ is invertible on~$N_\sigma \Sigma$, 
and the assumptions of the next lemma are satisfied.

\begin{lemma}\label{lem:comm}
Let $p\in\N$ and ${ M}$ be a non negative trace-class operator on $L^2(\R^p)$, and $F\in \mathcal C^\infty(\R^p\setminus\{0\})$ such that $\nabla_\zeta F(\zeta)\not=0$ for all $\zeta\in\R^p\setminus \{0\}$. 
Assume 
$\left[ F(D_z) , {M} \, \right]=0.$
Then ${ M}=0$.
\end{lemma}

\begin{proof}
Let $\phi\in L^2(\R^p)$ be an eigenvector of ${M}$ for an eigenvalue $\ell\not=0$. Then, for all $j\in \N$, 
$$\phi_j:=(F( D_z))^j \phi$$
 is also an eigenvector for $\ell$. Since $\ell$ is of finite multiplicity because $ M$ is trace-class, we deduce that the set $\{ \phi_j, j\in\N\}$ is of finite dimension. Let $k\in\N^*$ the first index such that the family $(\phi_j)_{0\leq j\leq k}$ is not a family of independent vectors.  Then, there exist  $\alpha_0,\cdots \alpha_k\in \R$ non all equal to~$0$, and such that 
$\displaystyle{\sum_{j=0}^k \alpha_j \phi_j=0.}$
In Fourier variables, we obtain 
$\displaystyle{\left(\sum_{j=0}^k \alpha_j F( \zeta)^j\right) \widehat \phi(\zeta)=0.}$
The set 
$\displaystyle{{\mathcal C} =\left\{ \zeta\in\R^p,\;\; \sum_{j=0}^k \alpha_j F(\zeta)^j =0\right\}}$ is the union of a finite number of sets ${\mathcal C}_\beta$,
$$\mathcal C_\beta= \{ F(\zeta)=\beta\}$$
for $\beta$ a real-valued root of the polynomial $\sum_{0\leq j\leq k} \alpha_j X^j$. Since $\nabla_\zeta F(\zeta)\not=0$ for all $\zeta\not=0$, these sets $\mathcal C_\beta$ are hypersurfaces of $\R^p$ and thus of Lebesgues measure~$0$. So it is for ${\mathcal C}$ and we deduce that $\phi=0$. 
\end{proof}

%%%%%%%%%%%%%%%%%%%%%%%%%%%%%%%%%%%%%%%%%%%%%%%%%

\section{Proof of Theorem~\ref{theo:2mic2}}\label{sec:proof2}

Theorem~\ref{theo:2mic2} contains two statements. First, it states  that the two-scale  semi-classical measures at infinity is $0$, what we prove in Section~\ref{sec:infinity2} below, by showing invariance properties of its diagonal elements. Secondly, it gives transport equations that allow to compute the two-scale semi-classical measures coming from finite distance from the knowledge of the  initial data. We focus on this latter point in Section~\ref{sec:finitedist2}.

 \subsection{Analysis at infinity} \label{sec:infinity2}
 We perform the proof for $q=2$, the proof for $q>2$ is similar. 
 Let $\Psi^\eps$ be a family of solutions to equation~\eqref{eq:Lambda}. 
 Let $a\in{\mathcal A}^{(2)}$ supported in $\R^d\times W\times \R^d$ where~$W$ is an open subset of $\R^d$ where we have tubular coordinates for the manifold $Y$. Let   $\chi\in{\mathcal C}_0^\infty(\R^d)$ such that $\chi=1$ on $B(0,1)$ and $\chi=0$ on $B(0,2)^c$ with $0\leq \chi\leq 1$.
We set for $R,\delta>0$
$$a^{R,\delta}(x,\xi,\eta)= a(x,\xi,\eta) ((1-\chi(\eta/R)) \chi ((\xi-\sigma_{Y}(\xi))/\delta)$$
and we consider the symbol 
$$\tilde a^{R,\delta}(x,\xi,\eta)=|\xi-\sigma_Y(\xi)|^{-1} a^{R,\delta}(x,\xi,\eta).$$
By Lemma~\ref{lem:singulier}~(1) (see also  Appendix~\ref{app:twomic}), there exists a constant $C>0$ such that  
$$\| {\rm op} _\eps(\tilde a^{R,\delta}_\eps)\|_{\mathcal L(L^2(\R^d))}\leq C (\eps R)^{-1}.$$
 In view of~\eqref{eq:Theta} and of (2) in Lemma~\ref{lem:structurevarrho}, we write with the notations introduced in~\eqref{def:geps}
 $$\Theta(\eps D)= \lambda_n(\eps D){\rm Id} - \eps^2 (g_n)_\eps (\eps D).$$
 Therefore,  if $E$ is a constant diagonal matrix of $\C^{2\times 2}$,  we obtain 
$$
  \left[ {\rm op}_\eps(\tilde a^{R,\delta}_{\eps} E ) \;,\; \Theta(\eps D) \right]  =
 \left[ {\rm op}_\eps(\tilde a^{R,\delta}_{\eps}) \;,\;  \lambda_n(\eps D)\right]   E -
\eps^2  \left[ {\rm op}_\eps(\tilde a^{R,\delta}_{\eps}) \;,\; (g_n)_\eps (\eps D)\right] EJ,
 $$
 where we have used that $EJ=JE$.
 We observe that setting 
 $$b(x,\xi,\eta)= |\eta|^{-1} a^{R,\delta} (x,\xi,\eta),$$
 we have 
 $$\eps\, {\rm op}_\eps(\tilde a^{R,\delta}_{\eps})
 ={\rm op}_\eps(b_{\eps})$$
 and we can apply Lemma~\ref{lem:calculRdelta} because $b\in\mathcal A^{(2)}_{-1}$ and $g_n\in\mathcal H_2$. We deduce 
$$ {\eps}  \left[ {\rm op}_\eps(\tilde a^{R,\delta}_{\eps}) \;,\; (g_n)_\eps (\eps D)\right] =
  \left[ {\rm op}_\eps(b_{\eps}) \;,\; (g_n)_\eps (\eps D)\right] = {\rm op}_\eps ((\nabla_x b \cdot \nabla _\eta g)_\eps)$$
  with $\nabla_x b(x,\xi,\eta)= |\eta|^{-1} \nabla_x a^{R,\delta} (x,\xi,\eta).$
 Therefore, we are left with
$$
\frac 1 \eps  \left[ {\rm op}_\eps(\tilde a^{R,\delta}_{\eps}E) \;,\; \Theta(\eps D) \right]   =
{\rm op}_\eps (\nabla_x \tilde a^{R,\delta}_{\eps}\cdot\nabla_\xi \lambda_n(\xi)  ) E - 
{\rm op}_\eps ((\nabla_x b \cdot \nabla _\eta g_n)_\eps) EJ + O(\eps)+O(R^{-1})+O(\delta).
$$
We use $\nabla \lambda_n(\xi)={\rm Hess}\, \lambda_n(\sigma_Y(\xi) )(\xi-\sigma_Y(\xi)) +O((\xi-\sigma_Y(\xi))^2)$ and we set 
$$
c(x,\xi,\eta) :=\nabla_x  a^{R,\delta}(x,\xi,\eta)\cdot {\rm Hess} \,  \lambda_n(\sigma_Y(\xi)) \frac{\eta}{|\eta|} E -
\nabla_x  a^{R,\delta}(x,\xi,\eta) \cdot \frac {1} {|\eta|} \nabla_\eta g_n\left( \xi,\eta\right) EJ.
$$
Note that  $c\in{\mathcal A}^{(2)}$ and 
$$
\frac 1 \eps  \left[ {\rm op}_\eps(\tilde a^{R,\delta}_{\eps} E) \;,\; \Theta(\eps D) \right]
= {\rm op}_\eps(c_\eps) + O(\eps)+O(R^{-1})+O(\delta).
$$
Therefore, passing to the limit $\eps$ to $0$, then $R$ to $+\infty$ and finally $\delta$ to~$0$, we obtain for all
$\theta\in{\mathcal C}_0^\infty(\R)$,
$$\int_{\R} \theta(t)\int_{\R^d\times Y\times {\bf S}^{d-1}} {\rm Tr}_{\C^{2\times 2}}\left( \nabla _xa_\infty(x, \sigma,\omega)\cdot 
({\rm Hess}\,  \lambda_n(\sigma) \omega \, E - \nabla_\eta g_n(\sigma, \omega) EJ)\right) \gamma^t(dx,d\sigma, d\omega)dt
=0.$$
Let us denote by  $\gamma_{n,n}^t$ and $\gamma_{n+1,n+1}^t$ the diagonal coefficients of the matrix-valued measure $\gamma^t$. Choosing the~$2\times 2$ diagonal matrix~$E$ such that~$EJ=E$, we deduce that the measure  $\gamma_{n,n}^t$ is  invariant by the flow 
$$(x,\sigma,\omega) \mapsto  (x+s({\rm Hess}\,\lambda_n(\sigma) \omega -\nabla_\eta g_n(\sigma,\omega) ) ,\sigma,\omega),\;\; s\in\R.$$
Then, choosing $E$ such that  $EJ=-E$, we obtain that 
the measure  $\gamma_{n+1,n+1}^t$ is  invariant by the flow 
$$(x,\sigma,\omega) \mapsto  \left(x+s({\rm Hess}\,\lambda_n(\sigma) \omega+\nabla_\eta g_n(\sigma, \omega)),\sigma, \omega\right),\;\; s\in\R.$$
From assumption~\textbf{H3'}, we deduce $\gamma^t_{n,n}=0$ and $\gamma^t_{n+1,n+1}=0$, and the positivity of $\gamma^t$  implies that $\gamma^t=0$.  One argues similarly when $q>2$, and proves that the term  in $g_n$ does not contribute to the limit.

\subsection{The two-scale semiclassical  measures  coming from  finite distance}\label{sec:finitedist2}
Here again, we write the proof for $q=2$. 
We choose $\theta\in L^1(\R^d)$, $a\in{\mathcal C}_0^\infty(\R^{d}\times W\times \R^d,\C^{2\times 2} )$ where $W$ is a tubular neighbrohood of $Y$ where the function $\sigma_Y$ is defined. Using the homogeneity of the function $g(\xi,\eta)$, we have 
\begin{equation}
\label{eq:toto37}
i\frac d{dt}
\left({\rm op}_\eps(a_\eps) \Psi^\eps(t),\Psi^\eps(t)\right) = I^\eps_1(t)+I^\eps_2(t)
\end{equation}
with
\begin{align*}
I^\eps_1(t)&= \left([ {\rm op}_\eps(a_\eps),\eps^{-2} \lambda_n(\eps D_x) +  V_{\rm ext}(t,x) ] \Psi^\eps(t) ,\Psi^\eps(t)\right)\\
I^\eps_2 (t)&=- 
\left( \left( {\rm op}_\eps(a_\eps) J (g_n)_\eps (\eps D)- (g_n)_\eps(\eps D)J {\rm op}_\eps(a_\eps) \right) \Psi^\eps(t) ,\Psi^\eps(t)\right).
\end{align*} 
Note that if $q>2$, the homogeneity implies $I^\eps_2 (t)= O(\eps^{q-2})$. 
 Section~5.1 in~\cite{CFM2}  gives
the uniform boundedness of the family of time dependent functions
 $t\mapsto I^\eps_1(t)$ and 
 Lemma~\ref{lem:gepsfinite}  yields the uniform boundedness of the family of time dependent functions $t\mapsto I^\eps_2(t)$. Therefore,  the left-hand side of~\eqref{eq:toto37}
 is uniformly bounded with respect to $\eps$. Therefore, the maps $t\mapsto M^t(\sigma, v) d\nu^t(\sigma,v)$ defined on $TY$  will be continuous in time.

\begin{remark}\label{rem:tata89}
At that level of the proof, one sees that by Ascoli theorem, one can find for each $T>0$ a sequence $\eps_\ell$ for which  the limit of $\left({\rm op}_{\eps_\ell}(a_{\eps_\ell}) \Psi^{\eps_\ell}(t),\Psi^{\eps_\ell}(t)\right)$ exists for all $t\in \R$. One then deduces the convergence $t$ by $t$ of these quantities. 
\end{remark}

We now  integrate equation~\eqref{eq:toto37} against a function $\theta$ and pass to the limit $\eps\rightarrow 0$. By Section~5.1 in~\cite{CFM2},
  we have for $\theta\in \mathcal C_0^\infty(\R)$, up to the subsequence defining $M^td\nu^t$
\begin{align*}
&\int\theta(t) I^\eps_1(t) dt \Tend {\eps}{0}
\\
&\;\int_\R \theta(t)  \int_{TY} {\rm Tr}_{L^2(N_\sigma Y,\C^2)} \Biggl(
 \left[ Q^Y_a(\sigma, v) , \frac 12 {\rm Hess} \lambda_n (\sigma)\,D_z\cdot D_z+m_{V_{\rm ext} (t,\cdot )}(x,v)\right]  M^t(\sigma,v) \Biggr)\nu^t(\sigma,v)dt.
\end{align*}
Besides, by
 Lemma~\ref{lem:gepsfinite} for studying the term~$I^\eps_2$.  
\begin{align*}
&\int\theta(t) I^\eps_2(t) dt \Tend {\eps}{0}\\
&\;
\int_\R \theta(t)  \int_{TY} {\rm Tr}_{L^2(N_\sigma Y,\C^2)}
 \left(  Q^Y_a(\sigma,v) J Q^Y_{g_n}(\sigma)-Q^Y_{g_n}(\sigma)JQ_a^Y(\sigma, v)\right)
 M^t(\sigma,v) \Biggr)\nu^t(\sigma,v)dt.\\
 &\qquad = \int_\R \theta(t)  \int_{TY} {\rm Tr}_{L^2(N_\sigma Y,\C^2)}
 \left(  \left[ Q^Y_a(\sigma,v) \;,\; J Q^Y_{g_n}(\sigma)\right] 
 M^t(\sigma,v) \right)\nu^t(\sigma,v)dt
\end{align*}
Reporting the result in~\eqref{eq:toto37}, we obtain 
\begin{align*}
-i\int_\R \theta'(t) &\int_{TY} {\rm Tr}_{L^2(N_\sigma Y,\C^2)} (Q^Y_a(\sigma,v)M^t(\sigma,v)) d\nu^t(\sigma,v)
 dt = \int_\R \theta(t)  \int_{TY} 
 {\rm Tr}_{L^2(N_\sigma Y,\C^2)} \Biggl(\\
& \biggl[ Q^Y_a(\sigma, v) 
  ,\left( \frac 12 {\rm Hess} \lambda_n (\sigma)\,D_z\cdot D_z+m_{V_{\rm ext} (t,\cdot )}(x,v)\biggr){\rm Id} 
 + J Q^Y_{g_n}(\sigma)
 \right]
 M^t(\sigma,v) \Biggr)\nu^t(\sigma,v)dt.
\end{align*}
We deduce 
$$\partial_t (M^t d\nu^t)=  \left[ \left( \frac 12 {\rm Hess} \lambda_n (\sigma)\,D_z\cdot D_z+m_{V_{\rm ext} (t,\cdot )}(x,v)\right)\,{\rm Id}
+ Q^Y_{g_n}(\sigma) J
\;,\; M^t d\nu^t\right]
 d\nu^t.$$
Taking the trace of this expression gives $\partial_t\nu^t=0$, whence $\nu^t=\nu^0$ (because of the continuity of $t\mapsto \nu^t$),  and the equation satisfied by $M^t$.

%%%%%%%%%%%%%%%%%%%%%%%%%%%%%%%%%%%%%%%%%%%%%%%%%

\section{Proof of the main Theorems}\label{sec:proofs}

\subsection{Proofs of Theorem~\ref{theorem0} and Proposition~\ref{prop:special_case}}
We prove here the results obtained under {\bf H1}, {\bf H2} and {\bf H3}, which corresponds to a general setting without too much assumptions on the initial data and with generic hypothesis on the Bloch energies. 

\begin{proof}[Proof of Theorem~\ref{theorem0}]
Let $(\eps_\ell)$ be a sequence given by Proposition \ref{prop:sum} and $\varsigma^t$ and $\mu^t_{n,n'}$ the corresponding semi-classical measures along that sequence. Because of  the assumption {\bf H1} and  Part (2) of Proposition~\ref{prop:easysupp}, $\mu^t_{n,n'}=0$ for a.e. $t\in\R$ as soon as $|n-n'|>1$. Besides, by~{\bf H2}, we can use Theorem~\ref{theorem43} to determine $\mu^t_{n,n}{\bf 1}_{\Lambda_n}$. Finally,  by~{\bf H3} and Theorem~\ref{theo:2mic1}, $\mu_{n,n}^t= \mu^t_{n,n}{\bf 1}_{\Lambda_n}$ and the result follows: we obtain that for a subsequence $\eps_\ell$
 for $a,b\in\R$ and $\varphi\in\mathcal C_0^\infty(\R^{d})$,
 $$\int_a^b\int _{\R^d} \phi(x) |  \psi^{\eps_\ell}(t,x)|^2dx dt\Tend{\eps_\ell}{0} \sum_{n\in\N_0}
\int_a^b \int_{T^*X}\Tr_{L^2(N_\sigma \Lambda_n}(m_\phi^{\Lambda_n}(\sigma)M_n^t(\sigma,v))\nu_n(d\sigma,dv)dt$$
once observed that for $a(x,\xi):= \phi(x)$, the operator $Q_a(\sigma,v)$ coincides with $m_\phi^{\Lambda_n}(\sigma)$. 
 \end{proof}
 
 \begin {proof}[Proof of Proposition~\ref{prop:special_case}]
 For the data considered in that statement, one has $U^\eps_0= \varphi_{n_0}(y,\eps D_x) u^\eps_{n_0}$. Therefore, $M_n=0$ for $n\not=n_0$ and $M_n^t$ too. We then focus on calculating $M_{n_0}$ above any $\xi\in \Lambda_{n_0}$. By Corollary~\ref{cor:titi} (1), since $\xi$ is an isolated point of $\Lambda_{n_0}$,   the measure $\nu_{n_0}$ is given by 
 \[
 \nu_{n_0}(d\xi)= \| v_{n_0}\|_{L^2}^2
 \sum_{j\in 2\pi \Z^d} |c_{n_0} (\xi_{n_0}+ j)|^2 \delta(\xi- \xi_{n_0}-j)
 \] and, for $j\in 2\pi\Z^d$, the operator~$M_{n_0}(\xi_{n_0} +j)$ is the projector on $\C v_{n_0}$. As a consequence,  
 the solution~$M^t(\xi_{n_0} +j)$ of~\eqref{eq:heis} is the orthogonal projection on~$\C \psi^{\xi_{n_0}}(t)$ where $\psi^{\xi_{n_0}}$ satisfies~\eqref{eq:tata}. We obtain that for any $\phi\in\mathcal C_0^\infty (\R^d)$,
 \begin{align*}
 \int _{\Lambda_{n_0}} {\rm Tr}_{L^2(\R^d)}m_\phi^{\Lambda_{n_0}} (M^t_{n_0}(\xi))\nu_{n_0} (d\xi) 
 &= \left(\int_{\R^d} \phi(x) | \psi^{\xi_{n_0}}(t,x)|^2 dx \right)\sum_{j\in\Z^d} |c_{n_0} (\xi_{n_0}+ j)  |^2\\
 &=\int_{\R^d} \phi(x)| \psi^{\xi_{n_0}}(t,x)|^2dx 
  \end{align*}
 where we have used that $N\Lambda_{n_0}=\R^d$, $m^{\Lambda_{n_0}}_\phi$ is the operator of multiplication by $\phi$, and the  conservation of $L^2$ norms for the equation~\eqref{eq:tata} ($\| v_{n_0}\|_{L^2}=\| \psi^{\xi_{n_0}}(t)\|_{L^2}$ for all $t\in\R$).
 \end{proof}

\subsection{Proof of Theorem~\ref{theo:interaction} and Proposition~\ref{prop:special_cas2}} We focus here on degenerate crossings involving two energies isolated from the remainder of the spectrum and well-prepared data that concentrate on these modes.  We  will indeed prove a more general result than Theorem~\ref{theo:interaction}, assuming that $\Sigma_n$ is  included in $\Lambda_n\cup \Lambda_{n+1}$ but not necessarily equal to $\Lambda_n$ and $\Lambda_{n+1}$ (the latter being non necessarily equal). Thus, one has to take into account the additional contributions to the energy densities generated by the points of $(\Lambda_n\cup \Lambda_{n+1})\setminus \Sigma_n$ and Theorem~\ref{theo:interaction} is a straightforward corollary of the next result. 

\begin{theorem}\label{theo:general2cross'}
Then, there exists a subsequence $\eps_\ell\Tend{\ell}{+\infty} 0$, three non negative measures $\nu_n\in\mathcal M^+(T^*\Lambda_n) $, $\nu_{n+1}\in\mathcal M^+(T^*\Lambda_{n+1})$ and $\nu^0\in\mathcal M^+(\Sigma_n)$ depending on  $(\psi^{\eps_\ell}_0)$,  three measurable trace-class operators $M_n$, $M_{n+1}$ and  $M$
\begin{align*}
M_n:\,&T^*_\xi\Lambda_n\ni (\xi,v)  \mapsto M_n (\xi,v)  \in\mathcal L^1_+(L^2(N_\xi \Lambda_n)),\;\;{\rm Tr} _{L^2(N_\xi \Lambda_n)}M_n (\xi,v)  =1\; d\nu_n \; a.e.\\
M_{n+1}:\,&T^*_\xi\Lambda_{n+1}\ni (\xi,v) \mapsto M_{n+1}(\xi,v)\in\mathcal L^1_+(L^2(N_\xi \Lambda_{n+1})),\;\;{\rm Tr} _{L^2(N_\xi \Lambda_{n+1})}M_{n+1} (\xi,v) =1\; d\nu_{n+1} \; a.e. \\
M:\,&T^*_\xi\Sigma_n\ni (\xi,v) \mapsto M(\xi,v) 
\in\mathcal L^1_+(L^2(N_\xi \Sigma_n,\C^2)),\;\;{\rm Tr} _{L^2(N_\xi \Sigma_n,\C^2)}M (\xi,v)  =1\; d\nu^0\; a.e.
\end{align*}
such that for every $a<b$ and every $\phi\in{\mathcal C}_0(\R^d)$ one has
\begin{align*}
\lim_{\ell\rightarrow +\infty}& \int_a^b\int_{\R^d}  \phi(x) |\psi^{\eps_\ell} (t,x)|^2 dx dt
  \\
& =\sum_{j=n,n+1} \int_a^b \int_{T^*(\Lambda_j\setminus\Sigma_n)}{\rm Tr} _{L^2(N_\xi \Lambda_j)}[m^{\Sigma_n}_\phi(\xi,v)M^t_j (\xi,v) ] \nu_j(d\xi,dv)dt\\
& + \int_a^b \int_{T^*\Sigma_n}{\rm Tr} _{L^2(N_\xi \Sigma_n, \C)}
[m^{\Sigma_n}_\phi(\xi,v)
(m_n^t + m_{n+1}^t +2 {\rm Re} (m_{n,n+1}^t ))(\xi,v) ] \nu^0(d\xi, dv)dt ,
 \end{align*}
 where 
  \begin{equation*}
M^t(\xi,v) = \begin{pmatrix} m_n^t(\xi,v) & m_{n,n+1}^t(\xi,v) \\ m_{n,n+1}^t(\xi,v)^* & m_{n+1}^t(\xi,v)\end{pmatrix}
\end{equation*}
is a non negative trace class operators on $L^2( N_\xi \Sigma_n,\C^2) $. 

Besides,
 the map $t\mapsto M^t_n(x,\xi)\in\mathcal C(\R, \mathcal L^1_+(L^2(N_\xi \Lambda_n))$ solves the von Neumann equation~\eqref{eq:heis} and similarly for $M^t_{n+1}$ and $\varrho_{n+1}$, 
 and the map
 $t\mapsto M^t (\xi,v)\in\mathcal C(\R, \mathcal L^1_+(L^2(N_\xi \Sigma_n,\C^2))$ 
 solves~\eqref{ed:heis2bis} if~$q>2$  and~\eqref{ed:heis2} if $q=2$. All the 
 initial data  depend on $(\psi^{\eps_\ell}_0)$ as in Remark~\ref{rem:M0}.
\end{theorem}

\begin{proof} 
We have
$$U^\eps_0(x,y)=\varphi_n(y,\eps D_x) u^\eps_{n}(x)+ \varphi_{n+1}(y,\eps D_x) u^\eps_{n+1}(x)$$
and we are going to take advantage of the fact that $U^\eps_0\in {\rm Ran } \,\Pi(\xi)$, the spectral projector on
$${\rm Ker}(P(\xi)-\varrho_n(\xi)) \oplus  {\rm Ker}(P(\xi)-\varrho_{n+1}(\xi)). $$
By assumption~{\bf H1'}, the band of the spectrum of $P(\xi)$ consisting of the pair $\{ \varrho_n(\xi), \varrho_{n+1}(\xi)\}$ is separated from the remainder of the spectrum by a gap, which implies that $\xi\mapsto \Pi(\xi)$ is analytic. We claim that a  consequence of this is that 
 if $(\eps_\ell)$ is a sequence given by Proposition \ref{prop:sum}, then 
\begin{equation}\label{step90}
\varsigma^t = \mu^t_{n,n}+\mu^t_{n+1,n+1} + \mu^t_{n,n+1} +\mu^t_{n+1,n}.
\end{equation}
By Proposition~\ref{prop:easysupp}, $\varsigma^t$ has only support above $\Lambda_n$ (because of $\mu^t_{n,n}$), $\Lambda_{n+1}$ (because of $\mu^t_{n,n+1}$)  and $\Sigma_n$ (because of the crossed terms).
Then, the result of Theorem~\ref{theo:interaction} comes from two observations: 
\begin{enumerate}
\item assumption {\rm\bf H2'}  allows to use Theorem~\ref{theorem43} to determine $\mu^t_{n,n}$ above $\Lambda_n\setminus \Sigma_n$ and  $\mu^t_{n+1,n+1}$ above  $\Lambda_{n+1}\setminus \Sigma_{n+1}$,
\item  assumption {\rm\bf H3'}  allows to use Theorem~\ref{theo:2mic2} to compute $\mu^t_{n,n}$, $\mu^t_{n+1,n+1}$, $\mu^t_{n,n+1}$ and $\mu^t_{n+1,n}$ above $\Sigma_n$  in terms of the coefficients of the matrix-valued measure $M^t d\nu^0$. In view of~\eqref{step90}, we obtain that for all $\phi\in\mathcal C_0^\infty(\R^d)$,
\begin{equation}\label{step91}
\int _{\R^d\times \Sigma_n} \phi(x) \varsigma^t(dx,d\xi) = 
\int _{T^*\Sigma_n} {\rm Tr} \left[ m_\phi^{\Sigma_n} (\xi,v) E M^t(\xi, v)\right] \nu^0  (d\xi,dv),\;\; E=\begin{pmatrix} 1 & 1 \\ 1 & 1  \end{pmatrix}
\end{equation}
\end{enumerate} 
It remains to discuss Equation~\eqref{step90}, which comes  from~\eqref{def:Leps} and the estimate 
$$\left\| (1-\Pi (\eps D_x)) U^\eps(t)\right\|_{H^s_\eps (\R^d\times \T^d)}\leq \eps \, C_s (1+|t|)$$
for $s>d/2$.
We observe that the family  
 $$W^\eps (t,x)= (1-\Pi)(\eps D_x) U^\eps(t,x).$$
satisfies the system 
$$i\eps^2\partial_t W^\eps =P(\eps D_x) W^\eps+\eps^2 V_{\rm ext} W^\eps +\eps^3 G^\eps,\;\; W^\eps(0)=0$$
with $G^\eps(t) =- \eps^{-1} \left[\Pi(\eps D_x), V_{\rm ext}(t)\right] U^\eps(t)$  uniformly bounded in $L^2(\R^d\times \T^d)$. 
Therefore, when $s=0$, the estimate comes from an energy argument.  We then proceed as in Lemma~6.7 in~\cite{CFM2} by induction in $s\in\N$ and   interpolation between $s$ and $s+1$,  observing that, in view of  Remark~\ref{rem:eqnorm}, it is enough to prove that 
$P(\eps D_x)^{s/2}W^\eps$ and $ \langle \eps D_x \rangle ^s W^\eps$ go to~$0$ in $L^2(\R^d\times \T^d)$.
\end{proof}

As a by-product of the proofs, we have the following remark:

\begin{remark}\label{rem:tata99}
In view of  Remark~\ref{rem:tata89}, 
 for all $T>0$, there exists  a sequence $\eps_\ell$ for which  one has for any $\phi\in\mathcal C^\infty(\R^d)$,
\begin{equation}\label{ineq:tata}
 \int_{\R^d}\phi(x)  |\psi^\eps(t,x)|^2 dx \geq \int_{T^*\Sigma_n}  {\rm Tr}_{L^2(N_\xi\Sigma_n)}
\left( m^{\Sigma_n}_\phi(\xi,v) (m_n^t+m^t_{n+1} +2{\rm Re}\,  m^t_{n,n+1} )(\sigma,v)\right) \nu^0(d\xi,dv).
\end{equation}
Let us assume 
\[
\int_{T^*\Sigma_n}  {\rm Tr}_{L^2(N_\xi\Sigma_n)}
\left( m_n^0+m^0_{n+1} +2{\rm Re}\,  m^0_{n,n+1} \right) d\nu^0=\| \psi^\eps_0\|_{L^2}
\]
then
$\displaystyle{
\int_{T^*\Sigma_n}  {\rm Tr}_{L^2(N_\xi\Sigma_n)}
\left( m_n^t+m^t_{n+1} +2{\rm Re}\,  m^t_{n,n+1} \right)  d\nu=
|\psi^\eps(t,x)|^2 dx}$,
and the inequality~\eqref{ineq:tata}  becomes  an equality. One has obtained  a $t$ by $t$  description of the limit of the energy density. 
The same observation holds in the frame of Theorem~\ref{theorem43}. 
\end{remark}

It remains to prove Proposition~\ref{prop:special_cas2}.

\begin{proof}[Proof of Proposition~\ref{prop:special_cas2}]
According to the assumptions, we have  $\Lambda_n=\Lambda_{n+1}=\Sigma_n$. Therefore, we only have to compute $M d\nu^0$ and solve the  von Neumann equations defining $M^t$ in both studied cases. 
For the data considered in that statement, one has 
$\psi^\eps_0=\psi^\eps_{0,n}+\psi^\eps_{0,n+1}$ (where the latter families are defined in~\eqref{eq:tata}) with $\xi_n=\xi_{n+1}$.
Therefore, by Corollary~\ref{cor:titi}, 
$\nu^0$ is given by 
 \[
 \nu^0(d\xi)= (\| v_{n}\|_{L^2}^2+ v_{n+1}\|_{L^2}^2)
 \sum_{j\in 2\pi \Z^d} |c_{n} (\xi_{n}+ j)|^2 \delta(\xi- \xi_{n}-j)
 \] and, for $j\in 2\pi\Z^d$, the operator~$M^0(\xi_{n} +j)$ is the projector on $\C \, ^t(v_{n},v_{n+1})$.
 The solution of the Heisenberg equations~\eqref{ed:heis2bis} and~\eqref{ed:heis2} then are orthogonal projectors on
 $\C \, ^t(\psi^{\xi_n}_{n},\psi^{\xi_n}_{n+1})$ as defined in the statement of Proposition~\ref{prop:special_cas2}.
 In view of~\eqref{step91} and of $m_\phi^{\Sigma_n} = \phi(x)$, we conclude 
 for $\phi\in\mathcal C_0^\infty (\R^d)$,
 \begin{align*}
 \int _{\Sigma_{n}} {\rm Tr}_{L^2(\R^d,\C^2)} (m_\phi^{\Sigma_{n}}  E M^t(\xi))\nu^0 (d\xi) 
 &= \left(\int_{\R^d} \phi(x) | \psi^{\xi_{n}}_n(t,x) +\psi^{\xi_n}_{n+1} (t,x)|^2 dx \right)\sum_{j\in\Z^d} |c_{n_0} (\xi_{n_0}+ j)  |^2\\
 &=\int_{\R^d} \phi(x)| \psi^{\xi_{n}}_n(t,x) +\psi^{\xi_n}_{n+1} (t,x)|^2dx 
  \end{align*}
\end{proof}

\subsection{Proof of Theorem~\ref{theorem:1d}} \label{sec:proof1d}
 By Lemma \ref{lem:critb}, the  Bloch energies $\varrho_n$ have only non-degenerate critical points and $\Lambda_n\subset \pi\Z$. Besides, they are smooth outside the set of crossing points $\Sigma_n=\pi\Z\setminus\Lambda_n$, that are all conical. Therefore, the assumptions of Theorem~\ref{theorem0} are satisfied and 
 \[
\varsigma^t=\sum_{n\in I_n}\mu^t_{n,n}. 
\]
with $\mu^t_{n,n}$ determined by the pairs  $M^t_n) \nu^t_n$. It remains to characterize the pairs $(M^t_n,\nu_n)$ that are associated with the discrete sets $\Lambda_n$. For this reason, 
$T^*\Lambda_n=\Lambda_n\times\{0\}$ and $N\Lambda_n = \R^d$, the measure $\nu^t_n$ is a sum of Dirac masses and the operator $M^t_n$ is constant and an orthogonal projector on a function $\psi _\xi^{(n)}$ that has to satisfy~\eqref{eq:schrohprofil} since $M^t_n$ satisfies~\eqref{eq:heis} (see also Corollary~1.4 in~\cite{CFM2}).

\subsection{Extension of the setting to more general situations}\label{sec:extension}

Our results could be formulated differently by assuming that the initial data is localised in Fourier variables on a set
 $\Omega$ in which the assumptions \textbf{H1}, \textbf{H2} and \textbf{H3}   of Theorem~\ref{theorem0}  are satisfied. More precisely, we assume that $\Omega$ is a  $\Z^d$-periodic open subset of $\R^d$ such that there exists a unit cell ${\mathcal B}$ of $\Z^d$ for which $\Omega\cap \mathcal B$ is strictly included in $\mathcal B$.
We prove here that the analysis of the semi-classical measure $\varsigma^t$ of $(\psi^\eps)_{\eps>0}$  in  $\R\times \R^d\times\Omega$  can be performed by localizing the initial data $\psi^\eps_0$, which allows to extend the results of Theorem~\ref{theorem0} to data with less strict assumptions on the Bloch energies. 

\begin{lemma}\label{lem:Ueps}
Let  $\Omega$ as above 
and $\chi\in{\mathcal C}^\infty(\R^d)$ be $2\pi\Z^d$-periodic, supported in the interior of $\mathcal B+2\pi\Z^d$ and equal to 1 on $\Omega$.
Let $U_\chi^\eps(t)$ be the solution of equation~\eqref{eq:U} with initial data $\chi(\eps D) U^\eps_0$. 
Then, for every $s\geq 0$ there exists a constant $C_s>0$ such that, for all $t\in\R$,
$$\left\| U^\eps_\chi (t)-\chi(\eps D_x) U^\eps(t)\right\|_{H^s_\eps (\R^d\times \T^d)}\leq \eps \, C_s (1+|t|).$$ 
Moreover, if $\psi^\eps_0=L^\eps U^\eps_0$, then there exist $C>0$ such that
\[
\left\| L^\eps U_\chi^\eps(t)-\chi(\eps D_x) \psi^\eps(t)\right\|_{L^2 (\R^d)}\leq \eps \, C (1+|t|).
\]
\end{lemma}

\begin{proof}[Proof of Lemma~\ref{lem:Ueps}]
Note that we have $U^\eps_\chi(0)=\chi(\eps D_x) U^\eps_0$. 
We observe that $\widetilde U^\eps = \chi(\eps D_x) U^\eps$ satisfies the system 
$$i\eps^2\partial_t \widetilde  U^\eps =P(\eps D_x)\widetilde  U^\eps+\eps^2 V_{\rm ext}\widetilde  U^\eps +\eps^3 F^\eps,$$
with $F^\eps(t) = \eps^{-1} \left[\chi(\eps D_x), V_{\rm ext}(t)\right] U^\eps(t)$ is uniformly bounded in $L^2(\R^d\times \T^d)$. A standard energy estimate then gives the result for $s=0$. Then, in view of  Remark~\ref{rem:eqnorm}, it is enough to prove that 
$P(\eps D_x)^{s/2}(\widetilde U^\eps-U^\eps_\chi)$ and $ \langle \eps D_x \rangle ^s (\widetilde U^\eps-U^\eps_\chi)$ go to~$0$ in $L^2(\R^d\times \T^d)$. We proceed by induction in $s\in\N$ and   interpolation between $s$ and $s+1$,  following the arguments of the proof of  Lemma~6.7 in~\cite{CFM2}. This proves the first estimate of the lemma.
\smallskip 

To prove the second estimate, note that whenever $\chi$ is $2\pi\Z^d$-periodic, we have 
$$\chi(\eps D) \left({\rm e}^{\frac{i}{\eps}k\cdot x}\cdot\right)= {\rm e}^{\frac{i}{\eps}k\cdot x} \chi(\eps D+k)= {\rm e}^{\frac{i}{\eps}k\cdot x}\chi(\eps D)$$ 
for $k\in 2\pi\Z^d$. Thus $[\chi(\eps D), L^\eps]=0$ where $L^\eps$ is the operator defined in~\eqref{def:Leps}. We deduce that
$$\chi(\eps D) \psi^\eps (t)= \chi(\eps D) L^\eps U^\eps (t) = L^\eps \chi(\eps D) U^\eps (t).$$
Therefore, combining \eqref{def:Leps} and the previous estimate, finishes the proof of the lemma.
\end{proof}

%%%%%%%%%%%%%%%%%%%%%%%%%%%%%%%%%%%%%%%%%%%%%%%%%%%%%%%%
\appendix

\section{One dimensional Bloch modes} \label{sec:1d}

We  review the main aspects of the theory that are needed here; the reader can refer to the books~\cite{McKean,RS} or the articles~\cite{Kuch16,McKeanTrubowitz,FK} among others for additional details. 

\smallskip

First of all, note that $\phi\in L^2(\T)$ solves $P(\xi)\phi=\lambda\phi$ for some $\xi,\lambda\in\R$ if and only if $f(y,\lambda):={\rm e}^{i\xi y}\phi(y)$ is a solution to:
\begin{equation} \label{eq:ODE} 
-{1\over 2} \partial_y^2 f(y,\lambda) +V_{\rm per}(y) f(y, \lambda)=\lambda f(y, \lambda), \quad y\in\R,
\end{equation}
satisfying the periodicity condition
\begin{equation}\label{eq:bper}
f(1,\lambda)={\rm e}^{i\xi}f(0,\lambda).
\end{equation}
Given $\lambda\in\R$, the solutions of \eqref{eq:ODE} are linear combinations of two solutions $f_1(y,\lambda)$ and $f_2(y,\lambda)$ satisfying
$$f_1(0,\lambda)=\partial_y f_2(0,\lambda)=1,\;\; f_2(0,\lambda)= \partial_y f_1(0,\lambda)=0.$$
Define:
\[
M_\lambda(y):=\left( \begin{array}{cc}
f_1(y,\lambda) & f_2(y,\lambda) \\ 
\partial_y f_1(y,\lambda) & \partial_y f_2(y,\lambda) 
\end{array} \right);
\]
then the existence of a solution to \eqref{eq:ODE} satisfying \eqref{eq:bper} is equivalent to the fact that ${\rm e}^{i\xi}$ is an eigenvalue of $M_\lambda(1)$. One can check that $\det M_\lambda(y)=1$ for every $y, \lambda\in\R$; therefore, letting $\Delta(\lambda):=\Tr M_\lambda(1)$, we find that ${\rm e}^{i\xi}\in \Sp M_\lambda(1)$ if and only if:
\begin{equation}\label{eq:Delta}
\Delta(\lambda)=2\cos\xi.
\end{equation}
It can be shown that solutions to \eqref{eq:ODE} depend analytically on $\lambda$, and that moreover, $\Delta$ extends to an entire function of order $1/2$. The real solutions to equations $\Delta(\lambda)=\pm 2$ form infinite increasing  sequences $(a_i^\pm)$ that tend to infinity. 

\smallskip 

The following facts hold (the reader may find helpful to consult~\cite[Figure~1, p.~145]{McKeanTrubowitz} or~\cite[Section XIII.16]{RS}):
\begin{itemize}
\item The sequences $(a_i^\pm)$ are intertwined. More precisely, one has:
\begin{equation}\label{eq:sequence}
a_1^+<a_1^-\leq a_2^-< a_2^+\leq a_3^+< a_3^-\cdots,
\end{equation}
\item Let be $I_{2i-1} = (a_{2i-1}^+,a_{2i-1}^-)$ and $I_{2i} = (a_{2i}^-,a_{2i}^+)$. Then $I_i$ has non-empty interior and $\restr{\Delta}{I_i}$ is strictly decreasing for $i$ odd and strictly increasing for $i$ even. 
\item If $a_i^\sigma=a_{i+1}^\sigma$ for some $i\in\N,\,\sigma\in\{+,-\}$ then
$
\Delta^\prime(a_i^\sigma)=0.
$
\end{itemize}
These properties have important implications on the behavior of Bloch energies. For every $n\in\N$ the following hold.
\begin{enumerate}
\item[Fact 1] The $n^{{\rm th}}$ Bloch energy is the solution to $\restr{\Delta}{I_n}(\varrho_n(\xi))=2\cos\xi$.
\item[Fact 2] $\varrho_n$ is $2\pi\Z$-periodic (we knew this already), and moreover 
\[
\varrho_n(\xi)=\varrho_n(2\pi-\xi),\quad \forall\xi\in\R.
\]
\item[Fact 3] $\restr{\varrho_n}{[0,\pi]}$ is strictly increasing if $n$ is odd (resp. strictly decreasing if $n$ is even) and analytic in the interior of the interval. If it is differentiable at $\xi=0,\pi$ then necessarily $\varrho_n^\prime(\xi)=0$ and $\varrho_n$ is analytic around that point. 
\item[Fact 4] A crossing can happen only at two consecutive Bloch energies. Let $n\in\N$ be such that
\[
\Sigma_n:=\{\xi\in\R\,:\,\varrho_n(\xi)=\varrho_{n+1}(\xi)\}\neq\emptyset;
\]
then $\Sigma_n=\pi\Z\setminus 2\pi\Z$ if $n$ is odd, $\Sigma_n=2\pi\Z$ if $n$ is even. Moreover
\begin{equation}\label{eq:Deltavan}
\Delta^\prime(\varrho_n(\xi))=0,\quad \forall \xi\in \Sigma_n.
\end{equation}
\end{enumerate}

\begin{figure}[htpb!]
	\label{fig:hill_discriminant}
	\includegraphics[width=0.49\textwidth]{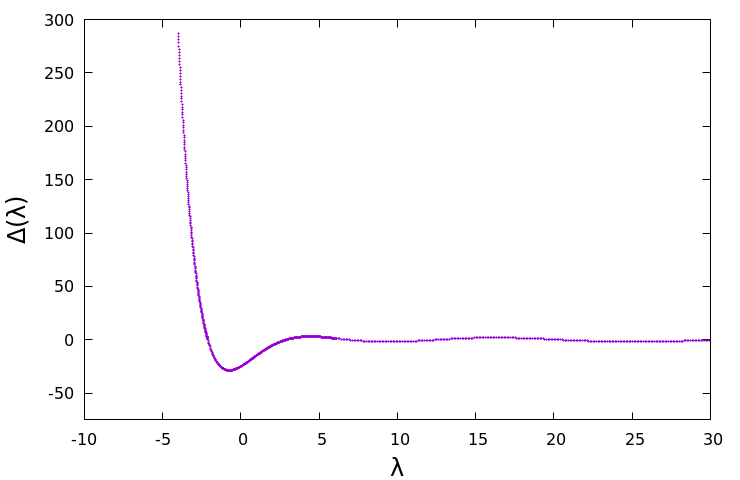} \
	\includegraphics[width=0.49\textwidth]{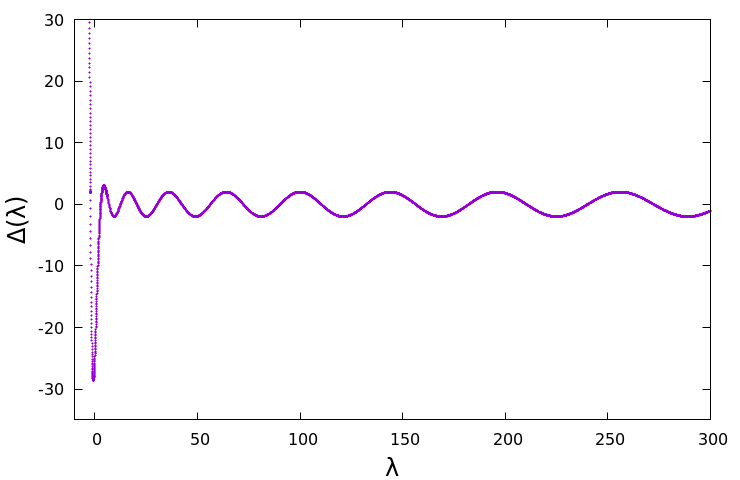} \
	
	\
	
	\includegraphics[width=0.49\textwidth]{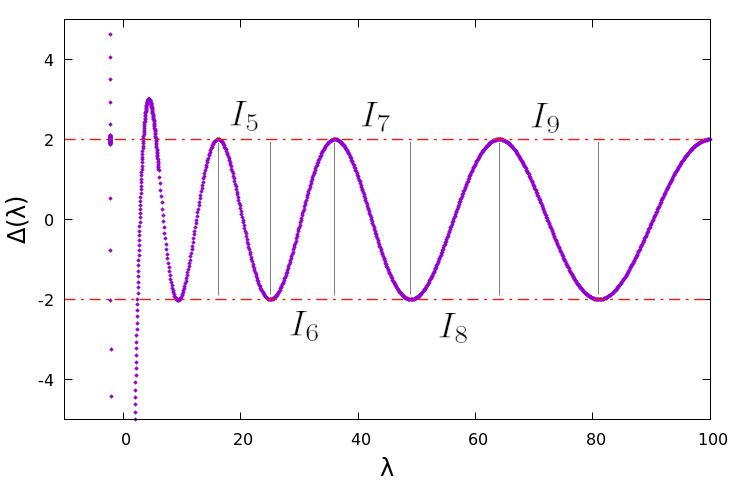} \
	\includegraphics[width=0.49\textwidth]{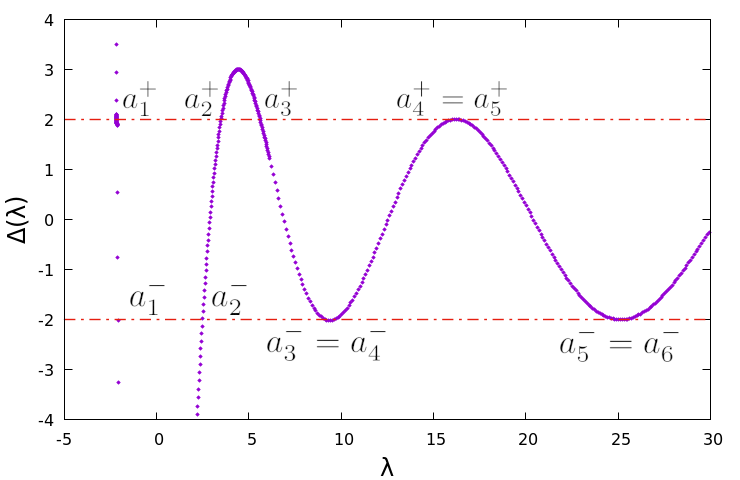}
	
	\caption{\footnotesize Hill's discriminant for $V_{\rm per}(y) = 5\cos(2y)$ numerically calculated. All graphics show the same set of data plotted with different ranges. In the first two it is possible to observe both divergence of $\Delta(\lambda)$ in the region $\lambda \approx \inf_y V_{\rm per}(y)$ and non-decaying oscillations of period $\sim \sqrt{\lambda}$ for $\lambda \longrightarrow \infty$. The third one illustrates how the intervals $I_i$, wherein solutions to \eqref{eq:ODE} are stable, are never empty, and why their borders $a_i^\pm$ must satisfy \eqref{eq:sequence}. Roots $a_i^+$ admit $\pi$-periodic solutions, and $a_i^-$ $2\pi$-periodic solutions; in each case, there are two linearly independent such solutions if and only if $a_i^\pm$ coincides with $a^\pm_{i+1}$ or $a^\pm_{i-1}$. If this is not the case, as for roots $a^\pm_1$, $a_2^\pm$ and $a^+_3$ seen in the last image, then the equation also admits a unstable solution.} A complete study of $\Delta(\lambda)$ in one dimension is found in \cite{magnus_winkler}.
\end{figure}

In addition, critical points of Bloch energies in the one dimensional case are never degenerate nor can occur at a crossing point.

\begin{lemma}\label{lem:critb}
The set of critical points of any Bloch energy $\varrho_n$ is contained in $\pi\Z$ and all the critical points are non-degenerate. Moreover, the crossing set $\Sigma_n$ associated with two consecutive Bloch energies $\varrho_n$ and $\varrho_{n+1}$ does not contain any critical points of the Bloch energies $\varrho_n$ and $\varrho_{n+1}$.
\end{lemma}

\begin{proof} The first assertion on the critical points is property (3) above, whereas the second follows from differentiating twice equation \eqref{eq:Delta} and evaluating at a critical point $\xi=k\pi$, $k\in\Z$ to get:
\[
\Delta^\prime(\varrho_n(k\pi))\varrho_n^{\prime\prime}(k\pi)=(-1)^{k+1}2.
\]
This relation also shows that $\Delta^\prime(\lambda)$ cannot vanish at $\lambda=\varrho_n(k\pi)$. Together with \eqref{eq:Deltavan} this shows that a critical point cannot be a crossing point.
\end{proof}

\begin{remark}
In the free case ($V_{\rm per}=0$) there is only a Bloch band of infinite multiplicity. More generally, it has been proved in~\cite{Borg} that the absence of spectral gap is equivalent to the periodic potential $V_{\rm per}$ being constant. 
\end{remark}

\section{The properties of the Bloch energies at crossing points}\label{app:crossings}

Here we present a normal form for the expression of two Bloch energies $\varrho_n(\xi)$ and $\varrho_{n+1}(\xi)$ close to 
the crossing set $\Sigma_n$ (defined as $\Sigma_n=\Sigma_{n,n+1}$ in~\eqref{def:Sigmann'}). 

\begin{lemma} \label{lem:structurevarrho}
Let $\sigma_0$ be a point in the crossing set $\Sigma_n$ of two consecutive Bloch energies $\varrho_n$ and $\varrho_{n+1}$ having neighborhood $U$ with the following properties:
\begin{enumerate}
\item[(i)] $\Sigma_n\cap U$ is a smooth manifold.
\item[(ii)] The multiplicities of $\varrho_n(\xi),\varrho_{n+1}(\xi)$ are constant on each connected component of  $U\setminus \Sigma_n$.
\item[(iii)] There exists $\delta_0>0$ such that for all $\xi\in U$, 
$$d\left(\{ \varrho_n(\xi),\varrho_{n+1}(\xi) \}, {\rm Sp} \, P(\xi) \setminus \{ \varrho_j(\xi), \;\;\varrho_j(\xi)=\varrho_n(\xi) \; or\; \varrho_j(\xi)=\varrho_{n+1}(\xi) \} \right)\geq \delta_0.$$
\end{enumerate}
Then, there exist $\Omega\subseteq U$, a neighborhood of $\sigma_0$ that is $2\pi\Z^d$-invariant, two functions $\lambda_n\in \cC^\infty(\Omega)$, $g_n\in\cC^\infty\left(\sqcup_{\xi\in\Omega} \left(\{\xi\} \times N_{\sigma_{\Sigma_n}(\xi)}\Sigma_n\right)\right)$, and a function $m\in L^\infty(U)$ which is constant on each connected component of $U$ such that 
\begin{align*}
\forall  \xi\in\Omega\setminus \Sigma_n, \;\;&\varrho_n(\xi) =\lambda_n(\xi)- 
g_n(\xi,\xi-\sigma_{\Sigma_n}(\xi)),\;\;
\varrho_{n+1}(\xi) = \lambda_n(\xi)+m(\xi)
g_n(\xi,\xi-\sigma_{\Sigma_n}(\xi)).
\end{align*}
Moreover, 
\begin{enumerate}
\item If  the crossing set $\Sigma_n$ is conical in $U$, then for all $\xi \in U $, the map $N_{\sigma(\xi)} \Sigma\ni\eta\mapsto g_n(\xi,\eta)$ is homogeneous of degree $1$ and $g_n(\sigma,\eta)\not=0$ when $(\sigma,\eta)\in N\Sigma_n$ with $\eta\not=0$,
\item If none of the points of $\Sigma_n$ are conical crossings in $U$, then there exists $\theta_n\in\cC^\infty(\R^d )$ such that $g_n(\xi,\eta)= |\eta|^2 \theta_n (\xi)$, which implies that $\varrho_n$, $\varrho_{n+1}\in \mathcal C^{1,1}(\R^d)$, 
\item 
If the multiplicities of $\varrho_n,\varrho_{n+1}$ are equal on $U\setminus \Sigma_n$ then $m=1$.
\end{enumerate}
\end{lemma}

\begin{remark}
Note that in case~(2), the function $\theta_n$ can be zero on $\Sigma_n$.  
\end{remark}

\begin{proof}
We denote by  $j_-(\xi),j_+(\xi)$ the functions valued in  $\N$  and constant on connected component of $U\setminus\Sigma_n$ such that for all $\xi\in U\setminus\Sigma_n$
$\varrho_{n-j+1}(\xi)=\varrho_n(\xi)$ for $1\leq j\leq j_-(\xi)$ and $\varrho_{n+j}(\xi)=\varrho_{n+1}(\xi)$ for $1\leq j\leq j_+(\xi)$.
We denote  by $\Pi(\xi)$ the projector on 
$$F_\xi={\rm Ker}(P(\xi)-\varrho_n(\xi)) \oplus  {\rm Ker}(P(\xi)-\varrho_{n+1}(\xi)).$$
By the assumptions on $U$, the pair $\{\varrho_n(\xi),\varrho_{n+1}(\xi)\}$ is isolated from the remainder of the spectrum of~$P(\xi)$ when $\xi\in U$, this implies that the map 
$U\ni\xi\mapsto \Pi(\xi)\in \cL(L^2(\T^d))$ is analytic and the function 
$\dim  F_\xi$ is constant for $\xi\in U$. We denote by $\ell_0$ this constant and we have $\ell_0=j_-(\xi)+j_+(\xi)$ for all $\xi\in U\setminus \Sigma_n$.
Moreover, $\varrho_n(\xi)$ and $\varrho_{n+1}(\xi)$ are the two only eigenvalues of the 
 operator $\Pi(\xi)P(\xi)\Pi(\xi)$ which maps~$F_\xi$ onto $F_\xi$ for any~$\xi\in\R^d$. 
 
 \smallskip 
 
Let us first show that it is possible to find $\Omega\subseteq U$, with $\sigma_0\in\Omega$ and  construct, for every $\xi\in\Omega$, an orthonormal basis  $(\phi_j(\xi,\cdot ))_{1\leq j\leq \ell_0}$  of $F_\xi$ such that the maps $\xi\mapsto \phi_j(\xi,\cdot)$ are analytic for all $j\in\{1,\cdots \ell_0\}$. To see this, consider $(\varphi_i(\sigma_0,\cdot))_{1\leq i\leq \ell_0}$, a basis of $F_{\sigma_0}$. Chose a neighborhood $\Omega$ of $\sigma_0$ small enough to ensure that the vectors 
 $$\Pi(\xi) \varphi_{j}(\sigma_0,\cdot),\;\;  j\in\{1,\hdots,\ell_0\}$$
form a rank~$\ell_0$ family. Then apply the standard Schmidt orthonormalization process to this family.

\smallskip

 Let $A(\xi)$, $\xi\in\Omega$, be the matrix of the operator  $\Pi(\xi)P(\xi)\Pi(\xi)$ in the basis we just constructed. This is a $\ell_0\times\ell_0$ analytic matrix that we can write 
 $$A(\xi)=\lambda_n(\xi) {\rm Id} +  A_0(\xi)$$
 with $\lambda_n(\xi):= \frac 1{\ell_0} {\rm Tr}_{\C^{\ell_0}} A(\xi)$ and $A_0(\xi)$  analytic and trace-free. 
Moreover, $A(\xi)$ is diagonalizable and has only two eigenvalues $\varrho_n(\xi)$ and $\varrho_{n+1}(\xi)$ that we write 
 $$\varrho_n(\xi) = \lambda_n(\xi) - g(\xi)  ,\quad
 \varrho_{n+1}(\xi) = \lambda_n(\xi) + m(\xi)g(\xi) ,$$
 with $g(\xi)>0$ and 
 where, for $\xi\in\Omega\setminus\Sigma_n$, $m(\xi)$ is the ratio between the multiplicities of $\varrho_n(\xi)$ and $\varrho_{n+1}(\xi)$, 
 $$m(\xi)= \frac {j_-(\xi)}{j_+(\xi)}$$
 and $m$ is constant in the connected componnent of $U\setminus\Sigma_n$. 
 
 \smallskip 
 
The functions  $-g(\xi)$ and $m(\xi)g(\xi)$ are the two eigenvalues of $A_0(\xi)$. Therefore,
they are  homogeneous function of degree $1$ of  the coefficients of $A_0(\xi)=(a_{i,j}(\xi))_{1\leq i,j\leq \ell_0}$: we write 
 $ g(\xi)= G(A_0(\xi))$ where $G$ is a homogeneous function on $\R^{\frac{\ell_0^2-1}2}$.
 Here, we have considered  that a $\ell_0\times\ell_0$ trace-free Hermitian matrix is a function of $\ell_0-1$ real-valued diagonal coefficients   and of $\frac{\ell_0(\ell_0-1)}{2}$ complex-valued coefficients (those under the diagonal being the conjugate of those above the diagonal), and we have observed that $(\ell_0-1) +\frac{\ell_0(\ell_0-1)}{2}= \frac{\ell_0^2-1}{2}$.

 \smallskip
 
By the definition of the crossing set, $A_0(\xi)=0$ if and only if $\xi\in\Sigma_n$. Since the map $\xi\mapsto A_0(\xi)$ is analytic, it vanishes on $\Sigma_n$ at finite order $q\in\N$ and the crossing set is conical if and only if $q=1$ for all points of~$\Sigma_n$. 
Therefore, in case (1), there exists a smooth tensor $T^{\ell_0,1}(\xi)$   such that 
 $$A_0(\xi) = T^{\ell_0,1}(\xi) [\xi-\sigma_{\Sigma_n}(\xi)],$$
 with 
 $$\forall \sigma\in\Sigma_n\cap \Omega, \;\; 
 \forall \eta\in N_{\sigma}\Sigma_n \setminus \{0\},\;\; T^{\ell_0,1}(\sigma) \eta\not=0_{\C^{\ell_0\times\ell_0}}.$$
 We deduce that
 $$g(\xi)= g_n(\xi,\xi-\sigma_{\Sigma_n}(\xi)),\;\;\mbox{with} \;\;g_n(\xi,\eta):=G\left( T^{\ell_0,1}(\xi)\left[ \eta\right]^q\right)
  $$
 where 
 $g_n$ is homogeneous of degree~$1$ in the variable~$\eta$.
Besides, if none of the crossing points are conical, we write 
 $A_0(\xi) = T^{\ell_0,2}(\xi) [\xi-\sigma_{\Sigma_n}(\xi)]^2$
 with $T^{\ell_0,2}(\xi)$ a smooth tensor, which allows to prove Point (2) with 
 $$\theta_n(\xi)= |\xi-\sigma_{\Sigma_n}(\xi)|^{-2}G(T^{\ell_0,2}(\xi) [\xi-\sigma_{\Sigma_n}(\xi)]^2).$$
 That concludes the proof since Point (3) is obvious. 
 \end{proof}

 \section{Semi-classical pseudo-differential calculus}\label{app:pseudo}

We recall here results about matrix-valued semi-classical pseudo-differential operators. 
We denote by $S_{N\times N}$ the set of functions  $a=(a_{i,j})\in\cC^\infty(\R^{2d},\C^{N\times N})$ which are bounded together with their derivatives in matrix-norm.
Then, for $a\in S_{N\times N}$, one defines the Weyl semi-classical pseudo-differential operator of symbol $a$ as 
$$\op_\eps(a) f(x)=\int_{\R^{2d}} {\rm e}^{{i\over \eps} \xi\cdot (x-y)}  a\left({x+y\over 2},\xi\right)f(y)dy\, \frac{d\xi}{(2\pi\eps)^{d}},\quad \forall f\in{\mathcal S}(\R^d,\C^N).$$ 
Properties of these matrix-valued pseudo-differential operators follow from the well-understood scalar theory, once the definition of the product of matrices and its non-commutativity is taken into account. Unless stated otherwise, the reader may found proofs of the scalar versions of the results presented here in~\cite{Dimassi1999,Zwobook,F14}, for instance. 

\smallskip 

The Calderón-Vaillancourt theorem \cite{CV,Boulk99} extends to the matrix-valued case and ensures the existence of a constant $C_d>0$ such that for every $ a\in S_{N\times N}$ one has
\begin{equation}\label{eq:pseudo}
\| \op_\eps(a)\|_{{\mathcal L}(L^2(\R^d,\C^{N}))}\leq C_d\, 
N_d^\eps(a),
\end{equation}
where
$$
N_d^\eps(a):=\sum_{\alpha\in\N^{2d},|\alpha|\leq d+2} \eps^{|\alpha|/2}\sup_{\R^d\times\R^d}|\partial_{x,\xi}^\alpha a|_{\C^{N\times N}}.
$$
This estimate shows that semi-classical pseudo-differential operators are uniformly bounded in $L^2(\R^d,\C^N)$ for $\eps\in(0,1]$. Moreover, for every $ a\in S_{N\times N}$,
\begin{equation}\label{eq:adj}
\op_\eps(a)^*=\op_\eps(a^*).
\end{equation}
In particular, semi-classical pseudo-differential operators whose symbols take values in the space of Hermitian matrices are self-adjoint on $L^2(\R^d,\C^N)$.

Besides, the symbolic calculus for matrix-valued pseudodifferential operators goes as follows.

\begin{proposition}\label{prop:symbol}
Let $a,b\in S_{N\times N}$, then
$$\op_\eps(a)\op_\eps(b)  = \op_\eps(ab)+\frac{\eps}{2i} \op_\eps(\{a,b\})+\eps^2 R^{(1)}_\eps,$$
with $\{a,b\}=\sum_{j=1}^d\partial_{\xi_j} a\,  \partial _{x_j} b-\partial _{x_j}a\, \partial_{\xi_j} b$ and 
$$\left[\op_\eps(a),\op_\eps(b)\right] =\op_\eps([a,b])+ {\eps\over 2i}(\op_\eps(\{a,b\}) -\op_\eps(\{b,a\}))+\eps^2 R_\eps^{(2)},$$
$$\| R^{(j)}_\eps\|_{{\mathcal L}(L^2(\R^d,\C^N))}\leq C \,\sup_{|\alpha|+|\beta|=2} N_d^\eps(\partial_\xi^\alpha \partial_x^{\beta}  a) N_d^\eps(
 \partial_\xi^\beta \partial_x^{\alpha}  b),\quad j\in\{1,2\},$$
for some constant $C>0$ independent of $a$, $b$ and $\eps$.
\end{proposition}

\begin{remark} \label{rem:calculpseudo}
The term of order $\eps^2$ above has a particularly simple expression when $b\in S_{N\times N}$ does not depend on $x$. The following hold in ${\mathcal L}(L^2(\R^d,\C^N))$: 
$$\displaylines{
\op_\eps(b)\op_\eps(a)  = \op_\eps(ba)+\frac{\eps}{2i}\sum_{j=1}^d \op_\eps( \partial_{\xi_j} b\,  \partial _{x_j}a)+{\eps^2\over 8} \sum_{1\leq \ell,p\leq d} \op_\eps( \partial^2_{\xi_\ell\xi_p}b\, \partial^2_{x_\ell x_p} a )+ O(\eps^3),\cr
\op_\eps(a)\op_\eps(b)  = \op_\eps(ab)-\frac{\eps}{2i} \sum_{j=1}^d\op_\eps(\partial _{x_j}a\, \partial_{\xi_j} b)+{\eps^2\over 8} \sum_{1\leq \ell,p\leq d} \op_\eps(\partial^2_{x_\ell x_p} a \, \partial^2_{\xi_\ell\xi_p}b)+ O(\eps^3).\cr}
$$
\end{remark}

There are also analogues of Gårding's inequality for elliptic differential operators \cite{Gar53} in this context. If $a\in S_{N\times N}$ takes values in the set of non-negative Hermitian matrices then there exist $C_a>0$, which depends on a finite number of derivatives of $a$ such that
\begin{equation}\label{eq:gengard}
\op_\eps(a)+C_a\eps \Id\geq 0.
\end{equation}
This follows by considering $a*\rho_\eps$, with $\rho_\eps(x,\xi):=(\pi\eps)^{-d} e^{-\frac{|x|^2+|\xi|^2}{\eps}}$. A direct computation shows that $\op_\eps(a*\rho_\eps)$ is a non-negative operator (see for instance \cite{GerLeich93}). One concludes by noticing that, due to \eqref{eq:pseudo},
\begin{equation}\label{eq:husimi}
\op_\eps(a)-\op_\eps(a*\rho_\eps)=O(\eps),
\end{equation}
in ${\mathcal L}(L^2(\R^d,\C^N))$. The operator $\op_\eps(a*\rho_\eps)$ is called the anti-Wick quantization of the symbol $a\in S_{N\times N}$, besides the fact that non-negative symbols correspond to non-negative operators, the operator norm satisfies a simpler bound than that satisfied by their Weyl counterparts, namely:
\begin{equation}\label{eq:bddhusimi}
\| \op_\eps(a*\rho_\eps)\|_{{\mathcal L}(L^2(\R^d,\C^{N}))}\leq \|a\|_{L^\infty(\R^{2d},\C^{N\times N})}.
\end{equation}

\smallskip 

Through the article, it is necessary to understand the boundedness and symbolic calculus properties of operators with symbols of limited regularity. Denote by $\cA$ the completion of $\cC^\infty_0(\R^{2d},\C^{N\times N})$ with respect to the norm:
\begin{equation}\label{eq:norma}
\|a\|_\cA :=  \sup_{\xi\in\R^d} \sup _{|\alpha|\leq d+2} \int_{\R^d} | \partial_x^\alpha a(x,\xi) |_{\C^{N\times N}} dx .
\end{equation}
\begin{remark}\label{rem:sep}
By a mollification argument, one can show that $\cA$ contains all the functions $a\in\cC_0(\R^{d}_x\times\R^d_\xi,\C^{N\times N})$ that are $d+2$ times continuously differentiable with respect to the first variable.
\end{remark}

This regularity assumption is sufficient for our purposes (see \cite[Lemma~3.7]{Ge91} and~\cite[Section~3]{FGL} for related results).
\begin{lemma}\label{lem:singulier}
The space $\cA$ enjoys the following properties.
\begin{enumerate}
\item There exists a universal constant $C_d>0$ only depending on $d$ such that, for every $a\in\cA$,
$$\| {\rm op}_\eps(a)\|_{\cL(L^2(\R^d,\C^N))} \leq  C_d \|a\|_\cA. $$
\item Suppose $\varrho\in \mathrm{Lip}(\R^d_\xi,\C^{N\times N})$ and $a\in\cA$. The following hold in ${\mathcal L}(L^2(\R^d,\C^N))$:
$${\rm op}_\eps(a\,\varrho)= \op_\eps(a) \varrho(\eps D_x)+O(\eps),$$
$$\op_\eps(\varrho\, a)=\varrho(\eps D_x)\op_\eps(a) +O(\eps).$$
\item Suppose $\varrho\in \mathcal{C}^{1,1}(\R^d_\xi,\R)$ (that is $\nabla \varrho\in\mathrm{Lip}(\R^d,\R^d)$) and $a\in\cA$. The following hold in ${\mathcal L}(L^2(\R^d,\C^N))$:
$$\frac i \eps [{\rm op}_\eps(a),\varrho(\eps D_x){\rm Id} ] ={\rm op}_\eps(\nabla_x a\cdot \nabla \varrho(\xi))+O(\eps).$$
\end{enumerate}
\end{lemma}

\begin{proof}
Suppose that $A\in \cL(L^2(\R^d,\C^N))$ is of the form
\begin{equation}\label{eq:ker1}
Af(x)=\frac{1}{\eps^d}\int_{\R^d}k\left(\frac{x+y}{2},\frac{x-y}{\eps}\right)f(y)dy, \quad \forall f\in \Sch(\R^d,\C^N);
\end{equation}
then one deduces, after change of variables and using Hölder's inequality, that
\begin{equation}\label{eq:ker2}
|(A\,f,g)_{L^2(\R^d,\C^N)}|\leq \|f\|_{L^2(\R^d,\C^N)}\|g\|_{L^2(\R^d,\C^N)}\int_{\R^d}\sup_{x\in\R^d}|k(x,v)|_{\C^{N\times N}}dv.\quad \forall f,g\in\Sch(\R^d,\C^N).
\end{equation}
In order to prove item (1), notice that the following identity holds,
\begin{equation}\label{eq:conj}
\op_\eps(a)=({\mathcal F}^\eps)^*\op_\eps(\underline{a}){\mathcal F}^\eps,\quad \underline a(x,\xi):=a(-\xi,x),
\end{equation}
where $\mathcal{F}^\eps$ stands for the semi-classical Fourier transform:
\[
{\mathcal F}^\eps(f)(\xi)= \int_{\R^d} {\rm e}^{-  i \frac{\xi}{\eps}\cdot x} f(x)  \frac{dx}{(2\pi\eps)^{d/2}}.
\]
The operator $\op_\eps(\underline{a})$ is of the form \eqref{eq:ker1} with $k=(2\pi)^{-d}\widehat{a}$, where $\widehat{a}$ denotes the Fourier transform of $a$ with respect to the first variable. 
Using \eqref{eq:ker2}, \eqref{eq:conj} and Plancherel's formula we conclude
\[
|(\op_\eps(a)f,g)_{L^2(\R^d,\C^N)}|\leq \|f\|_{L^2(\R^d,\C^N)}\|g\|_{L^2(\R^d,\C^N)}\int_{\R^d}\sup_{\xi\in\R^d}|\widehat{a}(v,\xi)|_{\C^{N\times N}}\frac{dv}{(2\pi)^d} \quad \forall f,g\in\Sch(\R^d,\C^N).
\]
The constant $C_d>0$  is then chosen such that:
\begin{equation*}
\int_{\R^d} \sup_{\xi\in\R^d} |\widehat{a}(v,\xi)|_{\C^{N\times N}}\frac{dv}{(2\pi)^d}\leq C_d  \|a\|_\cA. 
\end{equation*}
This concludes the proof of the first assertion.

\smallskip 

In order to prove the second assertion, we show that 
\[
R^\eps:=(\cF^\eps)^*({\rm op}_\eps(a\,\varrho)- \op_\eps(a) \varrho(\eps D_x))\cF^\eps,
\]
satisfies $\|R^\eps\|_{\cL(L^2(\R^d,\C^N))}=O(\eps)$;  the proof of the other identity is analogous. From~\eqref{eq:conj}, we deduce that
\[
R^\epsilon f(x)=\int_{\R^d}r^\eps\left(\frac{x+y}{2},\frac{x-y}{\eps}\right)\,f(y)\frac{dy}{(2\pi\eps)^d},\quad \forall f\in \Sch(\R^d,\C^N), 
\]
where $r^\eps(x,v):=\widehat{a}(v,x)(\varrho(x)-\varrho(x-\eps v))$. By \eqref{eq:ker2} we can estimate:
\[
\|R^\eps\|_{\cL(L^2(\R^d,\C^N))} \leq \int_{\R^d}\sup_{x\in\R^d} |r^\eps(x,v)|_{\C^{N\times N}}\frac{dv}{(2\pi)^d}.
\] 
By hypothesis, we can find $L_\varrho>0$ such that 
\[
|\varrho(x)-\varrho(x-\eps v)|_{\C^{N\times N}}\leq L_\varrho\eps |v|,\quad \forall (v,x)\in\supp \widehat{a}.
\]
Therefore, using 
\[
 |v||\widehat a(v,x) |\leq (1+|v|^2) | \widehat a(x,v)| = | \widehat a(x,v) |+ |\widehat{(-\Delta a)} (x,v)|,
 \]
 we deduce
\[
\|R^\eps\|_{\cL(L^2(\R^d,\C^N))} \leq \eps\, C_d L_\varrho (\| a\| _{\mathcal A} + \| \Delta a\|_{\mathcal A}).
\]

\smallskip 

For the third assertion, we show that 
\[ 
\tilde R^\eps:=(\cF^\eps)^*(\frac 1\eps[{\rm op}_\eps(a),\varrho(\eps D)]- \op_\eps(\nabla_x \cdot\nabla \varrho))\cF^\eps,
\]
satisfies $\|\tilde R^\eps\|_{\cL(L^2(\R^d,\C^N))}=O(\eps)$. Indeed, the kernel of $\tilde R^\eps$ is of the form~\eqref{eq:ker1} with 
\begin{align*}
\tilde r^\eps(x,v)& 
=\frac i\eps \widehat a(v,x)\left(\varrho(x)-\varrho(x-\eps v)\right) -\widehat {\nabla_x a} (v,x) \cdot \nabla \varrho(x)\\
&= i\widehat a(v,x) v\cdot \nabla \varrho(x) - \widehat{\nabla_x a}(v,x)\cdot \nabla \varrho(x) + \eps \theta(x,v)\widehat a(x,v)
\end{align*}
and there exists $L_{\nabla \varrho}>0$ such that 
$$|\theta(x,v) |\leq L_{\nabla\varrho} |v|^2.$$
Using $iv\widehat a(v,x)= \widehat {\nabla_x a}(x,v)$ and $ |v|^2 | \widehat a(x,v)| = |\widehat{(-\Delta a)} (x,v)|$, we deduce 
$$\|\tilde R^\eps\|_{\cL(L^2(\R^d,\C^N))}\leq \eps C_d L_{\nabla\varrho} \| \Delta a\|_{\mathcal  A} .$$
\end{proof}

 \section{Two-scale pseudodifferential operators}\label{app:twomic}
 
We prove here technical lemma concerning the pseudodifferential operators considered in Section~\ref{sec:twomic}, the formalism of which we follow. Due to the properties of Bloch energies, we are going to consider more general classes of symbols than those of ${\mathcal A}^{(2)}$ (as defined in Section~\ref{sec:twomic}). For $k\in\Z$, we introduce the class ${\mathcal A}^{(2)}_k$ of smooths functions on $\R^{3d}$ that are compactly supported in the variables $(x,\xi)$ uniformly with respect to $\eta$ and coincide with an homogeneous function of degree $k$ in $\eta$ as soon as $|\eta|>R_0$ for some $R_0>0$. With these notations, $\mathcal A^{(2)}=\mathcal A^{(2)}_0$. 
\smallskip

Of particular interest for us are functions  $g(\xi,\eta)$, independent of the variable~$x$, that are smooth in~$\xi$, 
% of degree~$k$ in the variable $\eta$ 
and satisfy a symbol estimate of order~$k$ in $\eta$. We denote by $\mathcal H_k$ the set of these functions. For $X$  a connected, closed embedded submanifold of $(\R^d)^*$ and $g\in \mathcal H_k$,  the operator 
\begin{equation}\label{def:geps}
g_\eps(\eps D)= g(\eps D, D-\eps^{-1} \sigma_X(\eps D)) 
\end{equation}
 is then well defined as an operator mapping $H^s(\R^d)$ into $H^{s-k}(\R^d)$ uniformly in~$\eps$ when $s\in \R$.
We are interested in the pseudodifferential calculus involving two-scale pseudodifferential operator of the form ${\rm op}_\eps(a)$ for $a$ in ${\mathcal A}_k^{(2)}$ for some $k\in\Z$ and  Fourier multipliers  $g_\eps(\eps D)$ with $g\in{\mathcal  H}_p$ for some $p\in\Z$.

\smallskip 

The first type of pseudodifferential calculus results that we shall use concerns the large values of $\eta$. We consider $\chi\in{\mathcal C}_0^\infty(\R^d)$ such that $\chi=1$ on $B(0,1)$ and $\chi=0$ on $B(0,2)^c$ with $0\leq \chi\leq 1$.
We set for $R,\delta>0$
$$a^{R,\delta}(x,\xi,\eta)= a(x,\xi,\eta) ((1-\chi(\eta/R)) \chi ((\xi-\sigma_{\Sigma}(\xi))/\delta).$$
On the support of $a^{R,\delta}$, $\eps R \leq |\xi-\sigma_{\Sigma}(\xi)|\leq 2\delta$.
 Note that if  $k>0$ and $a\in{\mathcal A}^{(2)}_{-k}$, the estimate~\eqref{eq:norma}  yields that in $\mathcal L(L^2(\R^d))$,
$${\rm op}_\eps(a_\eps^{R,\delta})= O((\eps R)^{-k}).$$

 \begin{lemma}\label{lem:calculRdelta}
 Let $k\in\N$, $a\in{\mathcal A}^{(2)}_{-k}$, $g=g(\xi,\eta)$ in ${\mathcal H}_{k+1}$ and $\delta,R>0$.
 Then, in $\mathcal L(L^2(\R^d))$, 
 $$\left[ {\rm op}_\eps (a_\eps^{R,\delta}), g_\eps(\eps D)\right] = i \,{\rm op}_\eps\left((\nabla_x a^{R,\delta} \cdot\nabla_\eta g)_\eps\right)+O(\eps)+O(\delta)+O(1/R).$$
\end{lemma}

\begin{proof}
We want to  take advantage from the fact that  we have  $ |\xi-\sigma_{\Sigma} (\xi)|>R\eps$ on the support of $a^{R,\delta}_\eps$ to avoid the singularity of the function $g$. Indeed,  $a^{R,\delta}_\eps$ is supported outside the singularity of the function $g$. 
Let $\kappa \in{\mathcal C}^\infty(\R^d)$ supported outside $0$ and such that $(1-\chi)\kappa=(1-\chi)$ and denote by $\kappa^R$ the function defined by $\kappa^R(\eta)=\kappa(\eta/R)$ so that  
$$\kappa_\eps^R(\xi)= \kappa \left(\frac {\xi-\sigma_{X} (\xi)}{R\eps}\right).$$
 Standard symbolic calculus gives that  for all $N\in\N$, we have 
\begin{align*}
{\rm op}_\eps(a^{R,\delta}_\eps) &= {\rm op}_\eps(a^{R,\delta}_\eps)\kappa_\eps^R(\eps D)+R^{-N} {\rm op}_\eps(
 r^{1,N} _\eps(x,\xi))\\
&=\kappa_\eps^R(\eps D) {\rm op}_\eps(a^{R,\delta}_\eps)+R^{-N} {\rm op}_\eps(
 r^{2,N} _\eps(x,\xi)) 
\end{align*}
where the symbols $r^{j,N}_\eps$ have symbol norms that are uniformly bounded and are supported in the set $\{c_0 \eps R < \xi-\sigma_{X} (\xi)<C_0\eps R\}$ for 
 some $0<c_0<C_0$.
 In particular,  by  Proposition~\ref{prop:symbol} and because $g\in\mathcal H_{k+1}$, $a\in{\mathcal A}^{(2)}_{-k}$,
 $${\rm op}_\eps (
 %\tilde\chi_\eps^R
  r^{1,N}_\eps) g_\eps(\eps D)=
 {\rm op}_\eps (
 %\tilde\chi_\eps^R 
 r^{1,N}_\eps g_\eps ) +O(R^{k+1-k})=  {\rm op}_\eps (r^{1,N}_\eps g_\eps ) +O(R),
 $$
 whence
$$ {\rm op}_\eps (
 %\tilde\chi_\eps^R 
 r^{1,N}_\eps) g_\eps(\eps D)= O(R).$$
 Similarly, we have 
  $$ g_\eps(\eps D) {\rm op}_\eps (
  %\tilde\chi_\eps^R
   r^{2,N}_\eps) = O(R).$$
We deduce 
$$
 \left[{\rm op}_\eps(a^{R,\delta}_\eps) ,g_\eps(\eps D)\right]
= \left[{\rm op}_\eps(a^{R,\delta}_\eps) ,(\kappa^Rg)_\eps(\eps D)\right]+O( R^{-N+1}).
$$
The function $\xi\mapsto \kappa^R g$ is now smooth, which allows to   use standard results of symbolic calculus, what we shall do in local coordinates.

\smallskip 

We consider a system of local coordinates $\varphi(\xi)=0$ of $X$ and the $d\times p$  smooth matrix $B(\xi)$ such that 
$$\xi-\sigma_X(\xi)= B(\xi) \varphi(\xi)$$
where $\varphi (\xi)\in \C^{p\times 1}$ is a column.
 We associate with $\varphi$ the diffeomorphism
$$ \Phi: (\,^t\varphi(\xi),\xi'')\mapsto \xi$$
and, according to Lemma~4.3 in~\cite{CFM2}, there exists an isometry ${\mathcal U}_\eps$ of $L^2(\R^d)$ such that 
for all $b\in{\mathcal A}^{(2)}$ and $f\in L^2(\R^d)$
$$({\rm op}_\eps(b_\eps) f,f)=
\left({\rm op}_\eps \left(b\left(
\, ^t d\Phi(\xi)^{-1}x, \Phi(\xi), B(\Phi(\xi) )\frac{\xi'}\eps 
\right)\right){\mathcal U}_\eps f,{\mathcal U}_\eps f\right)+O(\eps).$$
Note that if $\xi,\zeta,x\in\R ^d$, 
$$d\Phi(\xi)^{-1} \zeta= (d\varphi(\Phi(\xi)) \zeta,\zeta''),\;\;
\, ^td\Phi(\xi)^{-1} x= \, ^td\varphi(\Phi(\xi)) x'+(0,x''),$$
where $d\varphi(\xi)$ is the $p\times d$ matrix with lines  the gradient of each of the component of $\varphi$. 
Therefore, focusing on the commutator 
$$L^\eps= \left[
{\rm op}_\eps \left(a^{R,\delta}\left(\, ^t d\Phi(\xi)^{-1}x, \Phi(\xi), B(\Phi(\xi) )\frac{\xi'}\eps \right)\right)
,(\kappa^R g)\left(\Phi(\eps D), B(\Phi(\eps D) )D_{x'} \right)\right],$$
 we obtain in $\mathcal L(L^2(\R^d))$,
 \begin{align*}
 L^\eps & = {\rm op}_\eps \left(
 d\varphi(\Phi(\xi)) \nabla_{x} a^{R,\delta} \left(\, ^t d\Phi(\xi)^{-1}x, \Phi(\xi), B(\Phi(\xi) )\frac{\xi'}\eps \right) \cdot \,^tB(\Phi(\xi) )\nabla_\eta  (\kappa^Rg)\left(\Phi(\xi), B(\Phi(\xi) )\frac{\xi'}\eps \right)
 \right)\\
 &\qquad + O(R^{-1})+O(\eps).
 \end{align*}
We observe that if $\sigma\in X$, $B(\sigma) d\varphi(\sigma)={\rm Id} -d\sigma_X(\sigma)$ and $d\sigma_X(\sigma) \nabla_\eta g(\sigma,\eta)=0$ because $\nabla_\eta g(\sigma,\eta)\in N_\sigma X$ and $d\sigma_X(\sigma)\zeta=0$ if $\zeta\in N_\sigma X$.
We deduce 
$$ [{\rm op}_\eps(a^{R,\delta}_\eps) ,g_\eps(\eps D)]= i {\rm op}_\eps (b_\eps) +O(\eps) +O(R^{-1})+O(\delta),$$
with 
$$b(x,\xi,\eta)=  \kappa^R (\eta)\nabla_x a^R(x,\xi,\eta) \cdot \nabla_\eta  g\left(\xi,\eta\right)\in{\mathcal A}^{(2)}_{0}.$$
Using  $\kappa (1-\chi)=1-\chi$, we obtain
$$b(x,\xi,\eta)= \nabla_x a^R(x,\xi,\eta) \cdot \nabla_\eta  g\left(\xi,\eta\right).$$
\end{proof}

We shall also need properties of  two-scale symbolic calculus at finite distance, i.e. for symbols that are compactly supported in all the variables, including the variable $\eta$. Here again, the use of local coordinates  and  Lemma~4.3 in~\cite{CFM2}  are a crucial argument.

 \begin{lemma}\label{lem:gepsfinite}
Let $a\in{\mathcal C}_0^\infty(\R^{3d})$ and $g\in\mathcal H_k$ for $k\in\N$.
Let $(f^\eps)_{\eps>0}$ a bounded family in $L^2(\R^d)$ and $Md\nu$ the two-scale Wigner measure at finite distance associated with its concentration on~$X$. Then, there exists a constant $C>0$ such that for all $\eps>0$,  
$$\left({\rm op}_\eps(a) g_\eps (\eps D) f^\eps,f^\eps\right) \leq C\| f^\eps\|_{ L^2(\R^d)}.$$
Besides, up to the subsequence defining $Md\nu$, 
$$\left({\rm op}_\eps(a) g_\eps (\eps D) f^\eps,f^\eps\right) \Tend {\eps}{0} \int_{TX^*} {\rm Tr} _{L^2(N_\sigma X)}(Q^X_a(v,\sigma) Q^X_g(\sigma) M(\sigma,v) )d\nu (\sigma,v).$$
 \end{lemma}
 
 We recall that the notations of this section have been introduced in Section~\ref{sec:twomic}.
 Of course, this lemma has standard generalizations to  vector-valued families and to  time dependent families which are bounded in $L^\infty(\R, L^2(\R^d,\C^N)$.
 
 \begin{proof}
 Here again, we work in local coordinates $\varphi(\xi)=0$ of $X$ and we consider the $d\times p$  smooth matrix $B(\xi)$ such that 
$$\xi-\sigma_X(\xi)= B(\xi) \varphi(\xi)$$
where $\varphi (\xi)\in \C^{p\times 1}$ is a column.
 We associate with $\varphi$ the diffeomorphism
$$ \Phi: (\,^t\varphi(\xi),\xi'')\mapsto \xi$$
and, we consider the isometry ${\mathcal U}_\eps$ of $L^2(\R^d)$ given by Lemma~4.3 in~\cite{CFM2} such that 
for all $b\in{\mathcal A}^{(2)}$ and $f\in L^2(\R^d)$
$$({\rm op}_\eps(b_\eps) f,f)=
\left({\rm op}_\eps \left(b\left(
\, ^t d\Phi(\xi)^{-1}x, \Phi(\xi), B(\Phi(\xi) )\frac{\xi'}\eps 
\right)\right){\mathcal U}_\eps f,{\mathcal U}_\eps f\right)+O(\eps).$$
 We then concentrate on  the operator
 $${\rm op}_\eps \left(a\left(
\, ^t d\Phi(\xi)^{-1}x, \Phi(\xi), B(\Phi(\xi) )\frac{\xi'}\eps 
\right)\right) {\rm op}_\eps \left(g\left(\Phi(\xi), B(\Phi(\xi) )\frac{\xi'}\eps \right)\right) .$$
Since $a$ is compactly supported in all variables, we obtain  in $\mathcal L(L^2(\R^d))$
\begin{align*}
&{\rm op}_\eps \left(a\left(
\, ^t d\Phi(\xi)^{-1}x, \Phi(\xi), B(\Phi(\xi) )\frac{\xi'}\eps 
\right)\right)
\\
&\qquad
=
{\rm op}_\eps \left(a\left(
\, ^t d\Phi(0,\xi'')^{-1}x, \Phi(0,\xi''), B(\Phi(0,\xi'') )\frac{\xi'}\eps 
\right)\right) + O(\eps)\\
&\qquad=
{\rm op}_\eps \left(a^W\left(
\, ^t d\Phi(0,\xi'')^{-1}x, \Phi(0,\xi''), B(\Phi(0,\xi'') )D_{x'}
\right)\right) + O(\eps).
\end{align*}
One can then interpret this operator as an operator acting on $L^2(\R^d_{x''}, L^2(\R^d_{x'}))$ where, as explained in Section~4.1 of~\cite{CFM2}, for any $(\sigma,v)=((0,\xi''), (0,x''))\in TX^*$, the map 
$$(z,\zeta)\mapsto a\left(
\, ^t d\Phi(0,\xi'')^{-1}(z,x''), \Phi(0,\xi''), B(\Phi(0,\xi'') )\zeta
\right)$$
defines a function on $T^*(N_\sigma X)$, which implies that the operator 
$$a^W\left(
\, ^t d\Phi(0,\xi'')^{-1}x, \Phi(0,\xi''), B(\Phi(0,\xi'') )D_{x'}
\right)$$
acts on $L^2(N_\sigma X)$. It is the expression of the operator $Q_a(\sigma, v)$ in the local coordinates induced by the choice of equations $\varphi(\xi)=0$ of $X$. 

\smallskip 

The difficulty with $g_\eps(\eps D)$ is that the map $(\sigma,\eta) \mapsto g(\sigma,\eta)$ is not bounded in $\eta$. Therefore, we decompose $g(\sigma,\eta)$ into two parts thanks to a function $\chi\in{\mathcal C}_0^\infty(\R^d)$ such that $\chi=1$ on $B(0,1)$ and $\chi=0$ on $B(0,2)^c$ with $0\leq \chi\leq 1$. Writing like before $\chi^R(\eta)=\chi(\eta/R)$ for $R>0$, we set  
$$g=g\chi^R + g(1-\chi^R)$$
 and we first focus on ${\rm op}_\eps(a_\eps) (g(1-\chi^R))_\eps(\eps D)$. Since now $g(1-\chi_R)$ is smooth, we can use standard symbolic calculus, and we have 
$${\rm op}_\eps(a_\eps) (g(1-\chi^R))_\eps(\eps D)= {\rm op}_\eps((ag(1-\chi^R))_\eps) + O(R^{-1})$$
because $\nabla_\xi( (g(1-\chi^R))_\eps)= O(R^{-1})$. Moreover, as soon as $R$ is large enough, we have $ag(1-\chi^R)=0$. 
We conclude
$$g_\eps(\eps D)= (g\chi^R)_\eps(\eps D) +O(R^{-1})$$
 and 
$${\rm op}_\eps \left(g\left(\Phi(\xi), B(\Phi(\xi) )\frac{\xi'}\eps \right)\right)=
{\rm op}_\eps \left((g\chi^R)(\Phi(0,\xi''), B(\Phi(0,\xi'') ) D_{x'})\right) +O_R(\eps) +O(R^{-1})$$
in $\mathcal L(L^2(\R^d))$.
Note that in the following we will let first $\eps$ go to $0$, and then $R$ to $+\infty$, so that $O_R(\eps)$ is negligible. Besides, when $R$ goes to $+\infty$,  we are left with the operator $g(\Phi(0,\xi''), B(\Phi(0,\xi'') D_{x'})$ (with strong convergence), which is the expression in local coordinates of  the operator $Q_g(\sigma)$. 

\smallskip 

At this stage of the proof, we are left with the quantity 
$$\left( {\rm op}_\eps \left(a^W\left(
\, ^t d\Phi(0,\xi'')^{-1}x, \Phi(0,\xi''), B(\Phi(0,\xi'') )D_{x'}
\right)\right) {\rm op}_\eps (g(\Phi(0,\xi''), B(\Phi(0,\xi'') ) D_{x'})) 
{\mathcal U}_\eps f,{\mathcal U}_\eps f\right).$$
 It turns out that the pair $M d\nu$ has been defined in~\cite{CFM2} (see Proposition~4.2) as a semi-classical measure of the family $({\mathcal U}_\eps f)$, which is a bounded family in $L^2(\R^{d-p}_{x''}, L^2(\R^p_{x'}))$. Therefore, in coordinates
\begin{align*}
&\left( {\rm op}_\eps \left(a^W\left(
\, ^t d\Phi(0,\xi'')^{-1}x, \Phi(0,\xi''), B(\Phi(0,\xi'') )D_{x'}
\right)\right) {\rm op}_\eps (g(\Phi(0,\xi''), B(\Phi(0,\xi'') ) D_{x'})) 
{\mathcal U}_\eps f,{\mathcal U}_\eps f\right)
\\
&\Tend{\eps}{0} 
\int_{ \R^{2(d-p)}} {\rm Tr}_{L^2(\R^p)} \left(  {\rm op}_\eps \left(a^W\left(
\, ^t d\Phi(0,\xi'')^{-1}x, \Phi(0,\xi''), B(\Phi(0,\xi'') )D_{x'}
\right)\right) M(\xi'',x'') \right)d\nu(\xi'',x'')\\
&\qquad\qquad= \int_{TX^*} {\rm Tr} _{L^2(N_\sigma X)}(Q^X_a(v,\sigma) Q^X_g(\sigma)  M(\sigma,v) )d\nu (\sigma,v).
\end{align*}
 \end{proof}

 \section{Well-prepared data}\label{app:wpdata}

 We prove here properties of well-prepared initial~\eqref{def:wpdata}.
  We shall use the Fourier coefficients of the Bloch waves
 $\varphi_n(y,\xi)$ for $n\in\N^*$. Recall that they  satisfy the Bloch periodicity condition:
\begin{equation}\label{eq:bp}
\varphi_n(y,\xi+j) = {\rm e}^{-i j\cdot y} \varphi_n(y,\xi),\;\;\forall j\in 2\pi\Z^d.
\end{equation}
Integrating \eqref{eq:bp} with respect to $y$ on $\T^d$ gives an expression for the Fourier coefficients of $\varphi_n(\cdot,\xi)$: for all $y\in\T^d$ and $\xi\in \R^d$
\begin{equation}\label{eq:fsphi}
\varphi_n(y,\xi)=\sum_{j\in 2\pi \Z^d} c_n(\xi+ j){\rm e}^{i j\cdot y},\;\; c_n(\xi)=\int_{\T^d} \varphi_n(y,\xi) dy.
\end{equation}
Besides, for $s>0$, there exists $C_{n,s}>0$ such that 
\begin{equation}\label{peetre}
\sum_{j\in 2\pi\Z^d} \langle j\rangle ^{2s} |c_n(\xi+j)|^2 \leq C_{n,s} \langle \xi\rangle ^{2s}
\end{equation}
It turns out that these class of data are closely related to those studied in~\cite{AP05,AP06}. 

 \begin{lemma}\label{lem:APdata}
 Let $( \psi^\eps_{0,n})_{\eps>0}$  as in~\eqref{def:wpdata} with 
 $u^\eps_n= {\rm e}^{\frac i\eps \xi_n\cdot x}v^\eps_n(x)$, $v_n^\eps$ uniformly bounded in $H^s(\R^d)$, $s>d/2$ (or $s>1$ when $d=1$).
 Assume that $\xi\mapsto \varphi_n(\cdot, \xi)$ is Lipschitz.
 Then, we have in $L^2(\R^d)$
 \[
  \psi^\eps_{0,n}= {\rm e}^{\frac i \eps x\cdot \xi_n}\varphi_n\left(\frac x\eps,\xi_n\right) v^\eps _n+O(\eps).
 \]
 Besides, if $v^\eps_n \rightarrow v_n$ in $L^2$, then $\| \psi^\eps_{0,n}\|_{L^2}= \| v_n\|_{L^2}+o(1)$ as $\eps$ goes to $0$.
 \end{lemma}

 \begin{proof}
 We set $\widetilde \psi^\eps_{0,n}= {\rm e}^{\frac i \eps x\cdot \xi_n}\varphi_n\left(\frac x\eps,\xi_n\right) v^\eps _n$.
We write for $x\in\R^d$
\[ \psi^\eps_{0,n}(x)=\sum_{j\in 2\pi\Z^d} {\rm e}^{\frac i\eps j\cdot x}c_n(j+\eps D_x) u^\eps_n(x).
\]
Let $\theta\in L^2(\R^d)$, we write 
$(\theta,\psi^\eps_{0,n})_{L^2(\R^d)} = \sum_{j\in 2\pi\Z^d} \alpha^\eps_j$ with 
\[
\alpha^\eps_j= \left(\theta ,{\rm e}^{\frac i\eps j\cdot x}c_n(j+\eps D_x) u^\eps_n \right)_{L^2}
=\left(\theta ,{\rm e}^{\frac i\eps (j+\xi_n)\cdot x}c_n(j+\xi_n+\eps D_x) v^\eps_n \right)_{L^2}.
\]
By~\eqref{peetre}, there exists $C>0$ such that 
\[ |\alpha^\eps_j| \leq C \langle j\rangle ^{-s} \| \langle \eps D_x\rangle^s u^\eps_n\|_{L^2}\| \theta\|_{L^2}.
\]
Therefore, the sum enters into the frame of Lebesgue dominated convergence. The same holds for the series 
\[
(\theta,\widetilde \psi^\eps_{0,n})_{L^2(\R^d)} = \sum_{j\in 2\pi\Z^d} \widetilde \alpha^\eps_j,\;\;
\widetilde \alpha^\eps_j=  \left(\theta ,{\rm e}^{\frac i\eps (j+\xi_n)\cdot x}c_n(j+\xi_n) v^\eps_n \right)_{L^2}.
\]
Besides 
\[ \alpha^\eps_j- \widetilde \alpha ^\eps_j= 
\left( \theta , {\rm e}^{\frac i\eps (j+\xi_n)\cdot x}(c_n(j+\xi_n+\eps D_x)-c_n(j+\xi_n)) v^\eps_n \right)_{L^2}.
\]
The conclusion then comes from the observation that $\xi\mapsto c_n(\xi)$ is uniformly Lipschitz and thus, one has
\[ 
\left| \alpha^\eps_j- \widetilde \alpha ^\eps_j\right| \leq \| \theta\|_{L^2} \| (c_n(j+\xi_n+\eps D_x)-c_n(j+\xi_n)) v^\eps_n\|_{L^2} \leq C\eps \| \theta\|_{L^2} \| D_x v^\eps_n\|_{L^2}
\]
which gives the result by 
 the boundedness of the family~$(v^\eps_n)_{\eps>0}$ in $H^s(\R^d)$ with~$s\geq 1$.
\smallskip

Let us now compute the norm of $\psi^\eps_{0,n}$. 
One has $ \|\psi^\eps_{0,n}\|_{L^2}^2=\sum _{j,j'\in 2\pi\Z^d} \beta^\eps_{j,j'}$   with 
\begin{align*}
\beta^\eps_{j,j'}   &=  \left({\rm e}^{\frac i\eps \cdot x (j-j')}
     c_n(j+\xi_n+\eps D_x) v^\eps_n(x) ,c_n(j'+\xi_n+\eps D_x ) v^\eps_n(x)\right)_{L^2}.
\end{align*}
Besides, by~\eqref{peetre}, there exists $C>0$ such that 
\[
|\beta^\eps_{j,j'}|\leq C \langle j\rangle^{-s}\langle j'\rangle ^{-s} \|\langle \eps D_x\rangle^s v^\eps_n\|^2_{L^2}.
\]
Therefore, the sum enters into the frame of Lebesgue dominated convergence. If $|j-j'|\geq 1$ with $j_\ell-j'_\ell\not=0$, $0\leq \ell\leq d$, an integration by parts give 
\[ \beta^\eps_{j,j'}   =i \eps (j_\ell-j'_\ell)^{-1}   \int_{\R^d} {\rm e}^{\frac i\eps \cdot x (j-j')} \partial_{x_\ell}
   \left(  c_n(j+\xi_n+\eps D_x) v^\eps_n(x) \overline{c_n(j'+\xi_n+\eps D_x ) v^\eps_n(x)}\right)dx,
     \]
     whence $\beta^\eps_{j,j'}\rightarrow 0$ as $\eps \rightarrow 0$ since $(v^\eps_n)_{\eps>0}$ is uniformly bounded in $H^1(\R^d)$. Moreover 
     \[\sum_{j\in 2\pi\Z^d}\beta^\eps_{j,j}= \sum_{j\in 2\pi\Z^d}  (2\pi)^{-d}\int_{\R^d} |c_n(j+\xi_n+\eps \xi )|^2 |\widehat v^\eps_n (\xi) |^2 d\xi.
     \]
     Using $\sum_{j\in 2\pi\Z^d} |c_n(\xi )|^2=\| \varphi_n(\cdot,\xi)\|_{L^2(\T^d)}^2=1$, we deduce 
     \[\sum_{j\in 2\pi\Z^d}\beta^\eps_{j,j}= (2\pi)^{-d}\int_{\R^d}|v^\eps_n (\xi) |^2 d\xi=\| v^\eps_n\|^2_{L^2}\Tend{\eps}{0} \| v_n\|^2
\]
and the conclusion follows.
 \end{proof}
 
 Let us now examine weak limits of such families.
 
 \begin{lemma}\label{lem:data_toto}
 Let $( \psi^\eps_{0,n})_{\eps>0}$  as in~\eqref{def:wpdata} with 
 $u^\eps_n= {\rm e}^{\frac i\eps \xi_n\cdot x}v_n(x)$, $(v_n^\eps)_{\eps>0}$  bounded in $H^s(\R^d)$ with $s>d/2$ (or $s>1$ when $d=1$). Assume one has  as $\eps$ goes to $0$
 \[
 v_n^\eps \rightharpoonup v_n  \;\; \mbox{in}\;\; L^2.
 \]
 Then, the weak limits of $\left({\rm e}^{-\frac i\eps x\cdot \xi}\psi^\eps_{0,n}\right)_{\eps>0}$
 are non zero if and only if $\xi \in \xi_{n}+2 \pi \Z^d$. Besides, for any~$j\in 2\pi\Z^d$, one has as $\eps$ goes to $0$
 \[
 {\rm e}^{-\frac i\eps x\cdot (\xi_n+j)}\psi^\eps_{0,n}
 \rightharpoonup 
 c_n(\xi_n+j) v_{n} \;\; \mbox{in}\;\; L^2.
 \]
  \end{lemma}
  
 \begin{proof}
One writes 
\begin{align*}
    {\rm e}^{-\frac i\eps x\cdot \xi}\psi^\eps_{0,n}& = 
   {\rm e}^{-\frac i\eps x\cdot \xi}\sum_{j\in 2\pi\Z^d}  {\rm e}^{-\frac i\eps x\cdot j} c_n(\eps D_x+j)   {\rm e}^{\frac i\eps x\cdot \xi_n} v^\eps_n
   = \sum_{j\in 2\pi\Z^d}  {\rm e}^{-\frac i\eps x\cdot (j+\xi-\xi_n)} c_n(\xi_n+j +\eps D_x)  v^\eps_n
\end{align*}
and the result follows. 
 \end{proof}
 
 As a corollary, we obtain a description of the concentration of these families on subsets of interest. 
 
 \begin{corollary}\label{cor:titi}
 Let $X$ be a $2\pi\Z^d$ periodic subset~$X$ of $\R^d$ consisting of isolated points and containing~$\xi_n$.
 \begin{enumerate}
     \item  Let $M^Xd\nu^X$ be the two-microlocal measures at finite distance associated with the concentration of $(\psi^\eps_{0,n})_{\eps>0}$  on $X$. 
 Then,
 $$\nu^X=\| v_n\|^2_{L^2}\sum_{j\in 2\pi\Z^d}|c_n(\xi_n+j)|^2  \delta(\xi-\xi_n-j)$$ 
 and, for $j\in 2\pi\Z^d$, the operator~$M^X(\xi_{n} +j)$ is the projector of $L^2(\R^d)$ on $\C v_{n}$.\\
 \item Consider the family of $L^2(\R^d,\C^2)$, 
 $\Psi_{n,n'}^\eps=\, ^t\left( \psi^\eps_{0,n},\psi^\eps_{0,n'}\right)$ with $n\not=n'$ and $\xi_{n'}=\xi_n$. 
Let $M^Xd\nu^X$ be the two-microlocal measures at finite distance associated with the concentration of $\Psi_{n,n'} ^\eps$ 
 on~$X$. Then
 $$\nu^X=\left(\| v_n\|^2_{L^2} + \| v_{n+1}\|^2_{L^2}\right)
 \sum_{j\in 2\pi\Z^d}|c_n(\xi_n+j)|^2  \delta(\xi-\xi_n-j)$$ 
 and, for $j\in 2\pi\Z^d$, the operator~$M^X(\xi_{n} +j)$ is the projector of $L^2(\R^d,\C^2)$ on $\C \,^t(v_{n},v_{n+1})$.
  \end{enumerate}
 \end{corollary}
 
 \begin{proof}
 Since $\xi_n$ is an isolated point of $X$,  by Remark~\ref{rem:M0}, $M^X$ is the orthogonal projector on  a weak limit in $L^2$ of 
 \[
 x\mapsto {\rm e}^{-\frac i\eps x\cdot \xi}\varphi_{n}\left(\frac x\eps, \eps D_x\right) u^\eps_{n}.
 \]
 By Lemma~\ref{lem:data_toto}, any of this weak limit is $0$ if $\xi\notin\xi_{n}+2\pi\Z^d$ and if $\xi=\xi_{n}+j$ with $j\in 2\pi\Z^d$, then  the limit is  $c_{n_0}(\xi_{n_0}+j)v_{n_0}$. Whence the Part~1 of the result. The proof of Part~2 follows the same lines.
 \end{proof}

%%%%%%%%%%%%%%%%%%%%%%%%%%%%%%%%%%%%%%%%%%%%%%%%%%%%%%%%%

\def\cprime{$'$}

%\bibliographystyle{abbrv}
%\bibliography{../EffMassBib,../../../bibliosmc,../../../bibliodisp}

\end{document}